\documentclass{amsart}
\usepackage{amssymb,amsmath, amsthm,latexsym}
\usepackage{graphics}
\usepackage{psfrag}
\usepackage{amscd}
\usepackage{graphicx}
\newcommand{\cal}[1]{\mathcal{#1}}
\theoremstyle{plain}
\newtheorem{theo}{Theorem}

\newtheorem{lemma}{Lemma}[section]
\newtheorem{theorem}[lemma]{Theorem}
\newtheorem{proposition}[lemma]{Proposition}
\newtheorem{corollary}[lemma]{Corollary}
\theoremstyle{definition}
\newtheorem*{definition}{Definition}
\newtheorem{defi}[lemma]{Definition}
\newtheorem{remark}[lemma]{Remark}
\parskip=\bigskipamount

\let\egthree=\phi
\let\phi=\varphi
\let\varphi=\egthree

\begin{document}
\title{Teichm\"uller flow and Weil-Petersson flow}
\author{Ursula Hamenst\"adt}
\thanks
{
Research
partially supported by ERC Grant Nb. 10160104}
\date{May 4, 2015}

\begin{abstract}
For an oriented  
surface $S$ of genus $g\geq 0$ with 
$m\geq 0$ punctures and $3g-3+m\geq 2$,  let ${\cal Q}(S)$ and
${\cal Q}_{WP}(S)$ be the moduli space of area one
quadratic differentials and of quadratic differentials of 
unit norm for the Weil-Petersson metric, respectively.
We show that there is a 
Borel subset ${\cal E}$ of ${\cal Q}(S)$ which is invariant
under the Teichm\"uller flow $\Phi^t_{\cal T}$
and of full measure for every invariant Borel probability measure, 
and there is a measurable map $\Lambda:{\cal E}\to
{\cal Q}_{WP}(S)$ which conjugates $\Phi^t_{\cal T}\vert {\cal E}$ into
the Weil-Petersson flow
$\Phi^t_{WP}$. This conjugacy induces a
continuous injection 
of the space of all $\Phi^t_{\cal T}$-invariant 
Borel probability measures on ${\cal Q}(S)$ 
into the space of all $\Phi^t_{WP}$-invariant  
Borel probability measures on ${\cal Q}_{WP}(S)$.
The map $\Theta$ is not surjective, but its image contains 
the Lebesgue Liouville measure. A measure not in the image
corresponds to a locally finite infinite invariant Borel measure on 
${\cal Q}(S)$. 
\end{abstract}

\maketitle

\tableofcontents

\section{Introduction}

An oriented surface $S$ of 
finite type is a
closed surface of genus $g\geq 0$ from which $m\geq 0$
points, so-called \emph{punctures},
have been deleted. We assume that $3g-3+m\geq 2$,
i.e. that $S$ is not a sphere with at most four
punctures or a torus with at most one puncture.
We then call the surface $S$ \emph{nonexceptional}.

Since the Euler characteristic of $S$ is negative,
the \emph{Teichm\"uller space} ${\cal T}(S)$
of $S$ is the quotient of the space of all complete 
finite volume hyperbolic
metrics on $S$ under the action of the
group of diffeomorphisms of $S$ which are isotopic
to the identity. The \emph{mapping class group}
${\rm Mod}(S)$ of all isotopy classes of orientation
preserving diffeomorphisms of $S$ 
acts properly discontinuously
on ${\cal T}(S)$

There are two ${\rm Mod}(S)$-invariant metrics on Teichm\"uller space
${\cal T}(S)$
which have been studied extensively in the past: The so-called 
\emph{Teichm\"uller metric} and 
the \emph{Weil-Petersson metric}.

The \emph{Teichm\"uller metric} is a complete 
${\rm Mod}(S)$-invariant Finsler
metric on ${\cal T}(S)$. It is customary to 
view this metric as a metric on the 
cotangent bundle of Teichm\"uller space. This cotangent 
bundle is
the bundle of holomorphic quadratic differentials over
${\cal T}(S)$.
The unit sphere bundle of the metric is the 
bundle $\tilde {\cal Q}(S)$ of quadratic differentials of area one.
Although the large scale geometry of the
Teichm\"uller metric does not
resemble a metric of non-positive curvature,
 any two points of ${\cal T}(S)$ can be
connected by a unique geodesic.
We call such a geodesic 
a \emph{Teichm\"uller geodesic}. 

Teichm\"uller distances can be effectively estimated
(see \cite{R07b}), and there are many recent results on 
the asymptotic behavior of Teichm\"uller geodesics.
We refer to \cite{LM10,LR11} for more information.

The Teichm\"uller metric defines a \emph{geodesic flow} on
$\tilde {\cal Q}(S)$ which is equivariant with respect to the
action of the mapping class group and hence projects to 
a flow $\Phi^t_{\cal T}$ on the \emph{moduli
space} 
\[{\cal Q}(S)=\tilde {\cal Q}(S)/{\rm Mod}(S)\] 
\emph{of area one quadratic differentials}.

Although ${\cal Q}(S)$ is not compact, it admits many
$\Phi^t_{\cal T}$-invariant Borel 
probability measures.
For example, there are countably many periodic orbits
for the Teichm\"uller flow on ${\cal Q}(S)$, and each such orbit
supports a natural invariant Borel probability measure.
These periodic orbits are in bijection with conjugacy
classes of pseudo-Anosov mapping classes, and they 
can be counted according to their
lengths \cite{EM10,H13}. Borel probability measures
supported on periodic orbits are dense in the 
space ${\cal M}_{\cal T}({\cal Q}(S))$ 
of all $\Phi^t_{\cal T}$-invariant Borel probability
measures on ${\cal Q}(S)$ equipped with the weak$^*$-topology.
However, for every compact set $K\subset {\cal Q}(S)$
there are periodic orbits which do not intersect $K$
\cite{H05} and hence the space 
${\cal M}_{\cal T}({\cal Q}(S))$ 
is non-compact: There are sequences of such measures supported
on periodic orbits which
converge weakly to the trivial measure of vanishing total mass.

The \emph{Weil-Petersson metric} is a
${\rm Mod}(S)$-invariant K\"ahler metric on ${\cal T}(S)$ 
of negative sectional curvature which induces
an incomplete distance $d_{WP}$. 
The completion 
$\overline{{\cal T}(S)}$ of ${\cal T}(S)$
with respect to $d_{WP}$ is a ${\rm Cat}(0)$-space, however
it is not locally compact. As a consequence, 
any two points in $\overline{{\cal T}(S)}$
can be connected
by a unique geodesic. Such a geodesic will be called
a \emph{Weil-Petersson geodesic} in the sequel. Finite length
Weil-Petersson geodesic arcs
with both endpoints in ${\cal T}(S)$ are 
entirely contained in ${\cal T}(S)$ 
(Corollary 5.4 of \cite{W87}).

The Weil-Petersson metric can be viewed as a 
${\rm Mod}(S)$-invariant metric on the
cotangent bundle of Teichm\"uller space. The norm of 
a quadratic differential $q$ is the $L^2$-norm of $q$ with 
respect to the underlying hyperbolic metric on 
the base point of $q$. The unit cotangent bundle 
$\tilde {\cal Q}_{WP}(S)$ for the metric projects to 
an orbifold bundle 
${\cal Q}_{WP}(S)$ over the moduli space of curves.
The 
geodesic flow $\Phi^t_{WP}$ 
for the metric 
acts on ${\cal Q}_{WP}(S)$, however
this flow is incomplete. Periodic orbits for $\Phi^t_{WP}$ are in bijection
with conjugacy classes of pseudo-Anosov elements
\cite{DW03} and hence they are in bijection with 
 periodic orbits for the Teichm\"uller flow. Each
of these orbits supports an invariant Borel 
probability measure, and the set of all these measures
is dense in the space 
${\cal M}_{\rm WP}({\cal Q}_{\rm WP}(S))$ 
of $\Phi^t_{\rm WP}$-invariant Borel probability measures
on ${\cal Q}_{\rm WP}(S)$ equipped with 
the weak$^*$-topology \cite{H10b}. 
In other words, there is a natural bijection between
dense subsets of ${\cal M}_{T}({\cal Q}(S))$ and 
${\cal M}_{\rm WP}({\cal Q}_{\rm WP}(S))$. 
The space ${\cal M}_{WP}({\cal Q}_{WP}(S))$ contains
the \emph{Lebesgue Liouville measure} for the
Weil-Petersson metric. This measure was shown to be
ergodic in \cite{BMW12}.

For a measure $\mu\in {\cal M}_{\cal T}({\cal Q}(S))$
(or $\nu\in {\cal M}_{\rm WP}({\cal Q}(S))$) 
let $h(\mu)$ be the \emph{entropy} of $\mu$ (or of $\nu$).
The entropy of  
a measure $\mu\in {\cal M}_{\cal T}({\cal Q}(S))$ is at most
$6g-6+2m$ \cite{H11}. In contrast,  
Brock, Masur and Minsky constructed invariant
Borel probability measures
for the Weil-Petersson geodesic flow whose entropy is arbitrarily large
\cite{BMM11}. The paper \cite{PWW10}
also contains some information on the dynamics of 
$\Phi^t_{WP}$.

\begin{definition}\label{conjugacy}
A \emph{measurable (or continuous) 
conjugacy} of a continuous flow $\Psi^t$ on a 
topological space $X$
into a continuous flow $\Xi^t$ on a
topological space $Y$ is an injective 
measurable (or continuous) map
$\Lambda:X\to Y$ such that there is a 
measurable (or continuous) function $\rho:X\times \mathbb{R}\to \mathbb{R}$  
with the following properties.
\begin{enumerate}
\item $\rho(x,0)=0$ for all $x\in X$.
\item For each fixed $x\in X$ the function $\rho(x,\cdot):s\to \rho(x,s)$
is an increasing homeomorphism.
\item $\Lambda(\Psi^tx)=\Xi^{\rho(x,t)}\Lambda(x)\text{ for all }x\in X,
t\in \mathbb{R}.$
\end{enumerate}
\end{definition}

We show

\begin{theo}\label{conjug}
There is a $\Phi^t_{\cal T}$-invariant Borel subset
${\cal E}\subset {\cal Q}(S)$ 
with the following properties.
\begin{enumerate}
\item $\mu({\cal E})=1$ for every $\Phi^t_{\cal T}$-invariant 
Borel probability measure $\mu$ on ${\cal Q}(S)$.
\item There is a measurable
conjugacy $\Lambda: {\cal E}\to ({\cal Q}_{WP}(S),\Phi_{WP}^t)$
whose restriction to any compact invariant set is 
continuous. 
\item $\Lambda$ induces a 
continuous injection
\[\Theta:{\cal M}_{\cal T}({\cal Q}(S))\to 
{\cal M}_{\rm WP}({\cal Q}_{\rm WP}(S)).\]
\item $\Theta$ is not surjective, but its image contains
the Lebesgue Liouville measure.
\item \[\infty>h(\Theta(\mu))\geq h(\mu)/\sqrt{2\pi(2g-2+n)
}\text{ for all }
\mu\in {\cal M}_{\cal T}({\cal Q}(S)).\]
\end{enumerate}
\end{theo}

To show that the map $\Theta$ is not surjective we construct
an explicit $\Phi^t_{WP}$-invariant Borel probablity
measure on ${\cal Q}_{WP}(S)$ 
which is not contained in its image.
However, this measure is not ergodic. 
 
We also obtain information on $\Phi^t_{WP}$-invariant
Borel probability measures which are not contained in the
image of $\Theta$. For simplicity we restrict our attention to 
ergodic measures. 

\begin{theo}\label{conjug2}
Let $\nu$ be any  $\Phi^t_{WP}$-invariant ergodic Borel probability
measure on ${\cal Q}_{WP}(S)$. Then there is an invariant
Borel set $A\subset {\cal Q}_{WP}(S)$ with 
$\nu(A)=1$, and there is a measurable conjugacy
\[\Xi: A\to ({\cal Q}_{\cal T}(S), \Phi^t_{\cal T}).\] The measure $\Xi_*\nu$
determines a locally finite 
$\Phi^t_{\cal T}$-invariant Borel measure on ${\cal Q}(S)$ which is
finite if and only if $\nu\in \Theta({\cal M}_{\cal T}({\cal Q}(S))$).
\end{theo}




The strategy for the proof of Theorem \ref{conjug} consists
in a geometric comparison between biinfinite Teichm\"uller 
geodesics and biinfinite Weil-Petersson geodesics
provided that these geodesics satisfy suitable recurrence
properties under the action of the mapping class group.
Particular such geodesics are geodesics which are
contained in the thick part of Teichm\"uller space.  

Denote for $\epsilon >0$ by ${\cal T}(S)_\epsilon$
the ${\rm Mod}(S)$-invariant subset of all hyperbolic metrics
whose \emph{systole} is at least $\epsilon$.
The mapping class group acts cocompactly on 
${\cal T}(S)_\epsilon$ and hence 
by invariance,
the restrictions to ${\cal T}(S)_\epsilon$ 
of the Teichm\"uller metric and the
Weil-Petersson metric
are locally uniformly bilipschitz equivalent. 
It follows easily from the 
work of Masur and Minsky (see \cite{R07b}) and 
Brock \cite{B03} that on the large scale, the restrictions 
to ${\cal T}(S)_\epsilon$ of the
distances $d_{\cal T}$ and $d_{WP}$ induced
by the Teichm\"uller metric and the Weil-Petersson metric,
respectively, 
are not bilipschitz equivalent. 
For example, there is a sequence of 
points $\{x_i\}\subset {\cal T}(S)_\epsilon$ 
and a number $c>0$ such that 
$d_{\cal T}(x_0,x_i)\to \infty$ and $d_{WP}(x_0,x_i)\leq c$.
The Teichm\"uller geodesics connecting $x_0$ to $x_i$
enter arbitrarily deeply into the thin part of
Teichm\"uller space.

In \cite{BMM11}, 
Brock, Masur and Minsky showed that for every $\epsilon >0$, 
biinfinite Teichm\"uller geodesics which are entirely contained in 
${\cal T}(S)_\epsilon$ are fellow-traveled by biinfinite
Weil-Petersson geodesics, and binifinite Weil-Petersson geodesics
entirely contained in ${\cal T}(S)_\epsilon$ are fellow-traveled 
by biinfinite Weil-Petersson geodesics.

The arguments used for the proof of Theorem \ref{conjug} yield
a similar result. For a formulation, 
we use the distance $d_{\cal T}$ on ${\cal T}(S)$ 
induced by the Teichm\"uller metric to define
the \emph{Hausdorff-distance} $d_H(A,B)\in [0,\infty]$ between two
subsets $A,B$ of ${\cal T}(S)$ as the infimum of all numbers $r>0$
such that $A$ is contained in the $r$-neighborhood of $B$ 
and $B$ is
contained in the $r$-neighborhood of $A$.

\begin{theo}\label{fellowtravel}
For every $\epsilon>0$ there is a number
$R=R(\epsilon) >0$ with the following property.
\begin{enumerate}
\item
Let $J\subset \mathbb{R}$ be a closed connected set and let 
$\gamma:J\to {\cal T}(S)_\epsilon$
be a Teichm\"uller geodesic. Then 
there is a closed connected set $J^\prime\subset \mathbb{R}$ 
and there is a Weil-Petersson geodesic
$\xi:J^\prime\to {\cal T}(S)$ with 
$d_H(\gamma(J),\xi(J^\prime))\leq R$.
\item Let $J\subset \mathbb{R}$ be
a closed connected set and let
$\xi:J\to {\cal T}(S)_\epsilon$ be a 
Weil-Petersson geodesic. Then there is a
closed connected set $J^\prime\subset \mathbb{R}$ and there is a
Teichm\"uller geodesic $\gamma:J^\prime\to 
{\cal T}(S)$ with 
$d_H(\xi(J),\gamma(J^\prime))\leq R$.
\end{enumerate}
\end{theo}

The organization of this work is as follows.
In Section \ref{weilthick} we establish the second part of 
Theorem \ref{fellowtravel} from standard
properties of Teichm\"uller geodesics and 
some results of Brock, Masur and Minsky \cite{BMM10}. 

Section 3 contains some geometric results on the Weil-Petersson metric.
We begin with the fairly easy observation that
given $\epsilon >0$ and a Teichm\"uller geodesic
$\gamma:\mathbb{R}\to {\cal T}(S)_\epsilon$, 
there are infinite Weil-Petersson geodesic
rays $\gamma_+,\gamma_-:[0,\infty)\to {\cal T}(S)$ which are
uniform limits on compact sets of Weil-Petersson geodesic segments
connecting $\gamma(0)$ to points $\gamma(T_i),\gamma(-R_i)$ for
suitably choosen sequences $T_i\to \infty, R_i\to \infty$.
Since the completion of Teichm\"uller space with respect to the
Weil-Petersson metric is a ${\rm CAT}(0)$ geodesic metric space, 
Theorem \ref{fellowtravel} now predicts the 
existence of  
a biinfinite Weil-Petersson geodesic
which is forward \emph{asymptotic} to $\gamma_+$ and backward 
asymptotic to $\gamma_-$. However, the curvature of the
Weil-Petersson metric is not bounded from above by a negative
constant, and the existence of such a geodesic is not immediate.

With an 
argument based on ruled surfaces and angle comparison we
derive a sufficient condition for the existence of a 
biinfinite Weil-Petersson geodesic which is forward and backward asymptotic
to given Weil-Petersson geodesic rays. 
Roughly speaking, this condition is satisfied if the 
geodesic rays spend a sufficient (but finite) amount of time in the thick
part of Teichm\"uller space.

In Section 4 we 
find a sufficient condition for a
Weil-Petersson geodesic segment of uniformly bounded length to 
pass through the thick part of Teichm\"uller space. 
This condition is a consequence of a quantitative version of the following
result of Wolpert \cite{W03}: 
If $\zeta:[0,R]\to {\cal T}(S)$ is any Weil-Petersson geodesic
of uniformly bounded length and if there is a simple closed
curve $\alpha$ which is long at both endpoints of $\zeta$
and becomes very short along $\zeta$, then $\zeta$ twists about 
$\alpha$ or about a curve $\beta$ which is disjoint from $\alpha$.

Section 5 contains the main technical results of this work. 
We use hyperbolicity of the \emph{curve graph} and 
some of the combinatorial tools introduced 
by Masur and Minsky \cite{MM00} to establish a sufficient
condition for a Weil-Petersson geodesic segment of arbitrary
length to spend a fixed
proportion of time in the thick part of Teichm\"uller space.
This result 
is used in Section 6 to construct the Borel set ${\cal E}$
and the conjugacy $\Lambda:{\cal E}\to {\cal Q}_{WP}(S)$ 
from Theorem \ref{conjug}. The first part of 
Theorem \ref{fellowtravel} follows easily.

In Section 7 we show that the conjugacy $\Lambda$ induces
an injection $\Theta$ of the space of invariant 
Borel probability measures on ${\cal Q}(S)$ into the 
space of invariant Borel probability measures on 
${\cal Q}_{WP}(S)$. 
In Section 8 we prove Theorem \ref{conjug2} and use this to 
characterize the image of the map $\Theta$.  
We also construct an example
of a measure which is not contained in the image of the map $\Theta$.
Finally in Section 9 we find that the image of $\Theta$ contains the 
Lebesgue Liouville measure.

{\bf Notation:} Throughout the paper, we write
$f\asymp g$ for two nonvanishing functions $f,g$ if 
$f/g$ and $g/f$ are uniformly bounded.

\bigskip

{\bf Acknowledgement:} This work began in fall 2007 during a visit
of the MSRI in Berkeley and was inspired by discussions with 
Jeff Brock, Howard Masur and Yair Minsky. 
Most of the results in Sections 2-7 were written 
in spring 2008 while I visited the
IHES in Bures-sur-Yvette. I thank the institute for its hospitality and
the good working conditions it provided. 
The paper was completed in spring 2015 during another visit
of the MSRI in Berkeley. During this visit 
I benefitted from discussions with 
Howard Masur and Scott Wolpert.

\section{Weil-Petersson geodesics 
in the thick part of ${\cal T}(S)$}\label{weilthick}

Let $S$ be an oriented surface of genus $g\geq 0$ with 
$m\geq 0$ punctures and $3g-3+m\geq 2$ and let ${\cal T}(S)$
be the Teichm\"uller space for $S$. 
In this section we 
investigate geodesics for 
the Weil-Petersson metric on 
${\cal T}(S)$ which remain
entirely in a fixed thick part of Teichm\"uller space. 
Our goal is to prove the second part of Theorem \ref{fellowtravel}
which relates such a geodesic to a Teichm\"uller geodesic.

We begin with summarizing those properties of the Weil-Petersson 
metric which are needed in the sequel. More information 
and references are contained in the survey paper
\cite{W03}.

A geodesic for the Weil-Petersson
metric is also called 
a WP-geodesic.
Such a geodesic will always be parametrized by arc length. 
A \emph{WP-ray} in ${\cal T}(S)$ is a WP-geodesic
$\gamma:[0,T)\to {\cal T}(S)$ for some
$T\in (0,\infty]$ which can not be extended, i.e. which
leaves every compact set.

A \emph{geodesic lamination} for a complete
hyperbolic structure on $S$ of finite volume is
a \emph{compact} subset of $S$ which is foliated into simple
geodesics.
A geodesic lamination $\lambda$ on $S$ is called \emph{minimal}
if each of its half-leaves is dense in $\lambda$. Thus a simple
closed geodesic is a minimal geodesic lamination. A minimal
geodesic lamination with more than one leaf has uncountably
many leaves and is called \emph{minimal arational}.
A geodesic lamination $\lambda$ is said to \emph{fill $S$} if
every simple closed geodesic on $S$ intersects $\lambda$
transversely. This is equivalent to stating that 
the complementary components of $\lambda$ are all 
topological discs
or once punctured topological discs.

A \emph{measured geodesic lamination} is a geodesic lamination
$\lambda$ together with a translation invariant transverse
measure. Such a measure assigns a positive weight to each compact
arc in $S$ which intersects $\lambda$ nontrivially and
transversely and whose
endpoints are contained in complementary regions of
$\lambda$.
The geodesic lamination $\lambda$ is called the
\emph{support} of the measured geodesic lamination; it consists of
a disjoint union of minimal components. Every
minimal geodesic lamination is the support of a measured
geodesic lamination.
The space ${\cal M\cal L} $ of measured geodesic laminations
on $S$ can be equipped with the weak$^*$-topology.
Its projectivization ${\cal P\cal M\cal L}$ is called
the space of \emph{projective measured geodesic laminations},
and it is homeomorphic to the sphere $S^{6g-7+2m}$.

For every marked hyperbolic
metric $h\in {\cal T}(S)$,
every essential free homotopy
class $\alpha$ on $S$ can be represented by a closed
$h$-geodesic which is unique up to parametrization.
This geodesic is simple if the free homotopy
class admits a simple representative. The
$h$-length $\ell_\alpha(h)$
of the class $\alpha$ 
is the length of its geodesic representative.
Equivalently, $\ell_\alpha(h)$ equals
the minimum of the $h$-lengths of all closed
curves representing the class $\alpha$.

Length of simple closed curves extends to 
a continuous \emph{length function}
${\cal T}(S)\times {\cal M\cal L}\to (0,\infty)$
which assigns to a hyperbolic metric $h\in {\cal T}(S)$ and
a measured geodesic lamination $\mu$ the
$h$-length $\ell_\mu(h)$ of $\mu$. 
This length function
satisfies $\ell_{a\mu}(h)=a\ell_\mu(h)$ for all 
$h\in {\cal T}(S),\mu\in {\cal M\cal L}$ and every
$a>0$. For every fixed $h\in {\cal T}(S)$,
the set of all measured geodesic laminations of 
$h$-length one is a section of the projection
${\cal M\cal L}\to {\cal P\cal M\cal L}$.

The following result is due to Wolpert (Corollary 4.7 of
\cite{W87}) and 
is of fundamental importance for this work.

\begin{theorem}\label{wolpert}
For any $\mu\in {\cal M\cal L}$ and 
for every Weil-Petersson geodesic
$\gamma:[a,b]\to {\cal T}(S)$ the function
$t\to \ell_\mu(\gamma(t))$ 
is convex.
\end{theorem}

There is a continuous pairing 
$i:{\cal M\cal L}\times {\cal M\cal L}\to 
[0,\infty)$, the so-called \emph{intersection form}, which
extends the geometric intersection number between 
simple closed curves (see \cite{B86} for this result of Thurston).
Two measured geodesic laminations $\mu,\nu$ \emph{bind}
$S$ if $i(\zeta,\mu)+i(\zeta,\nu)>0$ for every
$\zeta\in {\cal M\cal L}$. Pairs $(\mu,\nu)$ of measured
geodesic laminations which bind $S$ 
and which satisfy $i(\mu,\nu)=1$ correspond
precisely to area one quadratic differentials for $S$ via a map 
which associates to an area one quadratic differential $q$ 
the ordered pair $(q^v,q^h)$ composed of its
vertical and its horizontal 
measured geodesic lamination, respectively.

A \emph{pants decomposition} for $S$ is a collection of $3g-3+m$
pairwise disjoint simple closed essential curves on $S$ which
decompose $S$ into $2g-2+m$ \emph{pairs of pants}. Here by a 
pair of pants we mean a
surface which is homeomorphic to a three-holed sphere. 
By a classical
result of Bers (see \cite{B92}), there is a number $\chi_0 >0$ only
depending on the topological type of $S$ such that for every
complete hyperbolic metric $h$ on $S$ of finite volume,  there is a
pants decomposition for $S$ consisting of simple closed curves of
$h$-length at most $\chi_0$. A number $\chi_0>0$ with this
property is called a \emph{Bers constant} for $S$, and a 
simple closed curve of $h$-length at most $\chi_0$ is called
a \emph{Bers curve} for $h$. Such a curve supports a unique
projective measured lamination.  
A pants decomposition which consists of Bers curves 
for $h$ will be called
a \emph{Bers decomposition} for $h$.

We begin with recalling from \cite{BMM10} (Definition 2.5) 
the definition
of an ending measure for a WP-ray.

\begin{defi}\label{endingmeasure}
A \emph{projective ending measure} $[\mu]\in {\cal P\cal M\cal L}$ 
for a WP-geodesic ray
$\gamma:[0,T)\to {\cal T}(S)$ $(T\in (0,\infty])$ 
is any limit in ${\cal P\cal M\cal L}$ 
of the projective classes of any infinite sequence
of distinct Bers curves for $\gamma$.
\end{defi}

The following theorem combines
Corollary 2.12 and Proposition 4.4 of \cite{BMM10}.
For its formulation, for a number $\epsilon >0$ denote by
${\cal T}(S)_\epsilon$ the set of all points whose
systole is at least $\epsilon$.
Call an infinite WP-ray
$\gamma:[0,\infty)\to {\cal T}(S)$ \emph{recurrent}
if there is some $\epsilon >0$ and an unbounded 
sequence $(t_i)\subset[0,\infty)$ such that 
$\gamma(t_i)\in {\cal T}(S)_\epsilon$.
Here as in the introduction, ${\cal T}(S)_\epsilon\subset
{\cal T}(S)$ is the subset of all surfaces whose
systole is at least $\epsilon$.
 
\begin{theorem}\label{fill}
Let $\gamma:[0,\infty)\to {\cal T}(S)$ be a recurrent
WP-geodesic ray. 
\begin{enumerate}
\item Any two ending measures for $\gamma$ have the
same support. 
\item The support of 
an ending measure for $\gamma$ is
a minimal geodesic lamination which fills $S$. 
\item If $\mu\in {\cal M\cal L}$ 
is any measured geodesic lamination whose
length along $\gamma$ is bounded then the support
of $\mu$ equals the support of an ending measure for $\gamma$.
\end{enumerate}
\end{theorem}

Unlike in the case of Teichm\"uller geodesics, however,
there are recurrent WP-rays so that the support of
an ending measure is not \emph{uniquely ergodic}, i.e. it admits
more than one 
transverse measure up to scale \cite{BMo14}.

The main tool for the proof of the 
second part of Theorem \ref{fellowtravel}
from the introduction is
the \emph{curve graph} ${\cal C\cal G}(S)$ of $S$.
The vertex set of this graph is the set
${\cal C}(S)$ of all free homotopy classes of 
unoriented \emph{essential}
simple closed curves on $S$, i.e. simple closed
curves which are neither contractible
nor freely homotopic into a puncture.
Two vertices 
are joined by an edge if and only if the corresponding
free homotopy classes 
can be realized disjointly. 
Since $3g-3+m\geq 2$ by assumption,
${\mathcal C\cal G}(S)$ 
is connected (see \cite{MM99} and the references given there).
In the sequel we often do not distinguish between
an essential simple closed curve $\alpha$ 
on $S$ and the vertex of the curve graph defined by $\alpha$.

Providing each edge in ${\mathcal C\mathcal G}(S)$ with the standard
euclidean metric of diameter 1 equips the curve graph
with a geodesic metric $d_{\cal C}$. However, 
${\mathcal C\mathcal G}(S)$ is not locally finite and therefore 
the metric space $({\mathcal C\mathcal G}(S),d_{\cal C})$
is not locally compact. Masur
and Minsky \cite{MM99} showed that nevertheless its geometry
can be understood quite explicitly. Namely, ${\mathcal C\cal G}(S)$
is hyperbolic of infinite diameter.
The mapping class group
naturally acts
on ${\mathcal C\mathcal G}(S)$ as a group of simplicial isometries.

Define a map 
\begin{equation}
\Upsilon_{\cal T}: {\cal T}(S)\to
{\cal C}(S)\notag\end{equation}
by associating to a complete hyperbolic metric $h$ on
$S$ of finite volume a Bers
curve $\Upsilon_{\cal T}(h)\in
{\cal C}(S)$. Note that
such a map is not unique. The ambiguity
in this definition is uniformly controlled
\cite{MM99} (or see Lemma 2.1 of \cite{H10a}).

\begin{lemma}\label{bounded}
For every $\chi>0$ there is 
a number $a(\chi)>0$ with the following property.
Let $h\in {\cal T}(S)$ and let 
$\alpha,\beta$ be two simple closed curves of $h$-length at most
$\chi$.
Then $d_{\cal C}(\alpha,\beta)\leq a(\chi)$.
\end{lemma}

Masur and Minsky \cite{MM99} showed that 
$\Upsilon_{\cal T}$ is coarsely
Lipschitz with respect to the 
Teichm\"uller distance $d_{\cal T}$ on ${\cal T}(S)$ and 
the distance $d_{\cal C}$ on the curve graph. 
We use the version which is 
explicitly stated
as Lemma 2.2 in \cite{H10a}.

\begin{lemma}\label{upsilonlipschitz}
There is a number $L_0>1$ such that
\begin{enumerate}
\item
$d_{\cal C}(\Upsilon_{\cal T}(g),\Upsilon_{\cal T}(h))\leq
L_0 d_{\cal T}(g,h)+L_0\text{ for all }g,h\in {\cal T}(S).$
\item  $d_{\cal C}(\Upsilon_{\cal T}(\phi h),\phi\Upsilon_{\cal T}(h))
\leq L_0\text{ for all }h\in {\cal T}(S),\phi\in {\rm Mod}(S).$
\end{enumerate}
\end{lemma}

Let $J\subset \mathbb{R}$ be a closed connected set.
For a number $L>1$, a map $\gamma:J\to {\cal C\cal G}(S)$ is an
\emph{$L$-quasi-geodesic} if 
\[\vert t-s\vert/L-L\leq d_{\cal C}(\gamma(s),\gamma(t))\leq
L\vert t-s\vert +L\text{ for all }s,t\in J.\]
A map $\gamma:J\to {\cal C\cal G}(S)$ is an \emph{unparametrized
$L$-quasi-geodesic} if there is a closed connected set
$I\subset \mathbb{R}$ and a homeomorphism
$\rho:I\to J$ such that $\gamma\circ \rho:I\to {\cal C\cal G}(S)$
is an $L$-quasi-geodesic.

The following result of Masur and Minsky (Theorem 2.3 and Theorem 2.6
of \cite{MM99}) is essential for the proof of Theorem \ref{conjug}.

\begin{theorem}\label{unparam}
There is a number $L_1>1$ such that the image
under $\Upsilon_{\cal T}$ of every
Teichm\"uller geodesic $\gamma:\mathbb{R}\to {\cal T}(S)$
is an unparametrized $L_1$-quasi-geodesic in ${\cal C\cal G}(S)$.
\end{theorem}

In the case that the Teichm\"uller geodesic  $\gamma$
remains entirely in the 
$\epsilon$-thick part of Teichm\"uller space, the path
$\Upsilon_{\cal T}(\gamma)$ is in fact a \emph{parametrized}
$L^\prime$-quasi-geodesic where $L^\prime>0$ only depends on 
$\epsilon$ \cite{H10a}.

We do not know whether the image under the map $\Upsilon_{\cal T}$
of a Weil-Petersson geodesic is an unparametrized quasi-geodesic
in ${\cal C\cal G}(S)$. But we can 
use ending measures to show that the image under $\Upsilon_{\cal T}$
of a Weil-Petersson geodesic ray which is entirely
contained in ${\cal T}(S)_\epsilon$ for some $\epsilon >0$
makes a definite progress in 
${\cal C\cal G}(S)$ in uniformly bounded time.

\begin{lemma}\label{thickgeo}
For every $\epsilon >0$ and every $R>0$ there is a
number $T_0=T_0(\epsilon,R)>0$ with the following
property. Let $b\geq T_0$ and let 
$\gamma:[0,b]\to {\cal T}(S)_\epsilon$
be a Weil-Petersson geodesic. Then 
$d_{\cal C}(\Upsilon_{\cal T}\gamma(0),
\Upsilon_{\cal T}\gamma(b))\geq R$.
\end{lemma}
\begin{proof}
Assume to the contrary that there is some $\epsilon >0$ and some
$R>0$ such that
there is no $T_0(\epsilon,R)>0$ with the properties
stated in the lemma. Then there is for
every $n>0$ a number $T_n>n$ and a WP-geodesic
$\gamma_n:[0,T_n]\to {\cal T}(S)_\epsilon$  
such that $d_{\cal C}(\Upsilon_{\cal T}(\gamma_n(0)),
\Upsilon_{\cal T}(\gamma_n(T_n)))\leq R$.

By coarse equivariance of the map 
$\Upsilon_{\cal T}$ under the action of the mapping
class group (part (2) of Lemma \ref{upsilonlipschitz}) 
and cocompactness of the 
action of ${\rm Mod}(S)$ on 
${\cal T}(S)_\epsilon$, 
we may assume that the 
WP-geodesics $\gamma_n$ 
issue from the same compact set $K\subset {\cal T}(S)_\epsilon$.
Since $\gamma_n\subset {\cal T}(S)_\epsilon$ 
for all $n$, after passing
to a subsequence we may assume that the geodesics $\gamma_n$ 
converge uniformly on compact sets
to a WP-ray $\gamma:[0,\infty)\to {\cal T}(S)_\epsilon$.

For $n>0$ let $\alpha_n\in {\cal M\cal L}$ 
be the measured geodesic lamination with
$\ell_{\alpha_n}(\gamma_n(0))=1$ which is contained in the projective
class of the curve 
$\Upsilon_{\cal T}(\gamma_n(T_n))$, viewed
as a projective measured geodesic lamination. 
Since $\gamma_n(0)\in K\subset{\cal T}(S)_{\epsilon}$, 
the systole of the metric $\gamma_n(0)$ 
is at least $\epsilon$. Therefore the lamination $\alpha_n$ is
obtained by multiplying the simple closed
curve $\Upsilon_{\cal T}(\gamma_n(T_n))$ 
with a weight which is bounded from above
by $1/\epsilon$. Now 
the $\gamma_n(T_n)$-length of the
simple closed curve $\Upsilon_{\cal T}(\gamma_n(T_n))$
is at most $\chi_0$ and hence the $\gamma_n(T_n)$-length
of $\alpha_n$ does not exceed $\chi_0/\epsilon$. By convexity
of the length function along WP-geodesics, 
this implies that the length 
of $\alpha_n$ along $\gamma_n[0,T_n]$ is uniformly bounded,
independent of $n$.

The set 
\[\{\alpha\in {\cal M\cal L}\mid 
\ell_\alpha(x)=1\text{ for some }x\in K\}\]
is compact. Thus by passing
to a subsequence we may assume that the measured
geodesic laminations
$\alpha_n$ converge as $n\to \infty$ 
to a measured geodesic lamination $\alpha$.
By continuity of the length function, 
we have $\ell_{\alpha}(\gamma(0))=1$. Moreover,
the length
of $\alpha$ along $\gamma$ is uniformly bounded. Namely,
by convexity, 
for each $T>0$ and each $n$ which is sufficiently large
that $T_n>T$, we have
\[\ell_{\alpha_n}(\gamma_n(T))\leq \chi_0/\epsilon.\]
Since $\alpha_n\to \alpha$ weakly and  
$\gamma_n(T)\to \gamma(T)$ as $n\to\infty$,  
continuity of the length function yields
$\ell_{\alpha}(\gamma(T))\leq \chi_0/\epsilon$ 
as well. Now $T>0$ was
arbitrary and therefore the length of $\alpha$ is uniformly bounded
along $\gamma$.

We claim that the support of $\alpha$ does not fill up $S$.
This is equivalent to stating  that there is a measured geodesic lamination
$\nu$ whose support does not coincide with the support
of $\alpha$ and such that $i(\nu,\alpha)=0$.

To see that this is indeed the case,
note that since $\gamma_n(0)\to \gamma(0)$, by coarse 
continuity of the map $\Upsilon_{\cal T}$ 
(part (1) of Lemma \ref{upsilonlipschitz})
and by our 
assumption that \[d_{\cal C}(\Upsilon_{\cal T}(\gamma_n(0)),
\Upsilon_{\cal T}(\gamma_n(T_n)))\leq R\text{ for all }n,\]
the distance in ${\cal C\cal G}(S)$ 
between $\Upsilon_{\cal T}(\gamma_n(T_n))$ 
and $c_0=\Upsilon_{\cal T}(\gamma(0))$ is bounded independent of $n$.
By passing to a subsequence we may assume that
this distance equals a fixed number $k\geq 0$. 
Then for each $n$ in the subsequence, there
is a collection $c_n^0,\dots,c_n^k\subset {\cal M\cal L}$ 
of weighted simple closed geodesics of (weighted) length one 
for the metric $\gamma(0)\in {\cal T}(S)$ such that
\[i(c_n^j,c_{n}^{j+1})=0\text{ for every }j<k\]
(here as before,
$i$ is the intersection form) and that $[c_0]=[c_n^0]$ and
$[c_n^k]=[\alpha_n]=[\Upsilon_{\cal T}(\gamma_n(T_n))]$ 
for all $n$ (where $[\mu]$ denotes the projective
class of $\mu\in {\cal M\cal L}$). 
By passing to another subsequence, we
may assume that for each $j$ the
measured geodesic laminations $c_n^j$ 
converge as $n\to \infty$ to 
a measured geodesic lamination
$\nu_j$. By continuity of the intersection form,
we have $i(\nu_j,\nu_{j+1})=0$ for all $j$.
Since $\nu_k=\alpha=\lim_{n\to \infty}\alpha_n$ and since $[\nu_0]=[c_0]$, 
if the support of $\alpha$ fills $S$ then 
the supports of the laminations $\nu_j$ have to coincide
with the support of $\alpha$.
But $[c_0^n]=[c_0]$ for all $n$ and the support of $c_0$ is a
simple closed curve and hence this is impossible
(compare \cite{MM99} for this
argument of Luo). As a consequence, the support of $\alpha$ does not
fill $S$.

However, the WP-ray $\gamma$ is entirely contained in 
${\cal T}(S)_\epsilon$, in particular it
is recurrent. Since the
length of $\alpha$ is bounded along $\gamma$, 
this violates the second part of Theorem \ref{fill}. 
The lemma follows from this contradiction.
\end{proof}

As in the introduction, let 
$d_H$ be the Hausdorff distance
for subsets of ${\cal T}(S)$ with respect to
the Teichm\"uller metric.
We use Lemma \ref{thickgeo} and an idea of Mosher \cite{Mo03} 
to show the second
part of Theorem \ref{fellowtravel} 
from the introduction. For its formulation,
denote as in the introduction by
${\cal Q}_{WP}(S)$ the moduli space of 
quadratic differentials of Weil-Petersson norm one.
Recall moreover the definition of a continuous
conjugacy of two flows on topological spaces. 

\begin{proposition}\label{wpcomp}
\begin{enumerate}
\item
For every $\epsilon >0$ there is a number
$R=R(\epsilon)>0$ with the following property.
Let $J\subset \mathbb{R}$ be a closed connected
set and let $\gamma:J\to {\cal T}(S)_\epsilon$
be a Weil-Petersson geodesic. Then there is a closed connected
set $J^\prime\subset \mathbb{R}$ and there is a 
Teichm\"uller geodesic $\xi:J^\prime\to{\cal T}(S)$ 
with $d_H(\gamma(J),\xi(J^\prime))\leq R$.
\item 
Let $C\subset {\cal Q}_{WP}(S)$ be a compact set which
is invariant under the geodesic flow $\Phi^t_{WP}$ for the
Weil-Petersson metric. Then there is a continuous conjugacy 
$\Psi:C\to {\cal Q}(S)$ of the restriction
of $\Phi^t_{WP}$ to $C$ into the Teichm\"uller geodesic
flow.
\end{enumerate}
\end{proposition}
\begin{proof} 
Let $\epsilon >0$, let $a(\chi_0)>1$ be
as in Lemma \ref{bounded}
and let $T_0=T_0(\epsilon,2a(\chi_0)+3)>0$ 
be as in Lemma \ref{thickgeo}. Note that
$T_0$ only depends on $\epsilon$.

Unit balls in the cotangent bundle of ${\cal T}(S)$ 
for both the Teichm\"uller metric and the Weil-Petersson metric depend
continuously on the basepoint. Thus by
invariance under the action of the
mapping class group and cocompactness of the 
action of ${\rm Mod}(S)$ on ${\cal T}(S)_\epsilon$, 
the restriction
to ${\cal T}(S)_\epsilon$ of the Weil-Petersson metric
is locally uniformly bilipschitz equivalent to the
restriction of the Teichm\"uller metric. 
Hence there is a number $L=L(\epsilon)>1$ such that
$d_{\cal T}(\gamma(0),\gamma(b))\leq Lb$ 
for any WP-geodesic $\gamma:[0,b]\to {\cal T}(S)_\epsilon$ 
where as before, $d_{\cal T}$ is the distance induced by the
Teichm\"uller metric. As a consequence,
for every $b\leq T_0$, every
WP-geodesic $\gamma:[0,b]\to {\cal T}(S)_\epsilon$ 
is entirely contained in the ball of radius $LT_0$ 
about $\gamma(0)$ for the Teichm\"uller metric and therefore 
$d_H(\gamma[0,b],\gamma(0))\leq LT_0$. 
Thus it is enough to show the proposition
for Weil-Petersson geodesics in ${\cal T}(S)_\epsilon$
of length at least $T_0$. 

By Lemma \ref{thickgeo}, if
$\gamma:[b,c]\to {\cal T}(S)_\epsilon$ is a WP-geodesic
of length $c-b\geq T_0$ then 
$d_{\cal C}(\Upsilon_{\cal T}\gamma(b),
\Upsilon_{\cal T}\gamma(c))\geq 2a(\chi_0)+3$ and hence
by Lemma \ref{bounded}, 
the $\gamma(b)$-length of any simple closed curve
$\alpha\in {\cal C}(S)$ with $\ell_{\alpha}(\gamma(c))\leq \chi_0$ is
bigger than $\chi_0$. In particular, by convexity of length functions
along Weil-Petersson geodesics, we have $\ell_{\alpha}(\gamma(t))\leq 
\ell_{\alpha}(\gamma(b))$ for every $t\in [b,c]$.

For the WP-geodesic $\gamma:[b,c]\to {\cal T}(S)_\epsilon$
we say that the projective measured geodesic
lamination $[\alpha]$ defined by a simple closed curve
$\alpha\in {\cal C}(S)$  
is \emph{realized} at the right endpoint 
$c$ of the parameter interval $[b,c]$ if the $\gamma(c)$-length of 
$\alpha$ does not exceed $\chi_0$.
By the above, we then have 
$\ell_{\alpha}(\gamma(t))\leq \ell_{\alpha}(\gamma(b))$
for every $t\in [b,c].$
If $\gamma:[0,\infty)\to {\cal T}(S)_\epsilon$
is an infinite WP-geodesic ray 
then the projectivization 
$[\lambda]\in {\cal P\cal M\cal L}$ of 
a measured geodesic lamination
$\lambda$ is \emph{realized} at the right endpoint $\infty$ 
if the length of $\lambda$ 
along $\gamma[0,\infty)$ 
assumes its maximum at $\gamma(0)$. By Theorem \ref{fill},
Lemma \ref{thickgeo} and the above discussion,
an ending measure for $\gamma$ is realized at the
right endpoint of $J=[0,\infty)$, moreover any projective
measured geodesic lamination 
which is realized at $\infty$ is supported in the support of
an ending measure.

Using an idea of Mosher \cite{Mo03},
define
$\Gamma_\epsilon$ to be the set of all
triples
$(\gamma:J\to {\cal T}(S)_\epsilon,\mu_+,\mu_-)$ with the
following properties.
\begin{enumerate}
\item $J\subset \mathbb{R}$ is a closed connected set
of diameter at least $T_0$ 
containing $0$.
\item $\gamma:J\to {\cal T}(S)_\epsilon$ is a Weil-Petersson
geodesic.
\item $\mu_+,\mu_-\in {\cal M\cal L}$ are measured geodesic
laminations of $\gamma(0)$-length one whose projectivizations
$[\mu_+],[\mu_-]$
are realized at the 
right and left endpoint of $J$, respectively.
\end{enumerate}
We equip $\Gamma_\epsilon$ with the product topology, using the
weak$^*$-topology on ${\cal M\cal L}$ for the second and
third factor and the compact open topology for the
arcs $\gamma:J\to {\cal T}(S)_\epsilon$.
The mapping class group ${\rm Mod}(S)$ naturally acts
diagonally on $\Gamma_\epsilon$.

We follow Mosher (Proposition 3.17 of \cite{Mo03}) and
show that the action of ${\rm Mod}(S)$
on $\Gamma_\epsilon$ is cocompact. 
Since ${\rm Mod}(S)$ acts isometrically and cocompactly
on ${\cal T}(S)_\epsilon$, for this it is enough to show
that the subset of $\Gamma_\epsilon$ consisting of 
all triples with the additional property that 
$\gamma(0)$ is contained in a fixed compact subset
$A$ of ${\cal T}(S)_\epsilon$ is compact.
Now the topology of $\Gamma_\epsilon$ is metrizable and hence
this follows if every sequence in $\Gamma_\epsilon$
contained in the subset $\{(\gamma:J\to {\cal T}(S),
\mu_+,\mu_-)\in \Gamma_\epsilon\mid \gamma(0)\in A\}$
has a convergent subsequence.

By the Arzela Ascoli theorem (or simply by properties
of geodesics for the Weil-Petersson metric),
the set of geodesic arcs $\gamma:J\to {\cal T}(S)_\epsilon$
where $J\subset \mathbb{R}$ is a closed connected subset
containing $0$ and such that $\gamma(0)\in A$
is compact with respect to the compact open topology. 
As the length function is continuous on ${\cal T}(S)
\times {\cal M\cal L}$, 
it is enough to show that the following holds.
Let $\gamma_i:J_i\to {\cal T}(S)_\epsilon$ $(i>0)$ 
be a sequence of Weil-Petersson geodesics which converge
locally uniformly to $\gamma:J\to {\cal T}(S)_\epsilon$. 
For each $i$ let $\mu_i$ be a measured
geodesic lamination of $\gamma_i(0)$-length one
whose projectivization 
$[\mu_i]$ is realized
at the right endpoint of $J_i$. If $\mu_i\to \mu\in 
{\cal M\cal L}$, then the projectivization $[\mu]$ 
of $\mu$ is realized 
at the right endpoint of $J$.

Assume first
that $J\cap [0,\infty)=[0,b]$ for some $b\in (0,\infty)$.
Then for sufficiently large $i$ we have
$J_i\cap [0,\infty)=[0,b_i]$ with $b_i\in (0,\infty)$ and
$b_i\to b$. Thus $\gamma_i(b_i)\to \gamma(b)$
$(i\to\infty)$ and therefore 
by continuity of length functions and the collar lemma, 
there is only a \emph{finite}
number of simple closed curves $\alpha\in {\cal C}(S)$
of length at most $\chi_0$  
with respect to one of the metrics $\gamma_i(b_i),\gamma(b)$.
Hence by passing
to a subsequence we may assume that there is a curve 
$\alpha\in {\cal C}(S)$ such that 
$[\mu_i]=[\alpha]=[\mu]$ for all sufficiently
large $i$. 
The $\gamma_i(b_i)$-length of $\alpha$ is
at most $\chi_0$ for all sufficiently large $i$ and hence by
continuity of length functions, 
the same is true for the $\gamma(b)$-length of $\alpha$.
In other words, 
the limit $[\mu]\in {\cal P\cal M\cal L}$ of the
sequence $[\mu_i]$ is realized at the right endpoint $b$ of $J$.

The same argument is also valid if the right 
endpoint of $J$ is infinite.
Namely, assume first that $J_i\cap [0,\infty)=[0,b_i]$ for some
$b_i>0$ with $b_i\to \infty$ $(i\to \infty)$. 
By the above discussion, for sufficiently large $i$ (namely, for all
$i$ such that $b_i>T_0$) 
the length
of $\mu_i$ along $\gamma_i[0,b_i]$ assumes its maximum at $\gamma_i(0)$. 
Thus if $T>0$ is arbitrary and if $i>0$ is 
sufficiently large that 
$b_i>\max\{T_0,T\}$ then $\ell_{\mu_i}(\gamma_i(T))\leq
\ell_{\mu_i}(\gamma_i(0))$. Since $\gamma_i(0)\to \gamma(0)$ and
$\gamma_i(T)\to \gamma(T)$ and $\mu_i\to \mu$, 
continuity of the length
function implies that $\ell_{\mu}(\gamma(T))\leq \ell_{\mu}(\gamma(0))$.
Now $T>0$ was arbitrary and therefore 
the length of $\mu$ along $\gamma$ assumes its
maximum at $\gamma(0)$.
However, this just means that the projectivization
$[\mu]$ is realized at the right infinite endpoint of $J$. 
The case that $b_i=\infty$ for infinitely many $i$ follows in the same way.

Any two simple closed curves $\alpha,\beta\in {\cal C}(S)$
with $d_{\cal C}(\alpha,\beta)\geq 3$ bind $S$.
Thus by 
Theorem \ref{fill}, Lemma \ref{thickgeo} and the choice of $T_0$,
for any $(\gamma:J\to {\cal T}(S),\mu_+,\mu_-)\in \Gamma_\epsilon$
the measured geodesic laminations $\mu_+,\mu_-$ bind $S$.
These laminations then  
determine up to parametrization a Teichm\"uller geodesic
$\eta([\mu_+],[\mu_-])$ 
whose vertical and horizontal projective
measured geodesic laminations
are just the classes $[\mu_+],[\mu_-]$. 

Let $\sigma(\gamma,\mu_+,\mu_-)$ be the unique
point on $\eta([\mu_+],[\mu_-])$ which is the foot-point
of the quadratic differential
with vertical and horizontal measured
geodesic laminations $\mu_+/\sqrt{i(\mu_+,\mu_-)},
\mu_-/\sqrt{i(\mu_+,\mu_-)}$.
By continuity of the length function and the
intersection form, the map taking
$(\gamma:J\to {\cal T}(S)_\epsilon,\mu_+,\mu_-)\in \Gamma_\epsilon$ to
$(\gamma(0),\sigma(\gamma,\mu_+,\mu_-)) \in {\cal
T}(S)\times {\cal T}(S)$ is continuous. Moreover
by construction, this map is equivariant with
respect to the natural diagonal 
action of ${\rm Mod}(S)$ on $\Gamma_\epsilon$
and on ${\cal T}(S)\times {\cal T}(S)$. Since the action of
${\rm Mod}(S)$ on $\Gamma_\epsilon$ is cocompact, the same is true for
the action of ${\rm Mod}(S)$ on the image of this map.
Thus the Teichm\"uller distance between $\gamma(0)$ and
$\sigma(\gamma,\mu_+,\mu_-)$ is bounded from above by a
universal constant $b>0$.

Let again $(\gamma:J\to {\cal T}(S)_\epsilon,
\mu_+,\mu_-)\in \Gamma_\epsilon$.
For each $s\in J$ define
\[a_-(s)=\frac{1}{\ell_{\mu_-}(\gamma(s))},\quad a_+(s)=\frac{1}
{\ell_{\mu_+}(\gamma(s))}.\] 
Let moreover $\gamma^s(t)=
\gamma(t+s)$. Then the ordered triple
$(\gamma^s,a_+(s)\mu_+,a_-(s)\mu_-)$ lies in the
${\rm Mod}(S)$-cocompact set $\Gamma_\epsilon$ and hence the distance
between $\gamma(s)$ and the point
$\sigma(\gamma^s,a_+(s)\mu_+,a_-(s)\mu_-)
\in \eta([\mu_+],[\mu_-])$
is at most $b$. As a consequence,
$\gamma(J)$ is contained in the $b$-neighborhood of the
geodesic $\eta([\mu_+],[\mu_-])$. 

Now $s\in J$ was arbitrary and the ordered triple
$(\gamma^s,a_+(s)\mu_+,a_-(s)\mu_-)$ depends continuously
on $s$ with respect to the topology of $\Gamma_\epsilon$. Hence
$\gamma(J)$ is contained in the $b$-neighborhood of 
a suitably chosen subarc of
$\eta([\mu_+],[\mu_-])$. Moreover, 
the map $\gamma(s)\to \sigma(\gamma^s,a_+(s)\mu_+,a_-(s)\mu_-)$
is
continuous in $s$. This means that the image subarc is contained
in the $b$-neighborhoood of $\gamma(J)$. Together we showed
that the Hausdorff distance between $\gamma(J)$ and 
a subarc of $\eta([\mu_+],[\mu_-])$ is at most $b$. 
The first part of the proposition is proven.

The results obtained so far show that
if $\gamma:\mathbb{R}\to {\cal T}(S)_\epsilon$
is any biinfinite Weil-Petersson geodesic and if 
$[\mu_+],[\mu_-]$ are ending measures for $\gamma_+=\gamma[0,\infty),
\gamma_-=\gamma(-\infty,0]$ then the Teichm\"uller geodesic
$\eta([\mu_+],[\mu_-])$ defined by a quadratic differential
with vertical measured geodesic lamination in the class 
$[\mu_+]$ and horizontal measured geodesic lamination
in the class $[\mu_-]$ is a uniform fellow-traveler of $\gamma$,
measured with respect to the Teichm\"uller metric.
In particular, there is a number $\kappa>0$ only
depending on $\epsilon$ such that
$\eta([\mu_+],[\mu_-])$ is contained in ${\cal T}(S)_\kappa$. 
Hence by a result of Masur \cite{M82}, 
the projective measured geodesic laminations $[\mu_+],[\mu_-]$ are
uniquely ergodic. Theorem \ref{fill} then implies that  
a projective ending measure
for a subray of $\gamma$ is unique.

Let $\mu_+(\gamma),\mu_-(\gamma)\in {\cal M\cal L}$ be
the representative of the forward and backward projective ending
measures whose $\gamma(0)$-length equals one. 
We claim that $\mu_+(\gamma)$ and $\mu_-(\gamma)$ 
depend continuously
on $\gamma$. Namely, assume that $\gamma_i\to \gamma$ uniformly
on compact sets and that $\beta$ is a weak limit of 
the sequence $(\mu_+(\gamma_i))$. Then by continuity of 
length functions we have $\ell_{\beta}(\gamma(0))=1$, moreover
we conclude as above (see also the proof of 
Lemma \ref{thickgeo}) that 
the length of $\beta$ is bounded along $\gamma[0,\infty)$. 
Since these two properties determine the
measured geodesic lamination $\mu_+(\gamma)$ uniquely,
continuous dependence of $\mu_+(\gamma)$ on 
$\gamma$ is immediate. Continuous dependence
of $\mu_-(\gamma)$ on $\gamma$ follows in the same way.

Let $C\subset {\cal Q}_{WP}(S)$ be a compact
set which is invariant under the geodesic flow $\Phi^t_{WP}$
for the Weil-Petersson metric. Let $\tilde C$ be the preimage
of $C$ in the space of quadratic differentials of Weil-Petersson
norm one. By compactness of $C$ there is a number $\epsilon >0$ such that for 
every $q\in \tilde C$ 
the Weil-Petersson geodesic $\gamma_q$ with 
initial velocity $q$ is entirely contained in ${\cal T}(S)_\epsilon$.

For $q\in {\tilde C}$ define   
$\tilde \Psi(q)\in \tilde {\cal Q}(S)$ to be the unique area one
quadratic differential
with vertical and horizontal measured geodesic lamination
\[\mu_+(\gamma_q)/\sqrt{i(\mu_+(\gamma_q),\mu_-(\gamma_q))},\quad 
\mu_-(\gamma_q)/\sqrt{i(\mu_+(\gamma_q),\mu_-(\gamma_q))},\] 
respectively. Here $\mu_+(\gamma_q)$ (or $\mu_-(\gamma_q)$) 
is as before the forward (or backward) ending measure for the geodesic 
$\gamma_q$ with initial velocity $q$. 
Then $q\to \tilde\Psi(q)$ is continuous 
and equivariant with respect to the
action of the mapping class group 
and hence it projects to a 
map $\Psi:C\to {\cal Q}(S)$. 
By convexity of length functions 
along WP-geodesics, the length of $\mu_+(\gamma_q)$ is 
strictly decreasing along $\gamma_q$, and the 
length of $\mu_-(\gamma_q)$ 
is strictly increasing. This implies that 
the restriction of $\tilde \Psi$ to 
the unit cotangent line $\gamma_q^\prime$ of the WP-geodesic  
$\gamma_q$ is a homeomorphism onto 
the unit cotangent line of the Teichm\"uller geodesic $\tilde \gamma$
with initial velocity 
$\tilde \gamma^\prime(0)=\tilde \Psi(q)$. Therefore the 
map $\Psi$ defines a continuous conjugacy
of the Weil-Petersson flow  
on $C$ into the
Teichm\"uller flow as defined in the introduction. This shows 
the second part of the proposition.
\end{proof}

\section{Asymptotic rays for the Weil-Petersson metric}
\label{asymptotic}

The goal of this 
section is to establish some differential geometric properties
of Weil-Petersson geodesic rays which are
needed for the proof of 
Theorem \ref{conjug} 
from the introduction.

Using the assumptions and notations from Section 2, 
we begin with 
collecting some additional results from
\cite{BMM10}.

The completion 
$\overline{{\cal T}(S)}$ of ${\cal T}(S)$ for the Weil-Petersson metric
is a ${\rm CAT}(0)$-space. Call two infinite Weil-Petersson geodesic rays
$\gamma:[0,\infty)\to {\cal T}(S)$, $\xi:[0,\infty)\to 
{\cal T}(S)$ \emph{asymptotic} if the function
\[t\to d_{WP}(\gamma(t),\zeta(t))\] 
is bounded. Since $\overline{{\cal T}(S)}$ is neither
locally compact nor
hyperbolic in the sense of Gromov (see \cite{W03}
for an overview and for references), it is difficult to 
find out whether or not for  
two given infinite non-asymptotic WP-rays $\gamma_1,\gamma_2$ there is
a biinfinite WP-geodesic which is forward asymptotic to 
$\gamma_1$ and backward asymptotic to $\gamma_2$.

Brock, Masur and Minsky \cite{BMM10}
found a sufficient condition for the existence
of a biinfinite WP-geodesic which is forward and backward
asymptotic to two given WP-rays. 
Namely, as in Section 2, 
call a WP-ray $\gamma:[0,\infty)\to {\cal T}(S)$ recurrent if 
there is a number $\epsilon >0$ and there 
is a sequence of numbers $t_i\to \infty$ such that
$\gamma(t_i)\in {\cal T}(S)_\epsilon$ for all $i$. In other words,
a geodesic ray is recurrent if its projection to moduli space
returns to a fixed compact set for arbitrarily large times.
Theorem 1.3 of \cite{BMM10} shows that 
for every recurrent 
WP-geodesic ray $\gamma$ and for \emph{every} WP-geodesic ray
$\xi$ which is not asymptotic to $\gamma$ there is a biinfinite WP-geodesic
which is forward asymptotic to $\gamma$ and backward asymptotic to $\xi$.

We use some ideas from \cite{BMM10} to establish
a related technical result (Corollary \ref{asympt1})
which is used in an essential way
in the proof of Theorem \ref{conjug}.

We begin with an 
observation which is a consequence of the Gau\ss{}-Bonnet
formula for ruled surfaces as in \cite{BMM10}. 
For its formulation, for $\epsilon >0$ let 
\begin{equation}\label{bepsilon}
2b(\epsilon)=\inf\{d_{\rm WP}(x,y)\mid x\in {\cal T}(S)_\epsilon,
y\in \overline{{\cal T}(S)}-{\cal T}(S)\}.\end{equation}
By invariance
of ${\cal T}(S)_\epsilon$ 
under the action of ${\rm Mod}(S)$ and cocompactness, we
have $b(\epsilon)>0$
(in fact $b(\epsilon)\asymp \epsilon^{1/2}$ by Wolpert's estimate
\cite{W03}). 
Moreover, the sectional curvature of the
Weil-Petersson metric on the $b(\epsilon)$-neighborhood
of ${\cal T}(S)_\epsilon$ is bounded from above
by a negative constant. 

A \emph{geodesic quadrangle} in $({\cal T}(S),d_{WP})$ 
consists of four WP-geodesic segments connecting four distinct
points in ${\cal T}(S)$. We always assume that a geodesic
quadrangle $Q$ is non-degenerate, i.e. that that no vertex of 
$Q$ is contained in the interior of any side of $Q$. 
Two sides $\alpha,\beta$ of such 
a quadrangle are \emph{opposite} if 
they do not share a vertex.

For a Weil-Petersson geodesic segment
$\gamma:[0,\tau]\to {\cal T}(S)$ and $\epsilon >0$ let
$\ell_{\epsilon-{\rm thick}}(\gamma)$ 
be the length of the intersection of
$\gamma$ with ${\cal T}(S)_\epsilon$
(in other words, $\ell_{\epsilon-{\rm thick}}(\gamma)$ is the Lebesgue
measure of the set $\{t\in [0,\tau]\mid \gamma(t)\in 
{\cal T}(S)_\epsilon\})$. 

In the remainder of this section, distances are always
distances with respect to the Weil-Petersson metric.
Moreover, angles are always unoriented angles
with respect to the Weil-Petersson inner product.

\begin{lemma}\label{gaussb}
\begin{enumerate}
\item
For every $\epsilon >0$ and every
$\alpha >0$ there is a number 
$k_1=k_1(\epsilon,\alpha)>0$ with the following
property. Let $\tau\geq k_1$ and let 
$\gamma:[0,\tau]\to {\cal T}(S)$
be a Weil-Petersson geodesic segment with 
$\ell_{\epsilon-{\rm thick}}(\gamma)
\geq k_1$. 
Assume that $\gamma$ is a side of a geodesic quadrangle $Q$ 
with angles at least $\alpha$ at $\gamma(0),\gamma(\tau)$.
Then the side of $Q$ which is 
opposite to $\gamma$ passes through the 
$b(\epsilon)$-neighborhood of $\gamma[0,\tau]\cap {\cal T}(S)_{\epsilon}$.
\item For every $\epsilon >0,\alpha>0$ and every $\theta >0$
there is a number $k_2=k_2(\epsilon,\alpha,\theta)>0$ with the following
property. Let $\tau\geq k_2$ and let $\gamma:[0,\tau]\to 
{\cal T}(S)$ be a Weil-Petersson geodesic segment 
with $\ell_{\epsilon-{\rm thick}}(\gamma)\geq k_2$. Assume
that $\gamma$ is a side of a geodesic triangle $T$ with angle
at least $\alpha$ at $\gamma(\tau)$. Then the angle
of $T$ at $\gamma(0)$ is at most $\theta$.
\end{enumerate}
\end{lemma}
\begin{proof} The idea of the proof is taken
from \cite{BMM10} (and has been used before
by other authors, notably by Bonahon \cite{B86} and 
Canary \cite{C93}). Namely, 
let for the moment $\tau>0$ be arbitrary and let 
$\gamma:[0,\tau]\to {\cal T}(S)$
be a WP-geodesic segment. Let $Q$ be a WP-geodesic 
quadrangle with vertices $\gamma(0),\gamma(\tau),
x_1,x_2$ and such that $\gamma(\tau)$ is connected
to $x_1$ by a side. 

The vertices $\gamma(0),\gamma(\tau),x_1$ determine
a WP-geodesic triangle which can be filled by
WP-geodesic segments issuing from the vertex
$x_1$ and connecting $x_1$ to the opposite side $\gamma$.
The thus obtained \emph{ruled
triangle} $T$ is an embedded subsurface of ${\cal T}(S)$ 
which is smooth in its interior,
with piecewise geodesic boundary. 
The intrinsic distance in $T$ between any two points
$x,y\in T$ is not smaller than $d_{\rm WP}(x,y)$.
Moreover, the (intrinsic) Gau\ss{} curvature
of $T$ with respect to the restriction of the Weil-Petersson
metric at a point $x\in T$
does not exceed the maximum of the sectional curvatures of
the Weil-Petersson metric at $x$.
In particular, the Gau\ss{} curvature of $T$ 
is negative, and for every $\epsilon >0$ 
there is a constant $\kappa(\epsilon)>0$ only depending on $\epsilon$ 
such that the Gau\ss{} curvature at every point in $T$ whose distance
to ${\cal T}(S)_{\epsilon}$ 
is at most $b(\epsilon)$ is bounded from above by
$-\kappa(\epsilon)$
(see \cite{BMM10} for this
construction of Bonahon \cite{B86}
and Canary \cite{C93} and for references).

\begin{figure}[ht] \begin{center} \psfrag{x1}{\tiny$x_1$}
    \psfrag{x2}{\tiny$x_2$} \psfrag{d1}{\tiny$\gamma{(\tau)}$}
    \psfrag{d2}{\tiny$\gamma{(\sigma)}$}
    \psfrag{d3}{\tiny$\gamma{(0)}$} \psfrag{a}{\tiny$\alpha$}
    \psfrag{bG}{\tiny$\beta_{\sigma}$}
    \psfrag{xi}{\tiny$\xi{(\sigma')}$} \psfrag{zeta}{\tiny$\zeta$}
    \psfrag{T}{\tiny$T$} \psfrag{T'}{\tiny$T'$} \includegraphics
    [width=0.7\textwidth] {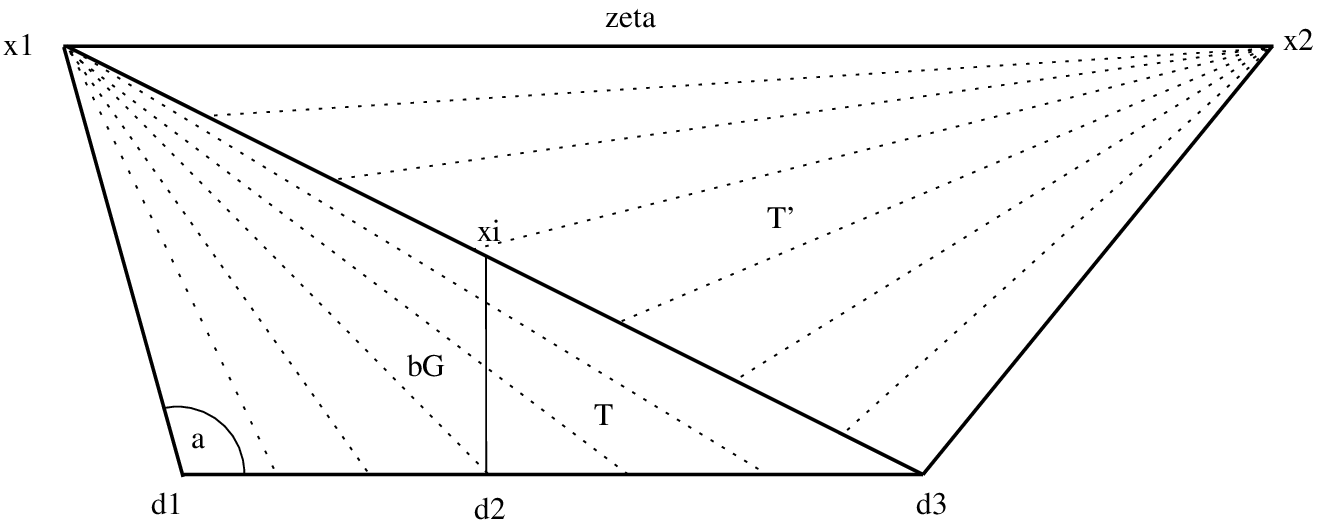}
\end{center}
\end{figure}

Even though the ruled triangle $T$ is not naturally
an embedded subsurface with piecewise geodesic
boundary of a smooth simply connected surface $U$
with a Riemannian metric,
the intrinsic angle of $T$ is
defined at every vertex of $T$. At the vertex $x_1$, this 
angle is just the length of the arc in the unit tangent
sphere for the Weil-Petersson metric which
consists of the directions of all geodesics joining $x_1$ to
$\gamma$. This length exists since the initial direction
of a geodesic depends smoothly on its endpoints. 
We claim
that the intrinsic angles of $T$ at the vertices $\gamma(0),\gamma(\tau)$
coincide with the angles for the Weil-Peterssen metric.
Namely, by slightly extending the
WP-geodesic $\gamma$ and the geodesics defining the ruling
of $T$ we obtain a smooth subsurface of ${\cal T}(S)$ 
containing a neighborhood of $\gamma(0),\gamma(\tau)$ in its interior.
The intrinsic angle at $\gamma(0),\gamma(\tau)$ of the triangle
$T$ is just the angle of $T$ at $\gamma(0),\gamma(\tau)$ 
in this subsurface equipped with the
restriction of the Weil-Petersson metric.

Let $\epsilon >0$ and let $\kappa=\kappa(\epsilon)>0$ be
as in the second paragraph of this proof.
Consider first the case that
the angle of the quadrangle $Q$ at the vertex
$\gamma(\tau)$ is not smaller than $\pi/2$.  
By the consideration
in the previous paragraph, this angle coincides with
the intrinsic
angle at $\gamma(\tau)$ of the 
triangle $T$.
By the Gau\ss{}-Bonnet formula,
the integral of the Gau\ss{} 
curvature over the triangle $T$ 
equals the sum of the intrinsic angles of $T$ minus $\pi$. 
Thus  if $\theta_1,\theta_2$ are the angles
of $T$ at $\gamma(0),x_1$ then 
this curvature integral
is not smaller than $-\pi/2 +\theta_1+\theta_2> -\pi/2$. 

Denote by $\xi$ be the side of $T$ connecting 
$\gamma(0)=\xi(0)$ to 
$x_1$.
The triangle $T$ is negatively curved
and therefore
intrinsic distance functions in $T$ are convex. 
Since the angle of $T$ at $\gamma(\tau)$
is not smaller than $\pi/2$, for each
$t\in [0,\tau]$ 
the endpoint of the intrinsic 
geodesic arc $\beta_t$ in $T$ which 
issues from $\gamma(t)$ and is perpendicular
to $\gamma$ at $\gamma(t)$ is contained in the side $\xi$. 
The length of $\beta_t$ equals the intrinsic distance between
its endpoint on $\xi$ and the side $\gamma$ of $T$ and hence
by convexity of the distance function, this length
is increasing with $t$. Thus 
if there is a number $\sigma\in [0,\tau]$ 
so that the length of $\beta_\sigma$ is 
not smaller than
$b(\epsilon)/2$, then
for every $s\in [\sigma,\tau]$ the length
of $\beta_s$ is not smaller than $b(\epsilon)/2$. Then 
$T$ contains an embedded strip $A_0$ of width
$b(\epsilon)/2$ and length $\tau-\sigma$ which consists of
all points in $T$ on the initial subsegments of length
$b(\epsilon)/2$ of the geodesics $\beta_s$
for all $s\in [\sigma,\tau]$. 

Assume that there is such a point 
$\sigma\in [0,\tau]$ such that moreover 
$\ell_{\epsilon-{\rm thick}}(\gamma[\sigma,\tau])
\geq \pi/\kappa(\epsilon) b(\epsilon)$
where $\kappa(\epsilon)>0$ is as in the second paragraph of 
this proof. 
Let $A\subset A_0$ be the closed subset of 
the embedded strip $A_0\subset T$ 
which consists
of the union of all initial subarcs of length $b(\epsilon)/2$
of those of the geodesic arcs $\beta_s$ 
which issue from a point in 
$\gamma[\sigma,\tau]\cap {\cal T}(S)_{\epsilon}$.
Since $\ell_{\epsilon-{\rm thick}}(\gamma[\sigma,\tau])
\geq \pi/\kappa(\epsilon) b(\epsilon)$ and the curvature of $T$ is negative,
comparison with the euclidean plane shows that 
the area of $A$ is at least
$\pi/2\kappa(\epsilon)$. Now the Gau\ss{} curvature
of $T$ at every point of $A$ does not exceed 
$-\kappa(\epsilon)$
and therefore the integral
of the Gau\ss{} curvature over $T$ is smaller than
$-\pi/2$. This contradicts the above observation that
by the Gau\ss{} Bonnet formula, this
curvature integral is bigger than $-\pi/2$. 
As a consequence, 
if $\sigma\in [0,\tau]$ is such that
$\ell_{\epsilon-{\rm thick}}(\gamma[\sigma,\tau])\geq
\pi/\kappa(\epsilon) b(\epsilon)$ then for $s\in [0,\sigma]$ the length
of the geodesic arc $\beta_s$ does not exceed $b(\epsilon)/2$.

Now assume more restrictively that the angles of the 
quadrangle $Q$ at the vertices $\gamma(0),\gamma(\tau)$
are not smaller than $3\pi/4$. Then the angle
of the triangle $T$ at the vertex $\gamma(\tau)$ is not smaller
than $3\pi/4$.
Since the triangle $T$ is negatively curved, the sum of its angles
is smaller than $\pi$. In particular, the Weil-Petersson angle
at $\gamma(0)$ between $\gamma$ and the WP-geodesic $\xi$ 
connecting $\gamma(0)$ to $x_1$ does not exceed $\pi/4$. 

Assume that $\ell_{\epsilon-{\rm thick}}(\gamma)\geq 2\pi/\kappa(\epsilon) b(\epsilon)$
and let $\sigma\in [0,\tau]$ be such that
\[\ell_{\epsilon-{\rm thick}}(\gamma[0,\sigma])
\geq \pi/\kappa(\epsilon) b(\epsilon)\text{ and  }
\ell_{\epsilon-{\rm thick}}(\gamma[\sigma,\tau])\geq
\pi/\kappa(\epsilon) b(\epsilon).\] Let 
$\sigma^\prime>\sigma$ be such that 
$\xi(\sigma^\prime)$ is the endpoint of 
the geodesic segment $\beta_\sigma$ 
in the ruled triangle $T$ issuing perpendicularly from $\gamma(\sigma)$.
Since 
$\ell_{\epsilon-{\rm thick}}(\gamma[\sigma,\tau])\geq \pi/\kappa(\epsilon) b(\epsilon)$,
the above
consideration shows that  
the distance between $\gamma(\sigma)$ and 
$\xi(\sigma^\prime)$ does not exceed $b(\epsilon)/2$.
Since the distance function for the Weil-Petersson metric
is convex, we conclude that $d_{WP}(\xi(t),\gamma(t\sigma/\sigma^\prime))\leq
b(\epsilon)/2$ for all $t\in [0,\sigma^\prime]$.

Let $T^\prime$ be the ruled triangle obtained by connecting
$x_2$ to each point of $\xi$ by a geodesic arc. 
The angle at $\gamma(0)$ of $T^\prime$ 
is at least $\pi/2$. 
Apply the above discussion to the ruled triangle 
$T^\prime$, the
side $\xi$ of $T^\prime$ and the side 
$\zeta$ connecting
$x_1$ to $x_2$. 
Note that $\zeta$ is the side of 
the quadrangle $Q$ opposite to $\gamma$. 

By the assumption that 
$\ell_{\epsilon-{\rm thick}}(\gamma[0,\sigma])\geq
\pi/\kappa(\epsilon) b(\epsilon)$, if
the distance between $\zeta$ and  
$\gamma[0,\sigma]\cap {\cal T}(S)_\epsilon$ is at least
$b(\epsilon)$ then there is an embedded strip
in $T^\prime$ with curvature integral 
smaller than $-\pi/2$. This strip consists of 
all points on geodesic arcs in $T^\prime$ of length 
$b(\epsilon)/2$ issuing perpendicularly
from a point $\xi(t)$ for some $t\in [0,\sigma^\prime]$ 
such that $\gamma(t\sigma/\sigma^\prime)\in {\cal T}(S)_\epsilon$.
Since the angle of $T^\prime$ at $\gamma(0)$ 
is bounded from below by $\pi/2$, this is impossible.
As a consequence, the side $\zeta$
intersects the $b(\epsilon)$-neighborhood of 
$\gamma[0,\sigma]\cap {\cal T}(S)_\epsilon$.
This shows the first part of the 
lemma for $\alpha=3\pi/4$ with $k_1(\epsilon,\alpha)=
2\pi/\kappa(\epsilon) b(\epsilon)$.

The first part of the lemma for an arbitrary angle $\alpha>0$
is now a consequence of the second part of the lemma for 
$\epsilon,\alpha$ and $\theta=\pi/4$.

Namely, 
assume that the second part of 
the lemma holds true and let $Q$ be a geodesic 
quadrangle with a side $\gamma:[0,\tau]\to {\cal T}(S)$ and 
angles at least $\alpha$ at the vertices $\gamma(0),\gamma(\tau)$.
Let $x_1,x_2$ be the vertices of $Q$ which are distinct from
$\gamma(0),\gamma(\tau)$ and assume that 
$\ell_{\epsilon-{\rm thick}}(\gamma)\geq 2\pi/\kappa(\epsilon) b(\epsilon)+
2k_2$ where $k_2=k_2(\epsilon,\alpha,\pi/4)>0$ is 
as in the second part of the lemma.
Let $0<t_1<t_2<\tau$ be such that
\[\ell_{\epsilon-{\rm thick}}(\gamma[0,t_1])\geq k_2,
\ell_{\epsilon-{\rm thick}}(\gamma[t_1,t_2])\geq 2\pi/\kappa(\epsilon) b(\epsilon),
\ell_{\epsilon-{\rm thick}}(\gamma[t_2,\tau])\geq k_2\]
and let $Q^\prime$ be the quadrangle with vertices 
$\gamma(t_1),\gamma(t_2),x_1,x_2$. The side $\zeta$ of $Q^\prime$ opposite
to $\gamma[t_1,t_2]$ coincides with the side of $Q$ opposite to $\gamma$.
By the choice of $t_1,t_2$ and by the second part of the lemma,
applied to the triangle with vertices $\gamma(0),\gamma(t_1),x_2$
and an angle at least $\alpha$ at $\gamma(0)$ and the triangle
with vertices $\gamma(t_2),\gamma(\tau),x_1$ and an angle
at least $\alpha$ at $\gamma(\tau)$, 
the angle of $Q^\prime$ at the 
vertices $\gamma(t_1),\gamma(t_2)$ is at least $3\pi/4$. Therefore
we conclude from 
the first part of the lemma for $\alpha=3\pi/4$ 
that $\zeta$ passes through
the $b(\epsilon)$-neighborhood of 
$\gamma[t_1,t_2]\cap {\cal T}(S)_\epsilon$. This is what
we wanted to show.

To establish the angle estimate in the second part of the lemma, 
let $\alpha >0,\theta >0$ and let $T$ be a triangle in 
$({\cal T}(S),d_{WP})$ with a side $\zeta:[0,\sigma]\to {\cal T}(S)$
of length $\sigma>0$,  
an angle $\theta_0\geq \theta$ at $\zeta(0)$ and an angle
$\alpha_0\geq \alpha$ at $\zeta(\sigma)$. 
Let $x_1$ be the vertex of $T$ opposite to the side $\zeta$ and
assume that $T$ is ruled by WP-geodesics connecting $x_1$ to 
the points on $\zeta$. The Gauss curvature of $T$ is negative, in 
particular we have $\alpha_0+\theta_0<\pi$.

Let $\hat T$ be a comparison triangle 
in the euclidean plane $\mathbb{R}^2$ with
a side $\hat \zeta:[0,\sigma]\to \mathbb{R}^2$ of length
$\sigma$, an angle $\theta$ at $\hat \zeta(0)$ and an angle  
$\alpha$ at $\hat \zeta(\sigma)$. 
Note that
the angles of $\hat T$ at $\hat \zeta(0),\hat \zeta(\sigma)$ may
both be smaller than the corresponding angles of $T$.
Since $\theta_0\geq \theta$, 
by ${\rm CAT}(0)$-comparison (see \cite{BH99}), 
for every $s\in [0,\sigma]$ the length of the intrinsic geodesic 
$\beta_s$ in the ruled triangle $T$ which is orthogonal to 
$\zeta$ at $\zeta(s)$ and which 
ends on one of the two sides different
from $\zeta$ is not smaller than the length of the geodesic
$\hat \beta_s$ in $\hat T$ which is orthogonal to 
$\hat \zeta$ at $\hat \zeta(s)$ and ends on one of the
sides of $\hat T$ distinct from $\hat \zeta$.

As the length $\sigma$ of $\hat \zeta$ tends to infinity, the 
distance between $\hat \zeta$ and the vertex of $\hat T$ not contained
on $\hat \sigma$ tends to infinity as well. Thus
there is 
a number $t(\alpha,\theta)>0$ only depending on $\alpha$ and $\theta$ 
such that for $s\in [t(\alpha,\theta),\sigma-t(\alpha,\theta)]$ the length of 
the geodesic arc $\beta_s$ in $T$ is 
not smaller than $b(\epsilon)/2$.
In particular, by comparison, the ruled triangle
$T$ contains an embedded strip of area at least
$b(\epsilon)(\sigma-2t(\alpha,\theta))/2$ which is the union
of the initial subsegments of length $b(\epsilon)/2$
of the geodesic arcs $\beta_s$ issuing from points
$\zeta(s)$ where $s\in [t(\alpha,\theta),\sigma-t(\alpha,\theta)]$.

Now by the argument in the beginning of this
proof, if $\rho=\ell_{\epsilon-{\rm thick}}(\zeta)$ then
the integral of the Gauss curvature over this strip is
at most $-\kappa(\epsilon) b(\epsilon)(\rho-2t(\alpha,\theta))/2$, and hence the
sum of the angles of the triangle $T$ is at most
$\pi-\kappa(\epsilon) b(\epsilon)(\rho-2t(\alpha,\theta))/2$. On the other hand,
the angle sum of $T$ is at least $\alpha+\theta$ by assumption which
implies that 
\[\rho\leq 2(\pi-\alpha-\theta)/\kappa(\epsilon) b(\epsilon) +2t(\alpha,\theta).\]

As a consequence, if $T$ is a WP-geodesic triangle in ${\cal T}(S)$ 
with a side $\zeta:[0,\sigma]\to {\cal T}(S)$ such that 
$\ell_{\epsilon-{\rm thick}}(\zeta)\geq
2(\pi-\alpha-\theta)/\kappa(\epsilon) b(\epsilon)+2t(\alpha,\theta)$ and if the
angle of $T$ at $\zeta(\sigma)$ is not smaller than
$\alpha$, then the angle of $T$ at $\zeta(0)$ does not exceed $\theta$.
This is just the statement in the second part of the lemma.
\end{proof}

As a consequence of Lemma \ref{gaussb}, we obtain a sufficient condition
for the existence of a biinfinite WP-geodesic which is forward and backward
asymptotic to WP-rays with specific geometric properties which does not
use require these rays to be recurrent.

\begin{corollary}\label{asympt1}
For $\epsilon >0$ 
and $\alpha>0$ let
$k_1=k_1(\epsilon,\alpha)>0$ be as in Lemma \ref{gaussb}.
Let $\gamma_0:[0,\tau]\to {\cal T}(S)$ be a WP-geodesic
segment with $\ell_{\epsilon-{\rm thick}}
(\gamma_0)\geq k_1$ and 
let $\gamma_1,\gamma_2:[0,\infty)\to 
{\cal T}(S)$ be infinite WP-rays issuing from $\gamma_1(0)=\gamma_0(0),
\gamma_2(0)=\gamma_0(\tau)$. Assume that the angle at
$\gamma_0(0),\gamma_0(\tau)$ between the unit tangent
vectors $\gamma_0^\prime(0),\gamma_1^\prime(0)$ and between 
the unit tangent vectors $-\gamma_0^\prime(\tau),\gamma_2^\prime(0)$ is 
at least $\alpha$.
Assume furthermore that there are measured geodesic laminations 
$\mu_1,\mu_2$ which fill $S$ and such that 
the length of $\mu_i$ is bounded along $\gamma_i$ (i=1,2). 
Then there is a unique biinfinite WP-geodesic $\xi$ which is forward 
asymptotic to $\gamma_1$ and backward asymptotic to $\gamma_2$.
\end{corollary}
\begin{proof}
Let $\gamma_0,\gamma_1,\gamma_2$ be as in the lemma.
For $t>0$ consider the geodesic quadrangle $Q_t$ with 
vertices $\gamma_0(0),\gamma_0(\tau),\gamma_2(t),\gamma_1(t)$.
By Lemma \ref{gaussb}, the side $\xi_t$ of $Q_t$ 
which connects $\gamma_1(t)$ to $\gamma_2(t)$ 
passes through a fixed compact
neighborhood $K$ of $\gamma_0[0,\tau]\cap {\cal T}(S)_\epsilon$.
We parametrize $\xi_t$ by arc length in such a way that 
$\xi_t(0)\in K$. Then the WP-distance between $\gamma_0(0)$ and 
$\xi_t(0)$ is uniformly bounded, independent of $t$.

By the ${\rm CAT}(0)$-triangle comparison property, the
Hausdorff distance with respect to 
the Weil-Petersson metric between the geodesic arc $\gamma_1[0,t]$ and
the subsegment of $\xi_t$ connecting $\xi_t(0)$ to 
$\gamma_1(t)$ is uniformly bounded, independent of $t$. Similarly, 
the Hausdorff distance with respect to
the Weil-Petersson metric between the geodesic arc $\gamma_2[0,t]$ and 
the subsegment of $\xi_t$ connecting $\xi_t(0)$ to $\gamma_2(t)$ is 
uniformly bounded.

The restriction to $K$ of the 
unit tangent bundle for the Weil-Petersson metric is compact.
Therefore  there is a sequence $t_i\to\infty$ such that
the directions $\xi_{t_i}^\prime(0)$ 
of $\xi_{t_i}$ at $\xi_{t_i}(0)$ converge as $i\to \infty$ 
to a direction $v$. Let $\xi:(-r,T)\to {\cal T}(S)$ be the
(maximal) WP-geodesic with initial velocity $v$. 
The WP-geodesics $\xi_{t_i}$ converge
uniformly on compact subsets of $(-r,T)$ to $\xi$.  
In particular, 
if $\xi$ is biinfinite (i.e. if $r=T=\infty$) 
then $\xi$ is indeed a geodesic which is 
forward
asymptotic to $\gamma_1$ and backward asymptotic to $\gamma_2$.

To see that $\xi$ is indeed biinfinite, let $\mu_1\in {\cal M\cal L}$ be
a measured geodesic lamination which fills up $S$ and whose
length is bounded along the geodesic ray $\gamma_1$. Such a measured
geodesic lamination exists by assumption.
Since $\xi_t(0)\in K$ for all $t$, the $\xi_t(0)$-length of
$\mu_1$ is bounded independent of $t>0$. Now for every $t$ 
the forward endpoint $\xi_t(\tau_t)$ of
$\xi_t$ equals $\gamma_1(t)$ and hence by convexity of length
functions along WP-geodesics, the length of $\mu_1$ is bounded 
along $\xi_t[0,\tau_t]$ by a universal constant,
independent of $t$. Note also that $\tau_t\to \infty$ 
$(t\to \infty)$. By continuity of the
length pairing, we conclude that the length of $\mu_1$ on $\xi[0,T)$ is 
uniformly bounded (compare the proof of Lemma \ref{thickgeo} and
Proposition \ref{wpcomp} for a more detailed argument). 
However, if $T<\infty$ then there is a simple
closed curve $c$ on $S$ so that the $\xi(t)$-length of $c$ tends
to zero as $t\to T$ (see \cite{W03} for more and for references).
Since $\mu_1$ fills $S$ we have
$i(c,\mu_1)>0$ and therefore 
the length of $\mu_1$ along $\xi[0,T)$ tends to infinity as $t\to T$
which is a contradiction. This argument also applies to the
ray $\xi(-r,0]$ and yields
that $\xi$ is indeed a biinfinite 
WP-geodesic which is forward asymptotic to $\gamma_1$ and 
backward asymptotic to $\gamma_2$.

To show that such a geodesic is unique, recall from 
comparison for ${\rm CAT}(0)$-spaces that if there is 
a second such geodesic $\xi^\prime$ then $\xi$ and $\xi^\prime$ bound
a flat strip. Since the sectional curvature of the Weil-Petersson metric
is negative, this is impossible. 
The corollary follows.
\end{proof}

\section{Short curves and twisting}

In Corollary \ref{asympt1} we established a sufficient condition for
the existence of a biinfinite Weil-Petersson geodesic
which is forward and backward asymptotic to two 
given Weil-Petersson geodesic rays. To apply this result, 
we have to find a sufficient condition for
a Weil-Petersson geodesic to spend a definitive 
amount of time in the thick part of Teichm\"uller space.

We approach this problem by analyzing WP-geodesic segments
of uniformly bounded length 
which enter deeply into the thin part of Teichm\"uller 
space. Our goal is 
a quantitative version of the
following result of Wolpert \cite{W03}: 
If $\gamma:[0,r]\to {\cal T}(S)$ is any Weil-Petersson
geodesic of uniformly bounded length 
and if $\gamma$ enters the thin part of Teichm\"uller space then 
$\gamma$ twists a definitive amount about 
one of the curves
which becomes very short along $\gamma$.  
Or, put differently, if $\gamma$ does not twist much about
its short curves then $\gamma$ can
not be entirely contained in the thin part of
Teichm\"uller space. 

We continue to use
the assumptions and notations from Sections 2-3.
Following Masur and Minksy, we measure the
amount of twisting about a curve $\alpha\in {\cal C}(S)$
as follows (see in particular p. 919 of 
\cite{MM00}).
Fix a complete hyperbolic metric of finite volume 
$h\in {\cal T}(S)$ and identify $\alpha$ with the $h$-geodesic
it defines. 
There is a locally isometric 
annular cover $\tilde A\to S$ of $(S,h)$ whose
geodesic core curve $\tilde \alpha$ 
projects isometrically onto $\alpha$.  
The hyperbolic annulus $\tilde A$ admits
a natural compactification
to a closed annulus $\hat A$ which is obtained as follows.
The fundamental group $<\alpha>$ of $\tilde A$
acts on the hyperbolic plane ${\bf H}^2$ as a group
of hyperbolic isometries fixing two points $a\not=b$ in the
ideal boundary $\partial {\bf H}^2$ of 
${\bf H}^2$. The quotient of ${\bf H}^2\cup
(\partial {\bf H}^2-\{a,b\})$ under the action of $<\alpha>$
is a compact annulus $\hat A$ containing $\tilde A$ as an open
dense subset. 
Any geodesic in $\tilde A$ for the hyperbolic metric 
which intersects the geodesic core curve $\tilde \alpha$  
transversely extends to a continuous path in 
$\hat A$ connecting the two distinct boundary components of $\hat A$.

Let 
${\cal C}(\alpha)$ be the set of all simple paths in $\hat A$ connecting 
the two distinct boundary components 
of $\hat A$ 
modulo homotopies that fix the endpoints. Then ${\cal C}(\alpha)$ 
is the set of vertices of a metric graph ${\cal C\cal G}(\alpha)$ 
whose edges are determined by requiring that
two such homotopy classes of arcs $\gamma,\gamma^\prime$ 
are connected by an edge of length one if and only
if $\gamma,\gamma^\prime$ 
have representatives with disjoint interior.

Following p. 920 of \cite{MM00}, define a projection
$\pi_\alpha$ of ${\cal C}(S)$ into the 
family of all subsets of ${\cal C\cal G}(\alpha)$ 
as follows. 
Let $\gamma\in {\cal C}(S)$ be represented by a simple closed $h$-geodesic.
If $\gamma$ does not intersect $\alpha$ transversely then we define
$\pi_\alpha(\gamma)=\emptyset$. Otherwise  
the preimage $\tilde \gamma$ of $\gamma$ in $\tilde A$ has at least
one component which extends continuously 
to an arc connecting the two distinct 
boundary components of $\hat A$. The set $\pi_{\alpha}(\gamma)$ of all
these components is a finite set of diameter at most 1 in 
${\cal C\cal G}(\alpha)$. This definition of $\pi_\alpha$ is 
essentially independent of the choice of the hyperbolic metric $h$
(see \cite{MM00} for more details). If $c$ is a 
\emph{simple multi-curve}, i.e. a disjoint union of mutually
not freely homotopic simple closed curves, then 
let $\pi_{\alpha}(c)$ be the union of the projections of its
components. As before, the diameter of $\pi_{\alpha}(c)$ is 
at most one (Lemma 2.3 of \cite{MM00}).
The projection $\pi_\alpha$ can be used to measure
the relative twisting about $\alpha$ of two simple closed curves
which intersect $\alpha$ transversely \cite{MM00}.

As in Section 2, let $\chi_0>0$ be a Bers constant for 
$S$ and let $\Upsilon_{\cal T}:{\cal T}(S)\to {\cal C}(S)$
be a map which associates to a hyperbolic metric 
$h$ a simple closed curve of $h$-length at most $\chi_0$.
We use the projections $\pi_\alpha$ $(\alpha\in {\cal C}(S))$ and 
Wolpert's description
of the Weil-Petersson metric near its completion locus 
(see \cite{W03,W08}) to obtain 
information on the image of a 
Weil-Petersson geodesic $\gamma$ under the map $\Upsilon_{\cal T}$.
We are in particular interested in 
the twisting 
behavior of points in $\Upsilon_{\cal T}(\gamma)$ 
about a simple closed curve which becomes
very short along $\gamma$.

For this we first need a better
control of the projections of WP-geodesics
into the graph ${\cal C\cal G}(\alpha)$ for a simple closed
curve $\alpha\in {\cal C}(S)$. We obtain such a control
using the \emph{pants graph} ${\cal P\cal G}(S)$ for $S$.

A pants decomposition $P$ for $S$ is 
changed to a pants decomposition $P^\prime$ by 
an \emph{elementary move} if $P^\prime$ is
obtained from $P$ by replacing one of the 
pants curves $\alpha$ of 
$P$ by a curve which does not intersect $P-\alpha$ and
intersects $\alpha$ in the minimal number of points 
(i.e. in precisely two points if the component of $S-(P-\alpha)$
containing $\alpha$ is a four-holed sphere, and in precisely
one point if this component is a one-holed torus).
The pants graph ${\cal P\cal G}(S)$ 
of $S$ is the geodesic metric graph whose set of 
vertices is the set ${\cal P}(S)$ of pants decompositions
for $S$ and where two such
pants decompositions $P,P^\prime$ are connected 
by an edge of length one if and only if $P^\prime$ can
be obtained from $P$ by an elementary move.

As in Section 2, 
call a pants decomposition $P$ for $S$ a
Bers decomposition for $h\in {\cal T}(S)$ if the
$h$-length of each of the components of $P$ is bounded from
above by $\chi_0$. Note that if $\alpha\in {\cal C}(S)$ is any simple
closed curve whose $h$-length is bigger than $\chi_0$ then every
Bers decomposition $P$ for $h$ contains a component which
intersects $\alpha$ transversely. In particular, the
projection $\pi_{\alpha}(P)$ is not empty, and
its diameter ${\rm diam}(\pi_\alpha(P))$ is at most one.

Define a map 
\[\Upsilon_{\cal P}:{\cal T}(S)\to {\cal P\cal G}(S)\]
by associating to a hyperbolic metric $x\in {\cal T}(S)$ a
Bers decomposition for $x$. The following Lemma is a quantitative
version of Lemma \ref{bounded} for pants decompositions 
which enables us to control for a simple closed curve $\alpha$ 
the projections into the graph ${\cal C\cal G}(\alpha)$ 
of Bers decompositions along Weil-Petersson geodesics. Compare also
Section 3 of \cite{B03}.

\begin{lemma}\label{pants}
There is a constant $\chi_1>\chi_0$ with the following property.
Let $\xi:[0,\sigma]\to {\cal T}(S)$ be a
WP-geodesic segment of length $\sigma\leq 1$ 
and let $P_0,P_\sigma$ be Bers decompositions
for $\xi(0),\xi(\sigma)$. 
Then $P_0$ can be connected to $P_\sigma$ by an edge path
$\rho$ in ${\cal P\cal G}(S)$ of length at most $\chi_1$ 
with the following property. For every vertex $P\in {\cal P}(S)$
passed through by $\rho$ there is some $t\in [0,\sigma]$ such that
the $\xi(t)$-length of every component curve of $P$ is 
smaller than $\chi_1$.
\end{lemma}
\begin{proof} 
By Lemma 3.12 of \cite{W08}, there is a number $a>0$ 
with the following property. Let $h\in {\cal T}(S)$ and let
$\alpha\in {\cal C}(S)$ be a simple closed curve
of $h$-length $\ell_h(\alpha)\leq 2\chi_0$. Then the 
norm at $h$ of the Weil-Petersson gradient of the length function
$x\to \ell_x(\alpha)$ is bounded from above by $a$.
This implies the following. 
Let $\tau\leq \chi_0/a$, let $\xi:[0,\tau]\to {\cal T}(S)$ be any
WP-geodesic and let $P$ be a Bers decomposition for $\xi(0)$. 
Then for every $t\in [0,\tau]$ the 
$\xi(t)$-length of every component of $P$ does not exceed $2\chi_0$.

Choose an integer $\ell>a/\chi_0$ and note that $\ell$ is a universal
constant. 
Let $\xi:[0,\sigma]\to {\cal T}(S)$ be a WP-geodesic
of length $\sigma\leq 1$ and let
$j\leq \ell$ be the smallest integer such that
$j/\ell\geq \sigma$. Let $P_0,P_\sigma$ be Bers
decompositions for $\xi(0),\xi(\sigma)$ and for 
each $i\in \{1,\dots, j-1\}$ 
let $P_i$ be a Bers decomposition for $\xi(\frac{i}{\ell})$. 
By the choice of the constant $\ell$, for each
$i$ the $\xi(\frac{i+1}{\ell})$-length of each component of 
$P_i$ does not exceed $2\chi_0$. Since $j\leq \ell$,
for the proof of the lemma
it is enough to show the existence of  
a number $\beta>0$ with the following property. 
For every $x\in {\cal T}(S)$,
any pants decomposition $Q_0$ of $S$ with components
of $x$-length at most $2\chi_0$ can be connected to a given Bers
decomposition $Q_1$ for $x$ by a path in ${\cal P\cal G}(S)$
of length at most $\beta$
passing through vertices of the pants graph 
which are pants decompositions
for $S$ with components of $x$-length at most $\beta$.

Thus let $x\in {\cal T}(S)$ and let 
$Q_0,Q_1$ be such pants decompositions whose components
are $x$-geodesics of length at most $2\chi_0$ and $\chi_0$, respectively.
Let $c_1,\dots,c_k$ $(0\leq k\leq 3g-3+m)$ be those components
of $Q_0$ which are also components of $Q_1$. Note that  
by the collar lemma, the
set $c_1,\dots,c_k$ contains every simple closed
$x$-geodesic of sufficiently small length. 
Let $\hat S$ be the metric completion of the 
(perhaps disconnected) surface which we obtain by cutting $S$ open
along the geodesics $c_1,\dots,c_k$.
Choose a component $\hat S_0$ of $\hat S$ 
which is different from a three-holed sphere.

The intersection $Q_0\cap \hat S_0$ of $Q_0$ with the
interior of $\hat S_0$ is a pants decomposition for $\hat S_0$, and the
same is true for the intersection $Q_1\cap \hat S_0$ of $Q_1$ with
the interior of $\hat S_0$. Thus 
$(Q_0\cup Q_1)\cap \hat S_0$ 
is an embedded connected piecewise geodesic graph 
$G$ in $\hat S_0$ which
decomposes $\hat S_0$ into non-essential
annuli, i.e. annuli whose core curves are 
homotopic to a boundary component
of $\hat S_0$ or to a puncture, and into topological discs
with piecewise geodesic boundary. In particular, the injection
of the graph $G$ into $\hat S_0$ induces a surjection of
fundamental groups. 
By the collar lemma and the
length bound for the components of $Q_0,Q_1$, 
the number of intersections between $Q_0$ and $Q_1$ is bounded
from above
by a number only depending on $\chi_0$ and
the topological type of $S$. This number of intersections is the 
number of vertices of the graph $G$. Now the valency of 
a vertex of $G$ equals four and therefore  
the number of 
edges of $G$ is uniformly bounded as well. 
The length of each edge
does not exceed $2\chi_0$.

Up to the action of the mapping class group ${\rm Mod}(\hat S_0)$ 
of $\hat S_0$, there
are only finitely many pairs of pants decompositions
of $\hat S_0$ whose union is an embedded connected
graph in $\hat S_0$ 
with a uniformly bounded number of edges and whose complementary
components are topological discs and
non-essential annuli. By invariance under the action of 
${\rm Mod}(\hat S_0)$, this means that there is a number
$p=p(\hat S_0)>0$ only depending on the topological type of $\hat S_0$,
there is a number $n\leq p$ and  
there is a sequence of pants decompositions 
$R_0=Q_0\cap \hat S_0,R_1,\dots,R_n=Q_1\cap \hat S_0$
for $\hat S_0$ with the following
properties. 
For each $i<n$, $R_{i+1}$ can
be obtained from $R_i$ by an elementary move. Moreover,  
each simple closed curve on $\hat S_0$ 
appearing  as a pants curve of one of the pants decompositions $R_i$  
is freely homotopic to an edge path in the graph  
$(Q_0\cup Q_1)\cap \hat S_0$ of uniformly 
bounded combinatorial length. Since the $x$-length of 
an edge of $(Q_0\cup Q_1)\cap \hat S_0$ is uniformly bounded, the  
$x$-length of each component curve of $R_i$
is bounded from above by a constant which only depends on $\hat S_0$. 

On the other hand, there are
only finitely many topological types of subsurfaces of $S$
which can arise as complementary components of a 
simple multi-curve on $S$. Thus a successive application of this
construction to all components of $\hat S$ shows
that $Q_0$ can be modified to $Q_1$ in a uniformly bounded
number of steps consisting of pants decompositions whose
components have uniformly bounded $x$-length. The lemma is proven.
\end{proof}

Lemma \ref{pants} together with the results of \cite{MM00} imply
the following projection-diameter control.

\begin{corollary}\label{pantslength}
Let $\xi:[0,R]\to {\cal T}(S)$ be any Weil-Petersson
geodesic. Let $\alpha\in {\cal C}(S)$ be a simple 
closed curve with $\ell_{\xi(t)}(\alpha)\geq \chi_1$
for every $t\in [0,R]$, where $\chi_1>0$ is as in Lemma \ref{pants}. 
Let $P,Q$ be Bers decompositions
for $\xi(0),\xi(R)$. 
Then 
\[{\rm diam}(\pi_\alpha(P)\cup \pi_\alpha(Q))\leq 4\chi_1R+8\chi_1.\]
\end{corollary}
\begin{proof}
Let $\xi:[0,R]\to {\cal T}(S)$ be a Weil-Petersson geodesic
and let $P_0,P_1$ be Bers decompositions for $\xi(0),\xi(R)$. 
Let moreover $\alpha\in {\cal C}(S)$ be a simple closed
curve with $\ell_{\xi(t)}(\alpha)\geq \chi_1$ for every
$t\in [0,R]$. By Lemma \ref{pants},
$P_0$ can be connected to $P_1$ by a path in ${\cal P\cal G}(S)$ of 
length $n\leq R\chi_1+\chi_1$ which passes through 
vertices $Q_0=P_0,\dots, Q_n=P_1$ 
of the pants graph defined by pants decompositions
with component curves of $\xi(t)$-length smaller than $\chi_1$
for some $t\in [0,R]$. Since $\ell_{\xi(t)}(\alpha)\geq \chi_1$ for all $t$, 
each of the pants decompositions $Q_i$ 
intersects $\alpha$ transversely.

By Lemma 2.3 of \cite{MM00}, 
for every pants decomposition
$P$ of $S$ with an essential intersection
with $\alpha$, the projection $\pi_\alpha(P)$ of
$P$ to ${\cal C\cal G}(\alpha)$ is non-empty and of diameter
at most $2$. If for some $i<n$ the pants decomposition  
$Q_{i+1}$ is obtained
from $Q_i$ by an elementary move preserving at least one
curve which intersects $\alpha$ transversely, then 
the projection of this curve to ${\cal C\cal G}(\alpha)$ 
is contained in $\pi_{\alpha}(Q_i)\cap 
\pi_{\alpha}(Q_{i+1})$ and therefore the diameter
of $\pi_{\alpha}(Q_i)\cup \pi_{\alpha}(Q_{i+1})$ is at most $4$.
Otherwise the elementary
move which transforms $Q_i$ to $Q_{i+1}$ 
exchanges two simple closed curves $\beta_i,\beta_{i+1}$, and the 
connected component $Y$ of $S-(Q_i-\beta_i)$
distinct from a pair of pants
contains $\alpha$. Now $Y$ is a four-holed sphere 
or a one-holed torus which is bounded
by simple closed curves in $Q_i\cap Q_{i+1}$, and $\alpha$
intersects both $\beta_i$ and $\beta_{i+1}$ transversely. 
However, in this case the diameter
of $\pi_{\alpha}(Q_i)\cup \pi_{\alpha}(Q_{i+1})
=\pi_\alpha(\beta_i)\cup \pi_\alpha(\beta_{i+1})$ 
is also bounded
from above by 4 by another application of 
Lemma 2.3 of \cite{MM00}. 

As a consequence and by induction,
the diameter in ${\cal C\cal G}(\alpha)$ of the projection
$\pi_{\alpha}(P_0)\cup \pi_{\alpha}(P_1)$ 
is at most $4(n+1)\leq 4\chi_1(R+1)+4\leq 4\chi_1 (R+1)+8\chi_1$
(note that $\chi_1>1$ by construction).
The corollary follows.
\end{proof}

Let again $\chi_0>0$ be a Bers constant for $S$ and let
$\chi_1>\chi_0$ be as in Lemma \ref{pants}. By the collar
lemma and the fact that the distance in ${\cal C\cal G}(S)$
between any two simple closed  curves
$\alpha,\beta\in {\cal C}(S)$ does not exceed $i(\alpha,\beta)+1$ 
(Lemma 2.1 of \cite{MM99}),
there is a number $p=p(\chi_0,\chi_1)>0$ with the following property.
Let $x\in {\cal T}(S)$ and let $\alpha$ be a simple closed
curve on $S$ whose $x$-length is at most $\chi_0$. If
$d_{\cal C}(\alpha,\beta)\geq p-1$ then the 
$x$-length of $\beta$ is bigger than $\chi_1$.

We use Corollary \ref{pantslength} and the results of Wolpert \cite{W03} to 
control Weil-Petersson geodesics in the thin part of 
Teichm\"uller space.

\begin{proposition}\label{finitecontrol}
For every $R>1,c >0$ there is a 
number $\epsilon=\epsilon(R,c)>0$ with 
the following property.
Let $\zeta:[0,\sigma]\to {\cal T}(S)$
be a WP-geodesic segment of length $\sigma\leq R$ and let 
$P_0,P_\sigma$ be Bers decompositions
for $\zeta(0),\zeta(\sigma)$. 
Let $\alpha\in {\cal C}(S)$ be a curve whose 
distance in ${\cal C\cal G}(S)$ to 
any component of $P_0,P_\sigma$
is at least $p$. Assume that 
${\rm diam}(\pi_{\beta}(P_0)\cup\pi_{\beta}(P_\sigma))
\leq c$ for every simple closed curve
$\beta\in {\cal C}(S)$ with $d_{\cal C}(\alpha,\beta)\leq 1$.
Then $\ell_{\zeta(t)}(\alpha)\geq \epsilon$ for all $t\in [0,\sigma]$.
\end{proposition}
\begin{proof}
The proof relies on Wolpert's description of Weil-Petersson geodesics 
near the completion locus of Teichm\"uller space as 
explained in \cite{W03}.

We argue by contradiction and we assume that 
the statement of the proposition does not hold. Then there
are numbers $R>1,c>0$ and
there is a sequence $\epsilon_i\to 0$,
a sequence of WP-geodesics $\zeta_i:[0,\sigma_i]\to
{\cal T}(S)$, a sequence of numbers $t_i\in (0,\sigma_i)$ and 
a sequence of simple closed curves $\alpha_{i}\in {\cal C}(S)$ 
such that for every $i$ the following holds true.
\begin{enumerate}
\item[(a)] $\sigma_i\leq R$.
\item[(b)] The distance in ${\cal C\cal G}(S)$ between $\alpha_i$ and  
any component of some Bers decomposition $P_i,P_{\sigma_i}$ 
of $\zeta_i(0),\zeta_i(\sigma_i)$ is at least $p$.
\item[(c)] 
${\rm diam}(\pi_{\beta}(P_i)\cup \pi_{\beta}(P_{\sigma_i}))
\leq c$ for every simple closed curve $\beta\in {\cal C}(S)$ 
with $d_{\cal C}(\alpha_i,\beta)\leq 1$.
\item[(d)] There is some $t_i\in (0,\sigma_i)$
such that $\ell_{\zeta_i(t_i)}(\alpha_i)\leq \epsilon_i$.
\end{enumerate}

Our strategy is to analyze a sequence of geodesics
$\zeta_i:[0,\sigma_i]\to {\cal T}(S)$ which has the 
properties (a), (b) and (d) above.
By the choice of the constants $p>0$ and $\chi_1>0$, for 
each $i$ we have
\begin{equation}\label{length}
\ell_{\zeta_i(0)}(\alpha_i)\geq  \chi_1,
\ell_{\zeta_i(\sigma_i)}(\alpha_i)\geq  \chi_1.
\end{equation} 
A result of Wolpert \cite{W03} gives some geometric
information on the sequence. This information allows us to 
formulate an additional condition on the sequence.
We then consider sequences which satisfy this additional
condition and  use Corollary \ref{pantslength} to show  
that for such a sequence, property (c) above is violated.
The general case is then reduced to the special case, 
applied to subarcs of the sequence $\zeta_i$

Let now $\zeta_i:[0,\sigma_i]\to {\cal T}(S)$ be a sequence
with properties (a),(b),(d). 
By Wolpert's gradient estimates for length functions 
(see Lemma 3.12 of \cite{W08}), for every $\beta\in {\cal C}(S)$
the norm of the Weil-Petersson gradient of the length function
$x\to \ell_\beta(x)$ is uniformly 
bounded on $\{x\mid \ell_\beta(x)\leq \chi_1\}$.
This implies that  
the Weil-Petersson distance between
a point $x\in {\cal T}(S)$ with $\ell_{\alpha_i}(x)\geq \chi_1$ and
a point $y\in {\cal T}(S)$ with $\ell_{\alpha_i}(x)\leq \chi_0/2$
is bounded from below by a universal constant $a>0$.
In particular, we have $\sigma_i\geq 2a$ for all 
$i$ which are sufficiently large that $\epsilon_i\leq \chi_0/2$ and
hence by passing to a subsequence
we may assume that $\sigma_i\to \sigma\in [2a,R]$.

The mapping class group acts on 
the Weil-Petersson completion 
$\overline{{\cal T}(S)}$ of Teichm\"uller space.
The quotient $\overline{{\cal T}(S)}/{\rm Mod}(S)$  
is just 
the Deligne-Mumford compactification of moduli
space, in particular it is compact.  
Thus up to passing to another subsequence and up to 
the action of the mapping class group,
we may assume that the initial points
$\zeta_i(0)$ of the geodesics $\zeta_i$ converge to a point
$x_0\in \overline{{\cal T}(S)}$ (compare the discussion in \cite{W03}).

A point in $\overline{{\cal T}(S)}-{\cal T}(S)$ 
is a surface with nodes, where a node
is obtained by pinching a simple closed curve on $S$ to a point.
For a simple multi-curve $c$ on $S$
let ${\cal T}(c)\subset
(\overline{{\cal T}(S)}-{\cal T}(S))$ be the stratum of 
the completion locus for the Weil-Petersson metric 
which consists of all Riemann surfaces with nodes at the 
components of $c$ (i.e. Riemann surfaces obtained from $S$
by pinching the components of $c$ to punctures). 
By Proposition 23 of \cite{W03}, up to passing to another 
subsequence and up to possibly a composition with 
Dehn multi-twists about the nodes of $x_0$, 
there exists a finite partition
$0=t_0<t_1<\dots <t_k=\sigma$ of the interval $[0,\sigma]$ 
and there are simple multi-curves $c_0,c_1,\dots, c_k$ 
and points $x_j\in {\cal T}(c_j)$ such that the
following holds true. 

For $0\leq  j\leq  k-1$ consider the (possibly trivial)
multi-curve $\tau_j=
c_j\cap c_{j+1}$. If $1\leq j<k-1$ then $\tau_j$ is a 
proper subset of both $c_j$ and  
$c_{j+1}$, and $\tau_0=c_0\cap c_1$ is a
proper subset of $c_1$, $\tau_{k-1}$ is 
a proper subset of $c_{k-1}$. For each $i$ 
and each $1\leq j\leq k-1$ there is a Dehn multi-twist 
$T_{(j,i)}$ about the components of $c_j-\tau_j$ 
such that on the parameter interval $[t_j,t_{j+1}]$ 
the arcs 
\begin{equation}\label{arcon}
T_{(j,i)}\circ \dots \circ T_{(1,i)}
\zeta_i\end{equation}
converge as $i\to \infty$ 
to the geodesic arc $\xi_j$ in $\overline{{\cal T}(S)}$  
connecting $x_j$ to $x_{j+1}$ in the
sense of parametrized unit-speed curves. 
The concatenation $\xi:[0,\sigma]\to \overline{{\cal T}(S)}$ 
of the arcs $\xi_{j}$ $(0\leq j\leq k-1)$ is
the piecewise Weil-Petersson geodesic 
connecting $x_0$ to $x_k$ and passing through $x_1,\dots,x_{k-1}$
in this order.

For each $i$ 
the $\zeta_i(0)$-length and the $\zeta_i(\sigma_i)$-length of 
the curve $\alpha_i$ is bounded from below by $\chi_1$,
and 
the minimum of the length of $\alpha_i$ along the WP-geodesics
$\zeta_i$ tends to zero as $i\to \infty$. 
This implies that $k\geq 2$ and that up to passing to a subsequence,
there is a number $j_2\in \{1,\dots,k-1\}$
such that for every sufficiently large $i$ 
we have \[T_{(j_2-1,i)}\circ \cdots\circ T_{(1,i)}\alpha_i=
\alpha\in c_{j_2}-\tau_{j_2}.\]

By Proposition 23 of \cite{W03}, the
sequence of Dehn multi-twists $T_{(j_2,i)}$ 
about the components of $c_{j_2}-\tau_{j_2}$ 
is unbounded
as $i\to \infty$. In particular, up to passing to 
a subsequence there is a simple closed
curve $\beta\in c_{j_2}-\tau_{j_2}$  
with the following property. Let $T_\beta$ be the Dehn twist about
$\beta$. Then 
there is a sequence $r(i)\to \infty$ such that 
\[T_{(j_2,i)}=T_\beta^{r(i)}\circ \hat T_{(j_2,i)}\] where
$\hat T_{(j_2,i)}$ is a (possibly trivial) Dehn multi-twist
about the components of $c_{j_2}-\tau_{j_2}-\beta$.
We note for later reference that property (b) was used here
to ensured that the length of $\beta_i$ at the endpoints o f
$\zeta_i$ is at least $\chi_1$.

As $\alpha,\beta\in c_{j_2}-\tau_{j_2}$ we have
$d_{\cal C}(\alpha,\beta)\leq 1$ and hence
by invariance under the action of the mapping class group, for
sufficiently large $i$ the distance in ${\cal C\cal G}(S)$ 
between $\alpha_i$ and 
\[\beta_i=(T_{(j_2-1,i)}\circ \cdots\circ T_{(1,i)})^{-1}\beta\] 
is at most one. By property (b) above and the choice of $p$, 
this implies that
the $\zeta_i(0)$-length of $\beta_i$ and the $\zeta_i(\sigma_i)$-length
of $\beta_i$ is at least $\chi_1$. Moreover, the curve $\beta_i$ becomes
short along $\zeta_i$.

Using the notations in the previous paragraph, 
we consider now the  case 
that $\beta$ does not intersect any
of the multi-curves $c_j$ $(0\leq j\leq k)$. 
Then $\beta$ is
invariant under each of the Dehn multi-twists $T_{(j,i)}$.
Thus we have $\beta_i=\beta$ for all 
$i$ and hence 
$\ell_{\beta}(\zeta_i(0))\geq \chi_1,
\ell_{\beta)(\zeta_i(\sigma_i)})\geq \chi_1$.
Since the arcs
$T_{(j,i)}\circ \dots \circ T_{(1,i)}
\zeta_i\vert [t_j,t_{j+1}]$ converge as $i\to \infty$ to the arc
$\xi_j$, the $\zeta_i(0)$-length and the 
$\zeta_i(\sigma_i)$-length of $\beta$ is bounded from
above independent
of $i$. Moreover,
the curve $\beta$ becomes short along the geodesics $\zeta_i$.

Call a pants decomposition $P$ of $S$ a Bers
decomposition for some $x\in \overline{{\cal T}(S)}$  
if the $x$-length of any component of $P$ is at most $\chi_0$ 
(this means in particular that $P$ contains all nodes of $x$).
Length functions on ${\cal T}(S)$ extend
continuously to functions on $\overline{{\cal T}(S)}$ 
with values in $[0,\infty]$. Thus by 
the collar lemma,  
there is a neighborhood $U$ of $x_0$ in $\overline{{\cal T}(S)}$
such that the number of pants decompositions 
which are Bers decompositions for points in $U$
is finite. Namely, if the simple closed curve $\omega$ is a node of $x_0$ then 
there is a neighborhood $V$ of $x_0$ in $\overline{{\cal T}(S)}$ 
such that every Bers decomposition for any $x\in V$ contains 
$\omega$. Consequently, by
passing to another subsequence we may assume that
the Bers decompositions $P_i$ for $\zeta_i(0)$ coincide for all $i$.
We denote this common pants decomposition by $P$.

For each $i$ let as before $P_{\sigma_i}$ be a Bers decomposition
for $\zeta_i(\sigma_i)$. We claim that
\[{\rm diam}(\pi_\beta(P)\cup\pi_{\beta}(P_{\sigma_i}))\to \infty\,
(i\to \infty).\]
Namely, the Dehn twist $T_\beta$ about $\beta$ commutes with 
each of the Dehn multi-twists $T_{(j,i)}$. Write   
$\hat T_{(j,i)}=T_{(j,i)}$ for 
$j\not=j_2$ and let as before $\hat T_{(j_2,i)}$ be the (possibly trivial)
Dehn multi-twist about the components of 
$c_{j_2}-\tau_{j_2}-\beta$
such that $T_{(j_2,i)}=T_\beta^{r(i)}\circ \hat T_{(j_2,i)}$
for some $r(i)\in \mathbb{Z}$. Since 
by the above assumption $\beta$ is disjoint from each of the
multicurves $c_j$ $(j\leq k-1)$ we 
have 
\[T_{(k-1,i)}\circ \dots\circ T_{(1,i)}= T_\beta^{r(i)}\circ
 (\hat T_{(k-1,i)}\circ \dots \circ \hat T_{(1,i)})\]
for all $i$.

Now if $\omega\in {\cal C}(S)$ is any simple closed curve and if
$T$ is any Dehn multi-twist about a simple closed curve 
$\kappa\in {\cal C}(S)-\{\beta\}$ which is disjoint from $\beta$ then 
$\pi_{\beta}(\omega)=\pi_{\beta}(T\omega)$. 
As a consequence, for each $i$ 
we have
\begin{equation}\label{twistinv}
\pi_{\beta}(P_{\sigma_i})=
\pi_{\beta}(\hat T_{(k-1,i)}\circ \dots \circ
\hat T_{(1,i)}(P_{\sigma_i}))\subset  
{\cal C\cal G}(\beta).\end{equation}

Proposition 23 of \cite{W03} together with 
${\rm Mod}(S)$-invariance of the Weil-Petersson metric shows that 
the Weil-Petersson distance between the 
points 
\[(T_{(k-1,i)}\circ\cdots \circ T_{(1,i)})^{-1}(x_k)=
(\hat T_{(k-1,i)}\circ \cdots \circ \hat T_{(1,i)})^{-1}
(T_\beta^{-r(i)}(x_k))\]
and the points $\zeta_i(\sigma_i)$
converges to zero as $i\to \infty$.
Using once more invariance of 
the Weil-Petersson metric under the mapping
class group, we conclude that 
\begin{equation}
d_{WP}(\hat T_{(k-1,i)}\circ \cdots \circ \hat T_{(1,i)}(\zeta_i(\sigma_i)),
T_\beta^{-r(i)}(x_k))\to 0\, (i\to \infty).\notag\end{equation}
Now there are only finitely many Bers decompositions for all points
in a small neighborhood of $x_k$ and therefore
by invariance under the action of the mapping class group
and by (\ref{twistinv}) above, 
up to passing to a subsequence there is a 
Bers decomposition $\hat P$ for $x_k$ such that 
\[\pi_\beta(P_{\sigma_i})=\pi_{\beta}(T_\beta^{-r(i)}\hat P)\] for 
all sufficiently large $i$. 
By property (2) for $\beta$,  the pants decomposition 
$P_{\sigma_i}$ has an essential intersection with $\beta$ and 
hence the same holds true for $\hat P$.

Since the
diameter in ${\cal C\cal G}(\beta)$ of the projection
$\pi_{\beta}(P)\cup \pi_{\beta}(\hat P)$ 
is finite and since
$\vert r(i)\vert \to \infty (i\to \infty)$, 
this implies that
the diameters of the projections
$\pi_{\beta}(P)\cup \pi_{\beta}(P_{\sigma_i})$ tend to infinity
with $i$. But $P_i=P$ for all $i$ and $d_{\cal C}(\beta,\alpha_i)\leq 1$
and hence this contradicts the 
assumption (c) above.

For the proof of the
proposition we are now left with the case
that there is some $j\in \{0,\dots,k\}$ such that
$\beta$ intersects the multi-curve $c_j$. 
Let $j_1\in \{0,\dots,j_2-1\}$ be the maximal
number $j\leq j_2-1$ such that $i(c_j,\beta)\not=0$.
If there is no such $j$ then write $j_1=-1$. Similarly,
let $j_3\in \{j_2+1,\dots,k\}$ be the mimimal number 
$j\geq j_2+1$ such that $i(c_j,\beta)\not=0$. If there is
no such number then write $j_2=k+1$. 
By our assumption, we either have $j_1\geq 0$ or $j_3\leq k$.
For $j_1<j<j_3$, the curve 
$\beta$ is invariant under the Dehn multi-twist $T_{(j,i)}$. 
Define 
\[\rho_i=T_{(j_1,i)}\circ \dots \circ T_{(1,i)}\zeta_i.\]
If $j_1>-1$ then we have 
$\ell_{\beta}(\rho_i(t_{j_1}))\to \infty$ $(i\to \infty)$,
$\ell_{\beta}(\rho_i(t_{j_1+1})) \to 0$ $(i\to \infty)$
and therefore by convexity of length functions along
Weil-Petersson geodesics, 
for sufficiently large $i$ there is a unique number
$s_i\in (t_{j_1},t_{j_2})$ such that
$\ell_{\beta}(\rho_i(s_i))=\chi_1$.
If $j_1=-1$ then using once more convexity and the
assumption that the $\zeta_i(0)$-length
of $\beta_i$ is not smaller than $\chi_1$ we also
find a unique number $s_i\in (t_{j_1},t_{j_2})$ such that
$\ell_{\beta}(\rho_i(s_i))=\chi_1$.
Similarly we find for all sufficiently large $i$ a unique number
$u_i\in (t_{j_2},t_{j_3})$ such that $\ell_{\beta}(\rho_i(u_i))=\chi_1$.
By passing to a subsequence, we may assume that
$s_i\to s,u_i\to u$. Using again Wolpert's gradient bound
for length functions, we have $s\in (t_{j_1},t_{j_2})$,
$u\in (t_{j_2},t_{j_3})$ and $s<u$.

The Weil-Petersson geodesics 
$\rho_i[s_i,u_i]$, the 
Bers decompositions $Q_i^0$ for $\rho_i(s_i)$,
$Q_i^1$ for $\rho_i(u_i)$ and the curve $\beta$
have all properties required to apply the special case
of the proposition established in the first part of this proof.  
The first part of this proof implies that 
the diameters of the projections
\[\pi_\beta(Q_i^0)\cup \pi_\beta(Q_i^1)\]
tend to infinity with $i$. By equivariance
under the mapping class group, if 
\[\beta_i=(T_{(j_1,i)}\circ \dots\circ T_{(1,i)})^{-1}\beta\]
and if 
$\hat Q_i^0,\hat Q_i^1$ are
Bers decompositions for $\zeta_i(s_i),\zeta_i(u_i)$, then 
\begin{equation}\label{tendtoinfty}
{\rm diam}(\pi_{\beta_i}(\hat Q_i^0)\cup \pi_{\beta_i}(\hat Q_i^1))\to \infty
\end{equation}
as well.

By equivariance of length functions under the action of 
${\rm Mod}(S)$, 
we have 
\[\ell_{\beta}(\zeta_i(s_i))=\chi_1,
\ell_{\beta_i}(\zeta_i(u_i))=\chi_i,\] moreover the length of 
$\beta_i$ become arbitrarily small along 
$\zeta_i[s_i,u_i]$ as $i\to \infty$. Thus 
by convexity, the length
of $\beta_i$ is at least $\chi_1$ 
on $[0,s_i]\cup [u_i,\sigma_i]$.
Corollary \ref{pantslength},
applied to the Weil-Petersson geodesics
$\zeta_i[0,s_i]$ and $\zeta_i[u_i,\sigma_i]$ of length 
at most $R$ and the curves $\beta_i\in {\cal C}(S)$, yields 
that 
\[\max\{{\rm diam}(\pi_{\beta_i}(P_i)\cup \pi_{\beta_i}(\hat Q_i^0)),
{\rm diam}(\pi_{\beta_i}(\hat Q_i^1)\cup \pi_{\beta_i}(P_{\sigma_i}))\}
\leq 
4\chi_1 R+8\chi_1.\]
Together with (\ref{tendtoinfty}) this implies that 
${\rm diam}(\pi_{\beta_i}(P_i)\cup \pi_{\beta_i}(P_{\sigma_i}))\to 
\infty$ $(i\to \infty)$ which contradicts assumption (c) above.
The proposition is proven.
\end{proof}

\section{Length bounds in the thick part of Teichm\"uller space}

In this section we use the results from Section 4
to estimate the length of the intersection of 
a family of Weil-Petersson geodesics with the thick part 
of Teichm\"uller space. In particular, we show that
for every $\epsilon >0$, a 
Weil-Petersson geodesic $\zeta:[0,\sigma]\to {\cal T}(S)$
which connects two points on a Teichm\"uller geodesic
$\gamma\subset {\cal T}(S)_\epsilon$ spends a fixed 
percentage of time in the thick part of Teichm\"uller space.
The main two ingredients for the proof of this fact
are Proposition \ref{finitecontrol} and the following
consequence of hyperbolicity of the curve graph
(which holds true for
every hyperbolic geodesic metric space).

\begin{lemma}\label{curvegraph}
For every $\kappa >1$ there is a number $b(\kappa)>0$ and for
every $n>0$ there are numbers $\tau=\tau(\kappa,n)>0$, 
$T=T(\kappa,n)>0$ with the
following properties. Let $k>0$, let 
$\rho:[0,k]\to {\cal C\cal G}(S)$ be any
geodesic 
and let $\omega:[0,r]\to {\cal C\cal G}(S)$ be a 
one-Lipschitz
curve of length $r\leq \kappa k$ 
connecting $\omega(0)=\rho(0)$ to 
$\omega(r)=\rho(k)$. 
Then there is a set $A\subset \{0,\dots,k-1\}$ of
cardinality at least $\tau k$ and for every $i\in A$ 
there are numbers $r_i\in [0,r],s_i\in [r_i ,r_i+T]$ 
with $r_{i+1}\geq s_i$ and there is a number $j_i\geq i+n$ so that
\[d_{\cal C}(\omega(r_i),\rho(i))\leq b(\kappa),\,
d_{\cal C}(\omega(s_i),\rho(j_i))\leq b(\kappa)\,(i=1,2).\]
\end{lemma}
\begin{proof}
Let $\rho:[0,k]\to {\cal C\cal G}(S)$ 
be any geodesic in the curve graph.
Let $\Pi:{\cal C\cal G}(S)\to \rho[0,k]$
be a shortest distance projection (i.e. $\Pi$ associates to
a point $x\in {\cal C\cal G}(S)$ a point $\Pi(x)\in\rho[0,k]$ 
of minimal distance;
note that such a point need not be unique, and the map
$\Pi$ need not be continuous).
By hyperbolicity of ${\cal C\cal G}(S)$, 
there are numbers $a\geq 1,b>0$ only depending on the hyperbolicity
constant 
such that the projection $\Pi$ satisfies
the following contraction properties (see Section 2 of \cite{MM99}
for these properties in the case of the curve graph).
\begin{enumerate}
\item If $d_{\cal C}(x,y)\leq 1$ then $d_{\cal C}(\Pi(x),\Pi(y))\leq a$.
\item If $d_{\cal C}(x,\Pi(x))\geq a$ 
and $d_{\cal C}(x,y)\leq bd_{\cal C}(x,\Pi(x))$
then $d_{\cal C}(\Pi(x),\Pi(y))\leq a$.
\end{enumerate}

As a consequence, 
the following holds true. For every $p>a$ and for every 
one-Lipschitz edge path 
$\theta:[0,s]\to {\cal C\cal G}(S)$ of length $s>0$ which
does not intersect the $p$-neighborhood of $\rho[0,k]$, 
we have 
\begin{equation}\label{diamcon}
{\rm diam}(\Pi\theta[0,s])\leq 
as/bp+a.
\end{equation}
Namely, by property (2) above, the claim holds true
if $d_{\cal C}(\theta(0),\theta(\sigma))<bp$ for all $\sigma\in [0,s]$. 
Otherwise let $s_1\geq bp$ be the smallest 
number such that $d_{\cal C}(\theta(0),\theta(s_1))=
bp$. Then $d_{\cal C}(\Pi\theta(0),\Pi\theta(\sigma))\leq a$
for every $\sigma\in [0,s_1]$. Moreover,
$d_{\cal C}(\theta(s_1),\Pi\theta(s_1))\geq p$ and 
hence we can repeat this argument for the interval $[s_1,s_2]$
where $s_2\geq s_1+bp$ is the smallest number such that
$d_{\cal C}(\theta(s_1),\theta(s_2))=bp$.
The estimate
(\ref{diamcon}) now follows by induction.

Let $\kappa >1$ and let  
$\omega:[0,r]\to {\cal C\cal G}(S)$ be any 
one-Lipschitz edge path of length
$r\leq \kappa k$ connecting $\omega(0)=\rho(0)$ to 
$\omega(r)=\rho(k)$. Then property (1) above implies that the
image of $\omega[0,r]$ under the projection $\Pi$ is 
$a$-dense in $\rho[0,k]$. This means that for
every $t\in [0,k]$ there is some $s\in [0,r]$
such that $d_{\cal C}(\Pi(\omega(s)),\rho(t))\leq a$.
Moreover, we have $\Pi(\omega(0))=\rho(0)$ and
$\Pi(\omega(r))=\rho(k)$.

Identify $\rho[0,k]$ with the line
segment $[0,k]$. Let $\ell>0$ be the largest
integer so that $5a\ell\leq k$; then 
$k/5a-1\leq \ell \leq k/5a$. 
Choose inductively a sequence
$0=q_0<\dots <q_{\ell}<r\leq \kappa k$ 
by requiring that $q_i$ is the smallest
number such that $d_{\cal C}(\Pi(\omega(q_i)),\rho(5ai))\leq a$.
Note that 
\[\Pi(\omega(q_{i}))+3a\leq \Pi(\omega(q_{i+1}))\leq \Pi(\omega(q_i))+7a
\text{ for all }i.\]

Define
\[A=\{i\leq \ell\mid q_{i+1}\leq q_i+15\kappa a\}.\]
Since the length of $\omega$ does not exceed $\kappa k$
and $\ell+1\geq k/5a$, 
the cardinality of the 
subset $A$ of $\{0,\dots,\ell\}$ 
exceeds $2(\ell+1)/3> k/8a$.

Now if $i\in A$ then the
arc $\omega[q_i,q_{i+1}]$ of length at most
$15\kappa a$ is mapped by $\Pi$ to a subset of
$[0,k]$ of diameter 
at least $3a$. Thus by the estimate
(\ref{diamcon}), if $\chi>0$ is such that
$15\kappa a^2/b \chi +a=3a$ then  
there is some $j(i)\in [q_i,q_{i+1}]$ such that
the distance between 
$\omega(j(i))$ and $\rho[0,k]$
is at most $\chi$. Since $q_{i+1}\leq q_i+15\kappa a$, since
$\omega$ is a one-Lipschitz-curve and since the projection
$\Pi$ is distance minimizing, 
the distance
between $\omega(q_i)$ and $\Pi\omega(q_i)$ is
bounded from above by $\chi +15\kappa a$, i.e. by
a universal constant only depending on $\kappa$.

Let $n\in [1,\ell-1]$ and let 
$T_n:\mathbb{Z}\to \mathbb{Z}$ be
the translation $T_n(z)=z-n$. Since the cardinality
of $A\subset \{0,\dots,\ell\}$ is not smaller than $2(\ell+1)/3$, 
the cardinality
of the intersection $A\cap T_n(A)$ is at least
$(\ell+1-n)/3$. The cardinality of a maximal subset 
$C\subset A\cap T_n(A)$ with the additional
property that if $c\in C$ then $p\not\in C$ for every 
$p\in \{c+1,\dots,c+n-1\}$ is at least $(\ell+1-n)/3n
\geq (k/5a-n)/3n=k/15an-1/3$.

If $c\in C$ then $c\in A$, 
$c+n\in A$ and $c+j\not\in C$ for $1\leq j\leq n-1$.
Recall that $c\in A$ implies that
the distance between $\omega(q_c)$ and $\rho[0,k]$
is uniformly bounded.

Define $\tau=\tau(\kappa,n)=1/40 \kappa an$ 
and $T=T(\kappa,n)=30\kappa an$.
Since the 
cardinality
of $C$ is not smaller than $k/15an-1$ and the 
length of $\omega$ does not exceed $\kappa k$, 
there is a subset $\hat C\subset C$
of cardinality
at least $\tau k$ such that 
$q_{c+n}-q_c\leq T$ for every $c\in \hat C$.

By construction, if $c\in \hat C$ then 
\[\max\{d_{\cal C}(\omega(q_c),\rho(5ac)),
d_{\cal C}(\omega(q_{c+n}),\rho(5a(c+n)))\}
\leq \chi +15\kappa a\]
and $d_{\cal C}(\rho(5ac),\rho(5a(c+n)))= 5an>n$
(recall that $a\geq 1$).
This shows the lemma with 
$b(\kappa)=\chi +15\kappa a$ and with 
$r_c=q_c$ and $s_c=q_{c+n}$ for $c\in \hat C$. 
\end{proof}

To use Lemma \ref{curvegraph} 
to obtain a control on the intersection of a Weil-Petersson
geodesic with the thick part of Teichm\"uller space 
we have to compare the distances $d_{\rm WP}$ and
$d_{\cal T}$ on ${\cal T}(S)$. 

The following statement is well known. We include it
as a lemma for easy reference.

\begin{lemma}\label{wpcomparison}
For all $x,y\in {\cal T}(S)$ the following holds true.
\begin{enumerate}
\item $d_{WP}(x,y)\leq \sqrt{2\pi(2g-2+m)} d_{\cal T}(x,y)$.
\item There is a number $L>1$ such that 
$d_{\cal C}(\Upsilon_{\cal T}(x),\Upsilon_{\cal T}(y))\leq
Ld_{WP}(x,y)+L$.
\end{enumerate}
\end{lemma}
\begin{proof} The first part of the lemma is due to 
Linch \cite{Li74}.

To show the second part of the lemma, 
let as before $\chi_0>0$ be a Bers constant for $S$.
Let $\Upsilon_{\cal P}:{\cal T}(S)\to {\cal P}(S)$ be any
map which associates to a hyperbolic metric
$x\in {\cal T}(S)$ a Bers decomposition for $x$.
By a result of Brock \cite{B03}, the
map $\Upsilon_{\cal P}$ is a quasi-isometry with respect to
the Weil-Petersson metric on ${\cal T}(S)$ and the
combinatorial metric $d_{\cal P}$ on the pants graph ${\cal P\cal G}(S)$:  
There is a number
$L_1>0$ such that
\begin{equation}\label{qisom}
d_{WP}(x,y)/L_1-L_1\leq 
d_{\cal P}(\Upsilon_{\cal P} x,\Upsilon_{\cal P} y)\leq L_1d_{WP}(x,y)+L_1
\text{ for all } x,y\in {\cal T}(S).
\end{equation}

Let $\Psi:{\cal P}(S)\to 
{\cal C}(S)$ be a map which associates to
a pants decomposition $P\in {\cal P}(S)$
one of its component curves.
If $P,P^\prime$ are pants decompositions
with $d_{\cal P}(P,P^\prime)=1$
then $P^\prime$ can be obtained from $P$ by an elementary
move and hence $P\cap P^\prime\not=\emptyset$ (recall that
by assumption, the surface $S$ is not a once punctured torus
or a four-punctured sphere).
This implies that the distance in ${\cal C\cal G}(S)$
between any component of $P$ and any component of $P^\prime$ is
at most $2$. Since the pants graph ${\cal P\cal G}(S)$ is a geodesic metric
graph, we have \[d_{\cal C}(\Psi P_1,\Psi P_2)\leq 
2d_{\cal P}(P_1,P_2)\text{ for all }P_1,P_2\in {\cal P}(S).\]

Now the map $\Upsilon_{\cal T}:{\cal T}(S)\to {\cal C}(S)$ which
associates to a point $x\in {\cal T}(S)$ a simple
closed curve of $x$-length at most $\chi_0$ can be
chosen to coincide with $\Psi\circ \Upsilon_{\cal P}$. 
Together with inequality (\ref{qisom}), 
the second part of the lemma follows.
\end{proof}

As before, let $\chi_0>0$ be a Bers constant for $S$.
By Lemma \ref{bounded}, if $\alpha,\beta\in {\cal C}(S)$ are such that
there is some $x\in {\cal T}(S)$ with 
$\ell_\alpha(x)\leq \chi_0,\ell_\beta(x)\leq \chi_0$ the 
the distance in ${\cal C\cal G}(S)$
between $\alpha,\beta$ 
is at most $a(\chi_0)>0$. In particular, if
$\alpha\in {\cal C}(S)$ is of $x$-length at most $\chi_0$ and 
and if $\beta\in {\cal C}(S)$ is such that $d_{\cal C}(\alpha,\beta)\geq
a(\chi_0)+3$ then $\beta\cup \xi$ binds $S$ for  
any simple closed curve $\xi\in {\cal C}(S)$ with 
$\ell_\xi(x)\leq \chi_0$.

For $\alpha\in {\cal C}(S)$ 
recall from Section 4 the definition of the graph
${\cal C\cal G}(\alpha)$ and
the definition of the projection 
$\pi_\alpha:{\cal C}(S)\to {\cal C\cal G}(\alpha)$. 
The following lemma is a version of a result of Masur
and Minsky  \cite{MM00}: 
Teichm\"uller geodesics in the thick part of 
Teichm\"uller space have bounded combinatorics
(see also  \cite{R05} and
\cite{R14}). 

\begin{lemma}\label{localcontrol}
For every $R>0$ there are numbers
$\delta=\delta(R)>0$, 
$c=c(R)>0$ with the property 
that for every Teichm\"uller geodesic arc
$\gamma:[0,r]\to{\cal T}(S)$
of length $r\leq R$ such that
$d_{\cal C}(\Upsilon_{\cal T}\gamma(0),\Upsilon_{\cal T}\gamma(r))\geq 2a(\chi_0)+3$
the following is satisfied.
\begin{enumerate}
\item $\gamma[0,r]\subset {\cal T}(S)_\delta$.
\item 
${\rm diam}(\pi_\beta(\Upsilon_{\cal T}\gamma(0))\cup
\pi_{\beta}(\Upsilon_{\cal T}\gamma(r)))\leq c$ for all $\beta\in {\cal C}(S).$ 
\end{enumerate}
\end{lemma}
\begin{proof}
For $R>0$ a Teichm\"uller geodesic
$\gamma:[0,r]\to {\cal T}(S)$ of length $r\leq R$ such that
$d_{\cal C}(\Upsilon_{\cal T}\gamma(0),\Upsilon_{\cal T}\gamma(r))\geq 2a(\chi_0)+3$ 
is entirely contained in ${\cal T}(S)_\delta$ where
$\delta=\chi_0e^{-R}$. Namely, 
by Lemma 3.1 of \cite{W79}, 
if $\zeta:J\subset \mathbb{R}\to {\cal T}(S)$ is any Teichm\"uller geodesic
and if $s\in J$, 
$\alpha\in {\cal C}(S)$ are such that 
$\ell_{\alpha}(\zeta(s))\leq \delta$
then $\ell_{\alpha}(\zeta(t))\leq \chi_0$ for every 
$t$ with $\vert s-t\vert \leq \log \chi_0-\log \delta$.
In particular, by Lemma \ref{bounded}, if 
$\vert s-t_i\vert \leq \log \chi_0-\log \delta$ 
for $i=1,2$ then 
we have
$d_{\cal C}(\Upsilon_{\cal T}(\zeta(t_1)),\Upsilon_{\cal T}(\zeta(t_2)))
\leq 2a(\chi_0)$. The first part of the lemma follows.

The second part of the lemma can be extracted from \cite{MM00}. 
Perhaps the most convenient reference is the distance formula
of \cite{R07b}.
\end{proof}

We use 
Proposition \ref{finitecontrol} and  
Lemma \ref{curvegraph}-\ref{localcontrol} 
to show that a WP-geodesic connecting two sufficiently far apart 
points on a Teichm\"uller 
geodesic $\gamma:\mathbb{R}\to {\cal T}(S)$
whose image under $\Upsilon_{\cal T}$ makes controlled
progress in the curve graph spends a definitive
proportion of time in the thick part of Teichm\"uller space.

\begin{proposition}\label{thickmeet}
For every $\theta >0$ there are numbers 
$\delta=\delta(\theta) >0$ and $\eta=\eta(\theta)>0$ 
with the following property. Let $\xi\geq 1/\eta$ and let 
$\gamma:[0,\xi]\to 
{\cal T}(S)$ be a Teichm\"uller geodesic 
such that 
\[d_{\cal C}(\Upsilon_{\cal T}(\gamma(0)),\Upsilon_{\cal T}(\gamma(\xi)))\geq
\theta \xi.\] If 
$\zeta:[0,\sigma]\to {\cal T}(S)$
is the Weil-Petersson geodesic connecting $\zeta(0)=\gamma(0)$
to $\zeta(\sigma)=\gamma(\xi)$ then 
$\ell_{\delta-{\rm thick}}(\zeta)\geq \eta \xi$.
\end{proposition}
\begin{proof}
Let  $\theta >0$ and let 
$\gamma:[0,\xi]\to {\cal T}(S)$ be any 
Teichm\"uller geodesic of length $\xi\geq 2a(\chi_0)+3/\theta$ 
such that 
\begin{equation}\label{synchron}
d_{\cal C}(\Upsilon_{\cal T}(\gamma(0)),\Upsilon_{\cal T}(\gamma(\xi)))\geq
\theta \xi.\end{equation}
Here $a(\chi_0)>0$ is as in Lemma \ref{bounded} for a Bers constant 
$\chi_0$ for $S$.

Write $C=\sqrt{2\pi(2g-2+m)}$. 
Lemma \ref{wpcomparison} yields that 
\begin{align}\label{squeeze}
d_{\cal C}(\Upsilon_{\cal T}\gamma(0),
\Upsilon_{\cal T}\gamma(\xi))/L-1 & \leq 
d_{WP}(\gamma(0),\gamma(\xi))\\ 
\leq C\xi= Cd_{\cal T}(\gamma(0),\gamma(\xi)) & \leq  
Cd_{\cal C}(\Upsilon_{\cal T}\gamma(0),\Upsilon_{\cal T}\gamma(\xi))/\theta.
\notag
\end{align}

Let $\zeta:[0,\sigma]\to 
{\cal T}(S)$ be the WP-geodesic connecting $\zeta(0)=\gamma(0)$
to $\zeta(\sigma)=\gamma(\xi)$. 
By the second part of Lemma \ref{wpcomparison},
we have $d_{\cal C}(\Upsilon_{\cal T}(\zeta(i)),
\Upsilon_{\cal T}(\zeta(i+1)))\leq 2L$ for all $i$.
Thus via connecting 
$\Upsilon_{\cal T}(\zeta(i))$ to $\Upsilon_{\cal T}(\zeta(i+1))$ 
by a geodesic segment
in ${\cal C\cal G}(S)$ for each $i$ 
and subsequent reparametrization
we obtain a 1-Lipschitz curve $\omega:[0,u]\to {\cal C\cal G}(S)$ 
connecting $\omega(0)=\Upsilon_{\cal T}(\gamma(0))$ to 
$\omega(u)=\Upsilon_{\cal T}(\gamma(\xi))$ whose
length $u$ does not exceed $2L d_{WP}(\gamma(0),\gamma(\xi))$.
(To obtain this estimate we slightly adjust the constant $L$ to 
accommodate for the problem that the length $\sigma$ of $\zeta$
is not integral in general. This adjustment simplifies the notation).
 
Inequality (\ref{squeeze}) then implies that 
the length $u$ of $\omega$ is also bounded from above by 
\[u\leq 2LC d_{\cal C}(\Upsilon_{\cal T}\gamma(0),
\Upsilon_{\cal T}\gamma(\xi))/\theta= \kappa 
d_{\cal C}(\Upsilon_{\cal T}\gamma(0),\Upsilon_{\cal T}\gamma(\xi))\]
where $\kappa =2LC/\theta>0$ only depends on $\theta$.
Moreover, every point on the curve $\omega$ is of distance at most
$L$ from a point in $\Upsilon_{\cal T}(\zeta[0,\sigma])$. 

Let $k=
d_{\cal C}(\Upsilon_{\cal T}\gamma(0),\Upsilon_{\cal T}\gamma(\xi))\geq
\max\{\theta \xi,u/\kappa\}$. 
By Theorem \ref{unparam}, there is a number $L_1>1$ such that
the curve 
$t\to \Upsilon_{\cal T}(\gamma(t))$ is an unparametrized
$L_1$-quasi-geodesic 
in the curve graph. 
Hence
by hyperbolicity, the Hausdorff distance between
its image and the image of 
a geodesic $\rho:[0,k]\to {\cal C\cal G}(S)$ with
the same endpoints is bounded from above 
by a universal constant $\beta>0$.

Let $p=p(\chi_0,\chi_1)\geq 1$ be as 
defined before Proposition \ref{finitecontrol}.
For $\kappa=2LC/\theta>1$ let $b(\kappa)>0$ be as in 
Lemma \ref{curvegraph} and write
\[B=b(\kappa)+\beta+L.\] 
Note that $B>0$ only depends on $\theta$. 

Apply Lemma \ref{curvegraph} to the geodesic
$\rho$ in ${\cal C\cal G}(S)$ of length $k$, the one-Lipschitz
edge path  $\omega:[0,u]\to {\cal C\cal G}(S)$ of length
$u\leq \kappa k$ connecting $\omega(0)=\rho(0)$ to 
$\omega(u)=\rho(k)$ and the constant 
\[n=2p+4B+2\beta+5a(\chi_0)+8.\] 
Note that $n>0$ only depend on $\theta$.
Using the fact that every point on 
$\omega[0,u]$ is of distance at most $L$ from
a point in $\Upsilon_{\cal T}(\zeta[0,\sigma])$ and 
that the Hausdorff distance between
$\Upsilon_{\cal T}\gamma[0,\xi]$ and $\rho[0,k]$ is 
at most $\beta$, we conclude that 
there are numbers 
$\tau=\tau(\theta) >0$ and $T=T(\theta)>0$ 
and there is a subset $A_0$ of $\{1,\dots,k\}$ of 
cardinality at least $\tau k$ and for each $i\in A_0$ there is some 
$r_i\in [0,\sigma]$ and some $t_i\in [0,\xi]$ 
such that the following holds true.
\begin{enumerate}
\item $d_{\cal C}(\Upsilon_{\cal T}(\zeta(r_i)),
\Upsilon_{\cal T}(\gamma(t_i)))\leq B$.
\item There are numbers $e_i\leq T,v_i>0$ such that
$d_{\cal C}(\Upsilon_{\cal T}\zeta(r_i+e_i),
\Upsilon_{\cal T}\gamma(t_i+v_i))\leq B$.
\item $d_{\cal C}(\Upsilon_{\cal T}\gamma(t_i),
\Upsilon_{\cal T}\gamma(t_i+v_i))\geq 2p+4B+5a(\chi_0)+8$.
\item $r_{i+1}\geq r_i+e_i$, $t_{i+1}\geq t_i+v_i$.
\end{enumerate}

Since  the length $\xi$ of the 
Teichm\"uller geodesic arc $\gamma$ does not exceed 
$k/\theta$ and since the cardinality of $A_0$ is at least
$\tau k$, there is a subset $A_1$  of $A_0$ of cardinality
at least $\tau k/2$ such that
$v_i\leq 2/\theta\tau$ for every $i\in A_1$.
Lemma \ref{localcontrol} then yields the existence
of a number $c>0$ only depending on $\theta$ such that
for every $i\in A_1$ and 
for every $\alpha\in {\cal C}(S)$ we have
\begin{equation}\label{rafi}
{\rm diam}(\pi_\alpha(\Upsilon_{\cal T}\gamma(t_i)),
\pi_\alpha(\Upsilon_{\cal T}\gamma(t_i+v_i)))\leq c.
\end{equation}

Let $i\in A_1$ and let 
$\phi,\psi\in {\cal C}(S)$ be such that
\[d_{\cal C}(\phi,\Upsilon_{\cal T}\gamma(t_i))\leq B,
d_{\cal C}(\psi,\Upsilon_{\cal T}\gamma(t_i+v_i))\leq B.\]
If 
$\alpha\in {\cal C}(S)$ is such that
$d_{\cal C}(\phi,\alpha)\geq
B+3,d_{\cal C}(\psi,\alpha)\geq B+3$ then a geodesic in 
${\cal C\cal G}(S)$ connecting 
$\phi,\psi$ to 
$\Upsilon_{\cal T}(\gamma(t_i)),
\Upsilon_{\cal T}(\gamma(t_i+v_i))$ 
does not intersect the ball of radius $2$ about $\alpha$.
In particular, each of the vertices of ${\cal C\cal G}(S)$ 
passed through by these geodesics intersects $\alpha$ transversely.
Now by Lemma 2.3 of \cite{MM00}, if $\beta_1,\beta_2\in {\cal C}(S)$ 
are two curves of distance one in ${\cal C\cal G}(S)$ which
intersect $\alpha$ transversely then the diameter in
${\cal C\cal G}(\alpha)$ of the projection
$\pi_{\alpha}(\beta_1)\cup \pi_{\alpha}(\beta_2)$ is at most 2. 
An inductive application of this fact to a geodesic
in ${\cal C\cal G}(S)$ connecting $\phi,\psi$ to
$\Upsilon_{\cal T}(\gamma(t_i)),\Upsilon_{\cal T}(\gamma(t_i+v_i))$ 
shows that 
\[{\rm diam}(\pi_{\alpha}(\phi)\cup
\pi_{\alpha}(\Upsilon_{\cal T}\gamma(t_i)))\leq B+1,\,
{\rm diam}(\pi_{\alpha}(\psi)\cup\pi_{\alpha}(\Upsilon_{\cal T}\gamma(t_i+v_i)))
\leq B+1.\]
Together with the estimate (\ref{rafi}) we obtain that
\begin{equation}\label{diam2}
{\rm diam}(\pi_{\alpha}(\phi)\cup \pi_{\alpha}(\psi))\leq c+2B+2.
\end{equation}

For $i\in A_1$ (where $A_1\subset \{1,\dots,k\}$ is
as above) consider the WP-geodesic arc
$\zeta[s_i,s_{i}+e_i]$.
By the properties (1)-(3) above, 
we have 
\[d_{\cal C}(\Upsilon_{\cal T}\zeta(s_i),
\Upsilon_{\cal T}\zeta(s_i+e_i))\geq 2p +2B+5a(\chi_0)+8.\]
Even though the map $t\to \Upsilon_{\cal T}(\zeta(t))$ is 
not continuous, by 
the choice of the constant $a(\chi_0)$ the set 
$\Upsilon_{\cal T}\zeta[0,\sigma]$ is $a(\chi_0)$-dense 
in a simplicial arc connecting $\Upsilon_{\cal T}(\zeta(0))$
to $\Upsilon_{\cal T}(\zeta(\sigma))$. Therefore
there is a number $t\in (s_i,s_i+e_i)$ so that
\begin{align}\label{midpoint}
d_{\cal C}(\Upsilon_{\cal T}\zeta(t),
\Upsilon_{\cal T}\zeta(s_i))\geq p+B+2a(\chi_0)+4,      \\
d_{\cal C}(\Upsilon_{\cal T}\zeta(t),
\Upsilon_{\cal T}\zeta(s_i+e_i))\geq p+B+2a(\chi_0)+4.\notag
\end{align}
Note that by the second part 
of Lemma \ref{wpcomparison}, there is a number 
$v_0>0$ such that $\min\{\vert s_i-t\vert,
\vert s_i+e_i-t\vert \}\geq v_0$.

If $P_i,Q_i$ are Bers decompositions for $\zeta(s_i),
\zeta(s_i+e_i)$, then using  
once more the definition of $a(\chi_0)$,  
inequality (\ref{midpoint}) implies that the distance in ${\cal C}(S)$ 
between any component of $P_i,Q_i$ and
every component $\psi$ of a Bers decomposition for $\zeta(t)$ is at least 
$p+B+4$. Thus the 
estimate (\ref{diam2}), applied to components 
$\phi,\psi$ of $P_i,Q_i$ and to  
every $\beta\in {\cal C}(S)$ with 
$d_{\cal C}(\xi,\beta)\leq 1$  
(which is possible
by the property (1) above) yields that 
\[{\rm diam}(\pi_{\beta}(P_i)\cup
\pi_{\beta}(Q_i))\leq c+2B+2.\]

Now $\psi$ is an arbitrary component of a Bers decomposition
for $\zeta(t)$, and the length $e_i$ of the
arc $\zeta_i[s_i,s_i+e_i]$ does not exceed $T$. Therefore 
Proposition \ref{finitecontrol} shows the existence
of a number $\delta=\delta(\theta) >0$ only depending on $\theta$ 
such that
$\zeta(t)\in {\cal T}(S)_{2\delta}$. Since by Wolpert's results
the Weil-Petersson distance between ${\cal T}(S)_{2\delta}$ and
${\cal T}(S)-{\cal T}(S)_\delta$ is bounded from below
by a universal constant (compare the beginning of the proof of
Lemma \ref{pants} for a more precise statement),
there is a number $v \leq v_0$ such that
$\ell_{\delta-{\rm thick}}(\zeta[s_i,s_i+e_i])\geq v$.

Recall that
$i\in A_1$ was arbitrary and that the cardinality of $A_1$ is not 
smaller than $\tau k/2$. Then the above consideration
shows that $\ell_{\delta-{\rm thick}}(\zeta)\geq 
\tau kv/2\geq \tau \xi v/2\theta$. Since $\tau,v$ only depend on 
$\theta$, the proposition follows.
\end{proof}

\section{Conjugating the flows}\label{conjugating}

For a $\Phi_{\cal T}^t$-invariant Borel subset
$A$ of ${\cal Q}(S)$ define a \emph{measurable
conjugacy} of the restriction of $\Phi_{\cal T}^t$ to $A$
into the geodesic flow of the Weil-Petersson metric 
to be an injective measurable map $\Lambda:A\to {\cal Q}_{WP}(S)$
such that there is a measurable function
$\psi:A\times \mathbb{R}\to \mathbb{R}$ with the
following properties.
\begin{enumerate}
\item $\psi(x,0)=0$ for all $x\in A$.
\item For each fixed $x\in A$ the function 
$\psi(x,\cdot):s\to \psi(x,s)$ is an increasing homeomorphism.
\item $\Lambda(\Phi_{\cal T}^tx)=\Phi_{WP}^{\psi(x,t)}\Lambda(x)$
for all $x\in A,t\in \mathbb{R}$.
\end{enumerate}

The goal of this section is to construct 
such a measurable conjugacy $\Lambda$
on a $\Phi^t_{\cal T}$-invariant 
Borel subset ${\cal E}$ of ${\cal Q}(S)$ which 
has full measure for every $\Phi^t_{\cal T}$-invariant
Borel probability measure. 
The restriction of the map to any compact invariant
subset of ${\cal Q}(S)$ is continuous. We use this to establish the 
first part of Theorem \ref{fellowtravel}.

We begin with isolating properties of typical orbits
for $\Phi^t_{\cal T}$-invariant Borel probability measures. 
We continue to use the assumptions and notations from 
Sections 2-6.
Let again $\Upsilon_{\cal T}:{\cal T}(S)\to {\cal C}(S)$
be a map which associates to a point $x\in {\cal T}(S)$
a Bers curve on $x$, i.e. a 
simple closed curve of $x$-length at most $\chi_0$ where
$\chi_0>0$ is a Bers constant for $S$. Our first goal is
to study the distances $d_{\cal C}(\Upsilon_{\cal T}\gamma(0),
\Upsilon_{\cal T}\gamma(t))$ as $t\to \infty$ for 
a Teichm\"uller geodesic $\gamma$ whose cotangent line
projects to an orbit of the flow $\Phi^t_{\cal T}$ 
on ${\cal Q}(S)$ which is typical for an
invariant ergodic Borel probability
measure on ${\cal Q}(S)$. 

For this we face the problem that
the map $\Upsilon_{\cal T}$ depends on choices and 
may not be measurable,
and the same problem may arise for the 
function $(x,y)\to d_{\cal C}(\Upsilon_{\cal T}(x),
\Upsilon_{\cal T}(y))$ on ${\cal T}(S)\times {\cal T}(S)$.
We resolve this problem by
using a construction from \cite{H10a}.

Namely, 
choose a smooth function
$\sigma:[0,\infty)\to [0,1]$ with $\sigma[0,\chi_0]\equiv 1$ and
$\sigma[2\chi_0,\infty)\equiv 0$. 
For every $h\in {\cal T}(S)$ we obtain a finite Borel
measure $\mu_h$ on ${\cal C}(S)$ by defining
\begin{equation}
\mu_h=\sum_\beta \sigma(\ell_h(\beta))\delta_\beta\notag
\end{equation}
where $\delta_\beta$ denotes the Dirac mass at $\beta$.
By the collar lemma, the 
number of simple closed geodesics on $(S,h)$ 
of length at most $2\chi_0$ is bounded
from above independent of $h$.
Thus the total mass of $\mu_h$ is bounded from
above and below by a universal positive constant, and
by Lemma \ref{bounded}, 
the diameter in ${\cal C\cal G}(S)$
of the support of $\mu_h$ 
is uniformly bounded as well. 

Define a symmetric non-negative function $\rho$ on
${\cal T}(S)\times {\cal T}(S)$ by
\begin{equation}
\rho(h,h^\prime)=\int_{{\cal C}(S)\times {\cal C}(S)}
d_{\cal C}(\cdot,\cdot)d\mu_h\times d\mu_h^{\prime}/\mu_h({\cal C}(S))
\mu_{h^\prime}({\cal C}(S)).\notag
\end{equation}
Lemma 3.3 of \cite{H10a} shows that 
the function $\rho$ is
continuous and invariant
under the diagonal 
action of ${\rm Mod}(S)$. Moreover, there is a universal
constant $a_0>0$ such that 
\begin{equation}\label{rho}
\rho(h,h^\prime)-a_0\leq
d_{\cal C}(\Upsilon_{\cal T}(h),\Upsilon_{\cal T}(h^\prime))
\leq \rho(h,h^\prime)+a_0 \text{ for all }h,h^\prime\in {\cal T}(S).
\end{equation}

For $q\in {\cal Q}(S)$ and for $t\geq 0$ we can now define a number
$r(q,t)\geq 0$ as follows. 
A lift $\tilde q$ of $q$ to $\tilde {\cal Q}(S)$ determines a  
Teichm\"uller geodesic $\gamma:\mathbb{R}\to {\cal T}(S)$ with
initial unit cotangent 
$\tilde q$. By invariance of the function $\rho$ under the
diagonal action of the mapping class group, the number 
\[r(q,t)=\rho(\gamma(0),\gamma(t))\] 
does not depend 
on the choice of $\tilde q$ and hence it defines
a continuous function
$r:{\cal Q}(S)\times \mathbb{R}\to [0,\infty)$.
We have.

\begin{lemma}\label{deviation}
Let $\mu$ be a $\Phi_{\cal T}^t$-invariant ergodic
Borel probability measure on ${\cal Q}(S)$. Then 
there is a number $b=b(\mu)>0$ such that 
\[\lim_{t\to \infty}\frac{1}{t}r(q,t)=b\]
for $\mu$-almost every $q\in {\cal Q}(S)$. 
\end{lemma}
\begin{proof} By the triangle inequality for $d_{\cal C}$ and
the estimate (\ref{rho}) above, 
for $q\in {\cal Q}(S)$ and $s,t>0$ we have
\[r(q,s+t)\leq r(q,s)+r(\Phi_{\cal T}^sq,t)+2a_0\] where
$a_0>0$ is as in (\ref{rho}). Thus by the
subadditive ergodic theorem \cite{Kr85}, 
for $\mu$-almost 
every $q\in {\cal Q}(S)$ the limit
\[\lim_{t\to \infty}\frac{1}{t}r(q,t)\]
exists and does not depend on $q$. (Here we
apply the subadditive ergodic theorem to the
function $(q,t)\to r(q,t)+2a_0$).

We have to show that this limit is positive 
almost everywhere. For this 
recall from Theorem \ref{unparam}
that there is a number $L_1>0$
such that the image under 
$\Upsilon_{\cal T}$ of any Teichm\"uller geodesic
is an unparametrized $L_1$-quasi-geodesic in 
${\cal C\cal G}(S)$.
By the Poincar\'e recurrence theorem, 
the $\Phi_{\cal T}^t$-orbit of $\mu$-almost
every $q\in {\cal Q}(S)$ is recurrent under
the Teichm\"uller flow. Therefore the support of the 
vertical measured geodesic lamination of
$\mu$-almost every $q\in {\cal Q}(S)$ fills $S$ and 
is uniquely ergodic \cite{M82}.
As a consequence, 
if $\gamma:\mathbb{R}\to {\cal T}(S)$
is a geodesic whose unit cotangent line is
a lift of an orbit of $\Phi^t_{\cal T}$ which is
typical for $\mu$ then
the map $t\to \Upsilon_{\cal T}(\gamma(t))$ is
an unparametrized $L_1$-quasi-geodesic
of infinite length \cite{Kl99,H06}.
This means that
$r(q,t)\to \infty$ $(t\to \infty)$ for 
$\mu$-almost every $q$.
Moreover, by Lemma 2.4 of \cite{H10a} 
and the estimate (\ref{rho}) above, 
there is a number $c>0$ such that 
\begin{equation}\label{curvegraphdist}
r(q,t+s)-c\leq 
r(q,s)+r(\Phi^s q,t)\leq r(q,t+s)+c \text{ for all }q\in {\cal Q}(S),s,t>0. 
\end{equation}

Now the function $r:{\cal Q}(S)\times \mathbb{R}\to [0,\infty)$ 
is continuous and the measure $\mu$ is Borel regular. 
Hence there is a compact 
subset $A$ of ${\cal Q}(S)$ with
$\mu(A)>3/4$ and a number $t_0>0$
such that
$r(q,t_0)\geq 4c$ for all $q\in A$ where $c>0$ is as in 
(\ref{curvegraphdist}) above. 
For the continuous 
non-negative function
$\phi(q)=r(q,t_0)$ on ${\cal Q}(S)$ we then have
$\int \phi d\mu\geq 3c$.
The Birkhoff ergodic theorem
implies that
\[\frac{1}{t_0}\int_0^{t_0}
\bigl(\lim_{n\to \infty}\frac{1}{n}
\sum_{i=0}^{n-1}\phi(\Phi_{\cal T}^{it_0}\Phi_{\cal T}^sq)\bigr) ds
\geq 3c\] for $\mu$-almost every $q$.
The estimate (\ref{curvegraphdist}) shows that
$r(q,nt_0)\geq \sum_{i=0}^{n-1}\phi(\Phi^{it_0}_{\cal T}q)-nc$
and therefore 
\[\lim\inf_{n\to \infty}\frac{1}{nt_0}r(q,nt_0)\geq 2c\]
for $\mu$-almost every $q$.
This completes the proof of the lemma.
\end{proof}


For $j>0,\ell>0$ define 
\[B^+(j,\ell)=\{q\in {\cal Q}(S)\mid 
r(q,t)\geq t/j \text{  for all }t\geq \ell\}.\]

\begin{lemma}\label{compact}
The set $B^+(j,\ell)$ is compact and consists
of quadratic differentials with uniquely ergodic
vertical measured lamination. 
\end{lemma}
\begin{proof} 
Since the function $r$ is continuous, we have
\[B^+(j,\ell)=\cap_{s\geq \ell,s\in \mathbb{Q}}
\{q\mid r(q,s)\geq s/j\}\]
and hence $B^+(j,\ell)$ is a countable intersection of 
closed sets. In particular, $B(j,\ell)\subset {\cal Q}(S)$ is closed.

For $\epsilon>0$ let 
${\cal M}(S)_\epsilon$ be the projection of ${\cal T}(S)_\epsilon$
into the moduli space of curves. 
Let $q$ be a quadratic differential whose
vertical measured geodesic lamination is not uniquely
ergodic. Then for every $\epsilon >0$ there is some
$\tau(\epsilon)>0$ such that for every 
$s\geq \tau(\epsilon)$, the surface underlying 
$\Phi^s_{\cal T}(q)$ is not contained in  
${\cal M}(S)_\epsilon$  \cite{M82}.
By Lemma 3.1 of \cite{W79}, if $\alpha$ is a curve of 
length at most $\epsilon$ on
the surface underlying $\Phi^s_{\cal T}(q)$, then the 
length of $\alpha$ on the surface underlying $\Phi^u_{\cal T}(q)$
is at most $\chi_0$ for every $u\in [s-\log(\chi_0-\epsilon),
s+\log(\chi_0-\epsilon)]$. 
This implies that if $\epsilon$ is sufficiently small then  
we have $r(z,s)\leq s/2j$ for all sufficiently large $s$ and every 
quadratic differential $z$ with the property that the projection
of the  
forward orbit $\{\Phi^t_{\cal T}z\mid t\geq 0\}$ of $z$ does
not intersect ${\cal M}(S)_\epsilon$.

By the definition of the set $B^+(j,\ell)$, this shows first that
$B(j,\ell)$ projects into ${\cal M}(S)_\epsilon$ for a number
$\epsilon >0$ depending on $j,\ell$. In particular, $B(j,\ell)$ is 
compact. Moreover, by 
the estimate (\ref{curvegraphdist}), the orbit of 
$q\in B^+(j,\ell)$ under $\Phi^t_{\cal T}$ 
recurs to a fixed compact
set for arbitrarily large times hence
its vertical measured lamination is uniquely ergodic
\cite{M82}: 
\end{proof}

Let ${\cal F}:{\cal Q}(S)\to {\cal Q}(S)$ be the 
\emph{flip} ${\cal F}(q)=-q$. Define 
\begin{align}
B(j,\ell)&=
B^+(j,\ell)\cap {\cal F}(B^+(j,\ell)),\notag\\  
E^+(j,\ell)&=B^+(j,\ell)\cap \lim\sup_{i\to \infty} \Phi^{-i}B^+(j,\ell)=
B(j,\ell)\cap (\cap_{u=1}^\infty
(\cup_{i=u}^\infty \Phi^{-i}B(j,\ell))),\notag\\
E(j,\ell)&=E^+(j,\ell)\cap {\cal F}(E^+(j,\ell))\text{ and }
{\cal E}=\cup_j(\cup_\ell E(j,\ell)).\notag\end{align}


\begin{lemma}\label{borel}
\begin{enumerate}
\item ${\cal E}$ is a $\Phi^t_{\cal T}$-invariant Borel subset of 
${\cal Q}(S)$.
\item
$\mu({\cal E})=1$ for every
$\Phi^t_{\cal T}$-invariant Borel probability measure. 
\item The horizontal and vertical meausured geodesic laminations of 
every $q\in {\cal E}$ are uniquely ergodic.
\item Every compact $\Phi^t_{\cal T}$-invariant subset of ${\cal Q}(S)$
is contained in ${\cal E}$.
\end{enumerate}
\end{lemma}
\begin{proof} That ${\cal E}$ is a Borel set is immediate
from Lemma \ref{compact} and the definitions.
Moreover, Lemma \ref{compact} shows that the horizontal and
vertical measured geodesic laminations of every $q\in {\cal E}$ are
uniquely ergodic.
 
Part (2) of the lemma follows from
Lemma \ref{deviation} and the fact that the image under
the flip of any $\Phi^t_{\cal T}$-invariant 
Borel probability measure on ${\cal Q}(S)$ is an
invariant Borel probability measure. 

To show invariance of ${\cal E}$ under 
$\Phi^t_{\cal T}$, it suffices to show that
$\cup_j(\cup_\ell E^+(j,\ell))$ is invariant. For this 
recall that 
$E^+(j,\ell)$ is the set of all points $q\in B^+(j,\ell)$ so that
$\Phi_{\cal T}^iq\in B^+(j,\ell)$ for infinitely many $i>0$.
Thus if $q\in E^+(j,\ell)$ and if $t>0$ then there is some
$u\geq 0$ with $\Phi_{\cal T}^u(\Phi_{\cal T}^tq)\in 
E^+(j,\ell)$. Inequality (\ref{curvegraphdist}) shows that 
\[r(\Phi^t_{\cal T}(q),s)\geq r(\Phi^{u+t}_{\cal T}(q),s-u)-c\geq 
(s-u)/j-c\]
for all $s\geq \ell+u$ and hence putting
$\ell^\prime=2\max\{\ell+u,2c\}$ we conclude that  
$\Phi^t_{\cal T}(q)\in B(2j,\ell^\prime)$.
The same argument also shows that
$\Phi^t_{\cal T}q\in E^+(2j,\ell^\prime)$.

Part (4) of the lemma is an immediate consequence of the
fact that the image under $\Upsilon$ of a Teichm\"uller geodesic
which is entriely contained in ${\cal T}(S)_\epsilon$ for some 
$\epsilon >0$ is a \emph{parametrized} quasi-geodesic in 
${\cal C\cal G}(S)$.
 \end{proof}

The idea for constructing the conjugacy is to 
find for a Teichm\"uller geodesic with uniquely ergodic
horizontal and vertical measured laminations 
a Weil-Petersson geodesic which has these laminations
as ending measures. For this we need a technical preparation
which we formulate as two lemmas.

\begin{lemma}\label{WPconv}
Let $\gamma_i:\mathbb{R}\to {\cal T}(S)$ 
be a sequence of Teichm\"uller geodesics converging 
locally uniformly to a Teichm\"uller geodesic $\gamma$.
Assume that the 
vertical measured geodesic lamination $\nu$ 
of the quadratic differential which defines $\gamma$ is uniquely ergodic.
Let $n_i\to \infty$ and for each $i$ let $\zeta_i:[0,\sigma_i]\to
{\cal T}(S)$ be the WP-geodesic connecting $\zeta_i(0)=\gamma_i(0)$
to $\zeta_i(\sigma_i)=\gamma_i(n_i)$. Then up to passing 
to a subsequence, the geodesics $\zeta_i$ converge locally
uniformly to an infinite WP-ray $\zeta:[0,\infty)\to {\cal T}(S)$.
The length of $\nu$ is bounded along $\zeta$. 
\end{lemma}
\begin{proof}
By Theorem \ref{unparam}, 
there is a number
$L_1>0$ such that the image under
$\Upsilon_{\cal T}$ of every Teichm\"uller geodesic 
$\gamma:\mathbb{R}\to {\cal T}(S)$ 
in an unparametrized $L_1$-quasi-geodesic
in ${\cal C\cal G}(S)$. 
By hyperbolicity
of the curve graph, this means in particular that there is a
number $b>0$  and there
is a geodesic $\rho:[0,b)\to {\cal C\cal G}(S)$ such that
the Hausdorff distance between $\rho(J)$ and 
$\Upsilon_{\cal T}(\gamma[0,\infty)$ is bounded from
above by a universal constant $p_1>0$.

Let $\gamma_i,\gamma:\mathbb{R}\to {\cal T}(S)$ be as in the lemma. 
Let $\nu$ be the vertical measured geodesic lamination of 
$\gamma$, normalized in such a way that the $\gamma(0)$-length 
of $\nu$ equals one. By assumption, 
$\nu$ is uniquely ergodic. This implies that the diameter
of $\Upsilon_{\cal T}(\gamma[0,\infty))$ is infinite and the support of 
$\nu$ defines a point in the boundary $\partial{\cal C\cal G}(S)$
of the curve graph ${\cal C\cal G}(S)$.
If $(\beta_i)\subset {\cal C}(S)$ is any sequence converging
in ${\cal C\cal G}(S)\cup \partial {\cal C\cal G}(S)$ to the support of $\nu$ 
then by unique ergodicity of $\nu$, 
the projective measured geodesic laminations $[\beta_i]$ supported in 
$\beta_i$ converge in ${\cal P\cal M\cal L}$ to the 
projective measured geodesic lamination $[\nu]$ which is
the class of $\nu$
(see Theorem 1.4 of \cite{Kl99} and Theorem 1.1 of \cite{H06}).

Since $\gamma_i\to \gamma$ uniformly on compact sets,
as $i\to \infty$ longer and longer subarcs of the uniform
unparametrized quasi-geodesic 
$\Upsilon_{\cal T}\circ \gamma$ are 
uniformly fellow-traveled by subarcs of the uniform
unparametrized quasi-geodesics
$\Upsilon_{\cal T}\circ \gamma_i$.
Thus for any sequence $n_i\to \infty$, 
the curves  $c_i=\Upsilon_{\cal T}(\gamma_i(n_i))$ 
converge in ${\cal C\cal G}(S)\cup \partial{\cal C\cal G}(S)$ to 
the support of $\nu$ (compare the more detailed
argument in the proof of Proposition 3.4 of \cite{H10a}). 
As a consequence, we
have $[c_i]\to [\nu]$ in ${\cal P\cal M\cal L}$.

For $i>0$ let $\zeta_i$ be the WP-geodesic connecting
$\zeta_i(0)=\gamma_i(0)$ to $\zeta_i(\sigma_i)=
\gamma(n_i)$. After
passing to a subsequence we may assume that the directions
$v_i=\zeta_i^\prime(0)$ of the geodesics $\zeta_i$ 
converge as $i\to \infty$ to a direction $v$ with
footpoint $\gamma(0)$. Let $\zeta:[0,T)\to {\cal T}(S)$
be the WP-ray with direction $v$. Then $\zeta_i\to \zeta$ uniformly
on compact subsets of $[0,T)$.

We claim that the length of $\nu$ is bounded along $\zeta$.
Namely,
for each $i$ let $\mu_i$ be the measured geodesic lamination
of $\gamma_i(0)$-length one which we obtain from 
$\Upsilon_{\cal T}(\gamma_i(n_i))$ by multiplication with
a positive constant. Since $\gamma_i(0)\to \gamma(0)$
$(i\to \infty)$ 
there is a number $\epsilon >0$ such that $\gamma_i(0)\in 
{\cal T}(S)_\epsilon$  for all $i$. This means that  
the shortest length of a simple closed curve
for the metric $\gamma_i(0)$ is not smaller than $\epsilon$.
Therefore $\mu_i$ is obtained from the counting measure
for $\Upsilon_{\cal T}(\gamma_i(n_i))$ by multiplication with
a constant which does not exceed $1/\epsilon$. By definition 
of the map $\Upsilon_{\cal T}$,  
the $\gamma_i(n_i)$-length of $\mu_i$ is not
bigger than $\chi_0/\epsilon$. Thus by convexity, the
length of $\mu_i$ along $\zeta_i$ is uniformly bounded, independent
of $i$.

By the above consideration and continuity of the length 
pairing, 
as $i\to \infty$ the measured geodesic laminations $\mu_i$ converge
weakly to the measured geodesic lamination $\nu $. Since
$\zeta_i\to \zeta$ 
locally uniformly, continuity of the length pairing 
implies that the
length of $\nu$ is uniformly bounded along $\zeta$
(compare the more detailed argument in the proof of
Proposition \ref{wpcomp}).

Since $\nu$ fills $S$, the WP-ray $\zeta$ is infinite.
Namely, otherwise there is a simple closed
curve $c$ on $S$ so that the $\zeta(t)$-length of $c$ tends
to zero as $t\to T$. But $i(c,\nu)>0$ and consequently
in this case  
the length of $\nu$ is unbounded along $\zeta$ which
is a contradiction. The lemma is proven.
\end{proof}

The following angle control is the second and last preparatory
step toward the construction of a WP-ray associated
to a Teichm\"uller geodesic $\gamma:\mathbb{R}\to {\cal T}(S)$
whose initial direction is contained in the preimage of 
one of the sets $B(j,\ell)$.

\begin{lemma}\label{anglecontrol}
Let $B\subset {\cal Q}(S)$ be a compact set
consisting of quadratic differentials with uniquely ergodic vertical 
and horizontal measured geodesic laminations. Then 
there are numbers $\alpha=\alpha(B)>0$
and $R_0=R_0(B)>0$ with the following property.
Let $\tilde B\subset \tilde {\cal Q}(S)$ 
be the preimage of $B$ and let 
$\gamma:\mathbb{R}\to {\cal T}(S)$ be 
a Teichm\"uller geodesic with initial velocity $\gamma^\prime(0)
\in \tilde B$. Let $R_1,R_2\geq R_0$ and 
let $\xi_1,\xi_2$ be the Weil-Petersson geodesics
which connect $\gamma(0)=\xi_1(0)=\xi_2(0)$ to 
$\gamma(-R_1),\gamma(R_2)$. Then the angle
$\angle_{\gamma(0)}(\xi_1^\prime(0),\xi_2^\prime(0))$ at 
$\gamma(0)$ between the geodesics $\xi_1,\xi_2$ is at least
$\alpha$.
\end{lemma}
\begin{proof}
We argue by contradiction and we assume that there is a 
set $B$ as in a lemma 
for which the statement of the
lemma does not hold. Let $\tilde B\subset \tilde
{\cal Q}(S)$ be the preimage of $B$. By assumption, 
there is a sequence of Teichm\"uller geodesics
$\gamma_i:\mathbb{R}\to {\cal T}(S)$ with initial 
velocity $\gamma^\prime(0)\in \tilde B$,
and there is a sequence of numbers $R_i,T_i\to \infty$ so that
the angle between the WP-geodesics $\zeta_i,\xi_i$ connecting
$\gamma_i(0)$ to $\gamma_i(-R_i),\gamma_i(T_i)$ 
tends to zero as $i\to \infty$.

By invariance under the action of the
mapping class group and cocompactness of the action of 
${\rm Mod}(S)$ on $\tilde B$, up to passing to a subsequence 
we may assume that the Teichm\"uller 
geodesics $\gamma_i$ converge as
$i\to \infty$ to a Teichm\"uller geodesic $\gamma$.
By Lemma \ref{WPconv}, by passing to another
subsequence we may assume that the WP-geodesics
$\zeta_i,\xi_i$ converge as $i\to \infty$ to infinite
Weil-Petersson rays $\zeta,\xi:[0,\infty)\to {\cal T}(S)$.

Let $q_h,q_v$ be the horizontal and vertical measured
geodesic lamination, respectively, of the area one quadratic differential 
$q$ which is the 
unit cotangent vector of $\gamma$ at $\gamma(0)$. By Lemma \ref{WPconv},
the length of $q_v$ is bounded along $\zeta$, and the
length of $q_h$ is bounded along $\xi$.

On the other hand, by assumption, the angle at $\gamma_i(0)$ 
between $\zeta_i,\xi_i$ converges to zero as $i\to \infty$ 
and therefore we have $\zeta=\xi$. As a consequence,
the lengths of both $q_h,q_v$ 
are bounded along $\zeta$.
But the measured geodesic
laminations $q_h,q_v$ bind $S$ 
(i.e. we have $i(\mu,q_h)+i(\mu,q_v)>0$ for
every measured geodesic lamination $\mu$ on $S$)
and hence the 
function on ${\cal T}(S)$ 
which associates to a point
$x\in {\cal T}(S)$ the value $\ell_x(q_h)+\ell_x(q_v)$ 
is proper (see Theorem 1.2 of \cite{K92}).
Since $\zeta$ is an infinite ray, this is a contradiction. The 
lemma follows.
\end{proof}

Denote by $P:T^*{\cal T}(S)\to {\cal T}(S)$ the canonical
projection of the vector bundle $T^*{\cal T}(S)$
of all holomorphic quadratic differentials 
(which is the cotangent bundle of ${\cal T}(S)$) onto the base.
Then $P$ restricts to the canonical projection of the
sphere bundles for both the Teichm\"uller metric
and the Weil-Peterssen metric.
The next theorem is the 
first part of Theorem \ref{conjug} from the introduction.

\begin{theorem}\label{muasympt}
There is a measurable
conjugacy $\Lambda:{\cal E}\to {\cal Q}_{WP}(S)$ of the restriction
of $\Phi_{\cal T}^t$ to ${\cal E}$ into the geodesic flow of the
Weil-Petersson metric. 
Its restriction to every compact subset of ${\cal E}$
is continuous.
\end{theorem}
\begin{proof} Recall that the preimage $\tilde {\cal E}\subset 
\tilde {\cal Q}(S)$ of ${\cal E}$ consists of 
differentials with uniquely ergodic
horizontal and vertical measured geodesic laminations.

For all $j,\ell$ let 
$\tilde E(j,\ell)\subset \tilde {\cal E}$ be the preimage of 
$E(j,\ell)$ in 
$\tilde {\cal Q}(S)$.
Let moreover $\tilde B(j,\ell)$ be the preimage
of $B(j,\ell)$.

Fix $(j,\ell)$ and recall that $E(j,\ell)\subset B(j,\ell)$.
By Lemma \ref{compact} 
we may apply Lemma \ref{WPconv} and 
Lemma \ref{anglecontrol} to the sets $B(j,\ell)$.  
Thus for every $q\in E(j,\ell)$ 
and every lift $\tilde q$ of $q$ 
there are unique WP-geodesics rays $\zeta_+(\tilde q),
\zeta_-(\tilde q)$ which are limits of segments connecting
points on the forward or backward geodesic
subray of the geodesic $\gamma_{\tilde q}$ with 
initial velocity $\tilde q$. The angle between 
$\zeta_+(\tilde q)$ and $\zeta_-(\tilde q)$ is bounded from
below by 
$\alpha=\alpha(B(j,\ell))$ as in Lemma \ref{anglecontrol}.
Moreover, these rays depend continuously
on $\tilde q\in \tilde E(j,\ell)$.

Let $\delta=\delta(1/j)$ and $\eta=\eta(1/j)$ as in 
Proposition \ref{thickmeet}. 
Let $k=k(\delta,\alpha)>0$ be as in the first part of 
Lemma \ref{gaussb}. 
Let $n>\max\{\ell,k/\eta\}$ be such that
$\Phi_{\cal T}^nq\in B(j,\ell)$; such a number exists by the
definition of $E(j,\ell)$.
By the choice of $\alpha$, if $n$ is sufficiently
large then for sufficiently large
$R\geq k/\eta$ 
the angle at $\gamma_{\tilde q}(0)$ of the geodesic triangle
with vertices $\gamma_{\tilde q}(n),\gamma_{\tilde q}(0),
\gamma_{\tilde q}(-R)$ is at least $\alpha$
and the same holds true for the angle at $\gamma_{\tilde q}(n)$
of the triangle with vertices 
$\gamma_{\tilde q}(n+R),\gamma_{\tilde q}(n), \gamma_{\tilde q}(0)$.
Proposition \ref{thickmeet} shows moreover  that
$\ell_{\delta-{\rm thick}}(\gamma_{\tilde q}[0,n])\geq k$.
Thus by Corollary \ref{asympt1},
there is a unique Weil-Petersson geodesic 
$\xi(\tilde q)$ which is forward asymptotic
to $\zeta_+(\tilde q)$ and backward asymptotic 
to $\xi_-(\tilde q)$.
Moreover, this geodesic depends continuously on $\tilde q\in
\tilde E(j,\ell)$. Its projective ending measures are
the classes $[\tilde q^v],[\tilde q^h]$ 
of the vertical and horizontal measured geodesic laminations of
$\tilde q$.

We complete the proof of the theorem using the arguments
from the proof of Proposition \ref{wpcomp}.
Namely, let $\tilde q\in \tilde E(j,\ell)$ 
and let $\tilde q^v,\tilde q^h$
be the vertical and the horizontal measured geodesic lamination of $\tilde q$,
respectively. The function $x\to \ell_x(\tilde q^v)$
is strictly convex along the Weil-Petersson geodesic $\xi(\tilde q)$
and tends to zero as $t\to \infty$. 
Let $\tilde \Lambda(\tilde q)$ be 
the unit cotangent vector of $\xi(\tilde q)$
at the unique point $\xi(\tilde q)(s)$ where this length equals one.
Since length functions are strictly convex along
Weil-Petersson geodesics, 
the assignment
$t\to \tilde \Lambda(\Phi^t_{\cal T}\tilde q)$ is a homeomorphism
of the orbit of the Teichm\"uller flow through $\tilde q$
onto the orbit of the Weil-Petersson flow through
$\tilde \Lambda(\tilde q)$. 

Since the Teichm\"uller geodesic
$t\to P\Phi_{\cal T}^t\tilde q$ is uniquely determined by the
projective classes of the horizontal and vertical measured
geodesic laminations, respectively, and these projective
measured geodesic laminations are the ending measures of the 
WP-geodesic determined by $\tilde\Lambda(\tilde q)$, 
the map $\tilde q\in \tilde E(j,\ell) \to \tilde\Lambda(\tilde q)$ is injective,
moreover it is continuous and  
equivariant under the action of the mapping class group.
Thus this map projects to a measurable map 
$\Lambda:{\cal E}\to {\cal Q}_{WP}(S)$
which defines a conjugacy of the Teichm\"uller geodesic flow
on 
${\cal E}$ into the Weil-Petersson geodesic flow. Its restriction
to each of the sets $B(j,\ell)$ is continuous. 
\end{proof}

\begin{remark}\label{remuasympt}
The proof of Theorem \ref{muasympt} moreover shows
that for every $q\in {\cal E}$ 
there is a number $c=c(q)>0$ with the following
property. Let $\tilde \Lambda:\tilde {\cal Q}(S)\to \tilde{\cal Q}_{WP}(S)$
be a ${\rm Mod}(S)$-equivariant 
lift of $\Lambda$, defined on the
preimage $\tilde {\cal E}$ of ${\cal E}$, and let 
$\tilde q\in \tilde {\cal E}$ be 
a lift of $q$.
Then there is a sequence $t_i$ $(i\in \mathbb{Z})$ with
$t_i\to \pm \infty$ $(i\to \pm \infty)$ and 
such that
$d_{\cal T}(P\Phi_{\cal T}^{t_i}\tilde q,
P\tilde \Lambda(\Phi_{\cal T}^{t_i}\tilde q))\leq c$ for all $i$.

Namely, by continuity, for any compact set 
$\tilde K\subset \tilde{\cal Q}(S)$ which projects onto $B(j,\ell)$ 
there is some
$c>0$ such that $d_{\cal T}(P\tilde q,P\Lambda(\tilde q))\leq c$ for
every $\tilde q\in \tilde K$. As for every $q\in E(j,\ell)$ the flow
line of $\Phi^t_{\cal T}$ through $q$ intersects $B(j,\ell)$ for arbitrarily
large and small times, by eqivariance 
this number $c$ satisfies the properties
stated in the proposition for all $q\in E(j,\ell)$.
\end{remark}

We conclude this section with the proof of the first part of 
Theorem \ref{fellowtravel} from the introduction 
(which is a version of  a special case of 
Theorem \ref{muasympt}).
As in the introduction,
we always denote by $J,J^\prime$ a closed 
connected subset of $\mathbb{R}$.

\begin{proposition}\label{tmcomp}
\begin{enumerate}
\item
For every $\epsilon >0$ there is a number $R=R(\epsilon)>0$
with the following property. For 
every Teichm\"uller geodesic 
$\gamma:J\to {\cal T}(S)_\epsilon$
there is a Weil-Petersson geodesic
$\xi:J^\prime\to {\cal T}(S)$ with
$d_H(\gamma(J),\xi(J^\prime))\leq R$.
\item 
Let $K\subset {\cal Q}(S)$ be any compact set which is 
invariant under the Teichm\"uller geodesic flow $\Phi^t_{\cal T}$. 
Then there is a 
conjugacy $\Lambda:K\to {\cal Q}_{WP}(S)$ of the restriction
of $\Phi^t_{\cal T}$ to $K$ into the geodesic flow 
$\Phi^t_{WP}$ for the Weil-Petersson metric.
\end{enumerate}
\end{proposition}
\begin{proof} The second part of the proposition is 
immediate from Theorem \ref{muasympt}. 

The argument for the first part is similar to the proof of 
Proposition \ref{wpcomp}. Namely, 
let $J\subset \mathbb{R}$ be a closed connected set
containing $0$ and
let $\gamma:J\to {\cal T}(S)_\epsilon$ be a Teichm\"uller geodesic.
We say that the projective measured geodesic
lamination $[\beta]$ defined by a simple closed curve  $\beta\in {\cal C}(S)$ 
is \emph{realized} at some $t\in J$ if the $\gamma(t)$-length of 
$\beta$ does not exceed $\chi_0$.
If $J$ contains $[0,\infty)$ then we say that the
projectivization $[q_v]$ of the
vertical measured geodesic lamination defined by the 
unit cotangent vector $q$ of $\gamma$ at $\gamma(0)$ is
realized at the right endpoint of $J$, and similarly for
a left infinite endpoint.

Let $\Gamma$ be the set of all triples 
$(\gamma:J\to {\cal T}(S)_\epsilon,\lambda_+,\lambda_-)$ with the
following properties.
\begin{enumerate}
\item $J\subset\mathbb{R}$ is a closed connected set containing $0$.
\item $\gamma:J\to {\cal T}(S)_\epsilon$ is a Teichm\"uller geodesic.
\item $\lambda_+,\lambda_-$ are measured geodesic
laminations of $\gamma(0)$-length one, and
the projective measured geodesic lamination
$[\lambda_+]$ is realized at the right endpoint of
$J$, the projective measured geodesic lamination
$[\lambda_-]$ is realized at the left endpoint of $J$.
\end{enumerate}

We equip $\Gamma$ with the product topology, using the weak$^*$-topology
on ${\cal M\cal L}$ for the second and the third component of the triple
and the compact-open topology for the arc $\gamma:J\to {\cal T}(S)_\epsilon$.
Note that this topology is metrizable. Moreover, $\Gamma$
is invariant under the natural action of the
mapping class group.

We claim that the action of ${\rm Mod}(S)$ on $\Gamma$ is 
cocompact. Since ${\rm Mod}(S)$ acts cocompactly
on ${\cal T}(S)_\epsilon$, for this it is enough to show
that the following holds true. If $\gamma_i:J_i\to {\cal T}(S)_\epsilon$
$(i>0)$ is any sequence of Teichm\"uller geodesics which 
converge locally uniformly to a Teichm\"uller geodesic
$\gamma:J\to {\cal T}(S)_\epsilon$, if the 
projective measured geodesic lamination $[\lambda_i]$
is realized at the right endpoint of $J_i$ and
if $[\lambda_i]\to [\lambda]$ in ${\cal P\cal M\cal L}$ 
$(i\to \infty)$ then $[\lambda]$ is realized at the 
right endpoint of $J$. However, that this holds true was
shown in the proof of Proposition 3.4 of \cite{H10a}
(compare also the argument in the proof of Proposition \ref{wpcomp}).

To each triple $(\gamma:J\to {\cal T}(S)_\epsilon,\lambda_+,\lambda_-)\in \Gamma$ 
associate a Weil-Petersson geodesic $\rho(\gamma,\lambda_+,\lambda_-)$
as follows. Assume first that $J=[-a,b]$ is bounded,
Then there is up to parametrization a unique WP-geodesic 
$\xi$ connecting $\gamma(-a)$ to $\gamma(b)$. 
The restriction to $\xi$ of 
the function which associates to $x\in {\cal T}(S)$ the sum
$\ell_x(\lambda_+)+\ell_x(\lambda_-)$ is strictly convex and non-constant 
(unless $a=b=0$) 
and hence it assumes a unique minimum along $\xi$ \cite{W08}.
Let $\rho(\gamma,\lambda_+,\lambda_-)$ be the parametrization
of $\xi$ so that this minimum is assumed
at $\rho(\gamma,\lambda_+,\lambda_-)(0)$.

If $J$ is one-sided unbounded, say if $J=[-a,\infty)$ for some 
$a\geq 0$, then there is a unique infinite WP-geodesic ray 
$\xi$ issuing from
$\gamma(-a)$ which is asymptotic to $\eta(\gamma_+)$ where
$\eta(\gamma_+)$ is as in Lemma \ref{WPconv}.
The function $x\to \ell_x(\lambda_+)+\ell_x(\lambda_-)$
assumes a unique minimum along $\xi$. 
We define $\rho(\gamma,\lambda_+,\lambda_-)$ to be
the parametrization of $\xi$ for which this minimum
is assumed at $\rho(\gamma,\lambda_+,\lambda_-)(0)$.

If $J$ is two-sided infinite then let $\rho(\gamma,\lambda_+,\lambda_-)$
be the parametrization of the geodesic $\zeta(\gamma)$ 
as in the second part of Lemma \ref{anglecontrol} so that the
minimum of the function $x\to \ell_x(\lambda_+)+\ell_x(\lambda_-)$ 
along $\eta(\gamma)$ is assumed at $\rho(\gamma,\lambda_+,\lambda_-)(0)$.

By Lemma \ref{anglecontrol} and continuity and convexity of length functions, 
the assignment which associates
to $(\gamma:J\to {\cal T}(S)_\epsilon,\lambda_+,\lambda_-)\in \Gamma$ the
point $\rho(\gamma,\lambda_+,\lambda_-)(0)$ is continuous, moreover
it is equivariant under the action of the mapping class group.
Since ${\rm Mod}(S)$ acts cocompactly on $\Gamma$, this means
that for every 
$(\gamma:J\to {\cal T}(S)_\epsilon,\lambda_+,\lambda_-)\in \Gamma$ the
Teichm\"uller 
distance between $\gamma(0)$ and $\rho(\gamma,\lambda_+,\lambda_-)(0)$ 
is uniformly bounded. The first part of the proposition now
follows as in the proof of Proposition \ref{wpcomp}.
\end{proof}

\section{Invariant measures for the 
Teichm\"uller flow}\label{invariant}

Denote by 
$h(\mu)\geq 0$ the entropy of a $\Phi^t_{\cal T}$-invariant
Borel probability measure $\mu$ on ${\cal Q}(S)$ 
(or of a $\Phi^t_{\rm WP}$-invariant Borel probability
measure $\mu$ on ${\cal Q}_{\rm WP}(S)$).
We continue to use the assumptions and notations from
Section 2-7

\begin{theorem}\label{final}
The conjugacy $\Lambda$ induces a 
continuous injective map 
\[\Theta:{\cal M}_{\cal T}({\cal Q}(S))
\to {\cal M}_{\rm WP}({\cal Q}_{\rm WP}(S)).\] Moreover,
\[h(\Theta(\mu))\geq h(\mu)/\sqrt{2\pi(2g-2+m)}\]
for all $\mu\in {\cal M}_{\cal T}({\cal Q}(S)).$ 
\end{theorem}
\begin{proof}
Since the Teichm\"uller space ${\cal T}(S)$ is contractible and
the Teichm\"uller metric on ${\cal T}(S)$ is complete,
the action of $\Phi_{\cal T}^t$ on $\tilde {\cal Q}(S)$ is proper.
This implies that the \emph{space of oriented geodesics} 
${\cal G}(S)$ for 
the Teichm\"uller metric is a locally compact ${\rm Mod}(S)$-space
(which is naturally homeomorphic to a ${\rm Mod}(S)$-invariant 
open subset of ${\cal P\cal M\cal L}\times {\cal P\cal M\cal L}-\Delta$
(where $\Delta$ denotes the diagonal). This set consists of all 
pairs $([\mu],[\nu])$ which bind $S$.

The measure $\mu$ induces a locally finite $\Phi_{\cal T}^t$-invariant
${\rm Mod}(S)$-invariant measure $\tilde \mu$ on
$\tilde {\cal Q}(S)$.
Via disintegration, the measure
$\tilde \mu$ projects to a locally finite
${\rm Mod}(S)$-invariant measure $\hat \mu$ on ${\cal G}(S)$.
Since $\mu$ is ergodic under the Teichm\"uller flow, the
measure $\hat \mu$ is ergodic under the action of 
${\rm Mod}(S)$.

Let $\tilde {\cal E}\subset \tilde {\cal Q}(S)$ be the preimage 
of the set ${\cal E}$ defined in Section \ref{conjugating}. 
Then $\tilde {\cal E}$ is invarariant under $\Phi^t_{\cal T}$ and hence
it projects to a ${\rm Mod}(S)$-invariant Borel subset 
$\hat {\cal E}$ of ${\cal G}(S)$ of full measure for $\hat \mu$.

Let ${\cal G}_{WP}(S)$ be the space of biinfinite geodesics
for the Weil-Peterson metric. 
The conjugacy $\Lambda:{\cal E}\to {\cal Q}_{WP}(S)$ lifts to 
a ${\rm Mod}(S)$-equivariant conjugacy $\hat \Lambda:
\hat {\cal E}\to {\cal G}_{WP}(S)$.
The push-forward
$\hat \Lambda_*(\hat \mu)$ of the measure $\hat \mu$  
is a ${\rm Mod}(S)$-invariant ergodic measure on ${\cal G}_{WP}(S)$.
Its product with the standard Lebesgue measure
on the flow lines of the Weil-Petersson geodesic flow defines
a $\Phi_{WP}^t$-invariant ${\rm Mod}(S)$-invariant 
locally finite Borel measure on $\tilde {\cal Q}_{WP}(S)$
which determines a locally finite Borel measure
$\mu_0$ on ${\cal Q}_{WP}(S)$ (we explain in more detail in the sequel
that $\mu_0$ is in fact finite).
The measure $\mu_0$ is
ergodic under the action of the Weil-Petersson geodesic flow
since the measure 
$\hat \Lambda_*(\hat \mu)$ is ergodic under the 
action of ${\rm Mod}(S)$.

Let $b=\sqrt{2\pi(2g-2+m)}$. We claim  
that the the total mass of the measure $\mu_0$ 
on ${\cal Q}_{WP}(S)$ 
is bounded from above by $b$.
To see that this is the case, let 
$\psi:{\cal E}\times \mathbb{R}\to \mathbb{R}$ be the
measurable function defined by the conjugacy. The function $\psi$ satisfies the
\emph{cocycle identity}
\begin{equation}\label{cocycle}
\psi(x,s+t)=\psi(x,s)+\psi(\Phi_{\cal T}^sx,t)\,(x\in {\cal E},s,t\in \mathbb{R}).
\end{equation}
Moreover, its restriction to ${\cal E}\times [0,\infty)$ is non-negative,
and we have 
$\psi(x,0)<\psi(x,s)< \psi(x,t)$ for $0<s< t$.

Let  
$q\in {\cal E}$ be a typical point for $\mu$ and let
$\tilde q$ be a lift of $q$ to $\tilde {\cal Q}(S)$.
By Remark \ref{remuasympt} there is a number $c>0$ and there is a 
sequence $t_i\to \infty$
such that $d_{\cal T}(P\Phi^{t_i}_{\cal T}\tilde q,
P\tilde \Lambda(\Phi^{t_i}_{\cal T}\tilde q))\leq c$.
Now $\tilde \Lambda(\Phi_{\cal T}^{t_i}\tilde q)=
\Phi_{WP}^{\psi(q,t_i)}\tilde \Lambda(\tilde q)$ 
and hence since geodesics for the Weil-Petersson metric are
globally length minimizing, 
Lemma \ref{wpcomparison} shows that 
$\psi( q,t_i)\leq b(t_i+2c)$ for all $i$.
As a consequence, we have 
\[\lim\inf_{t\to \infty}\frac{1}{t}\psi(q,t)\leq b.\]

Since $q\in {\cal E}$ was an arbitrary typical point for $\mu$, 
the Birkhoff ergodic theorem together with the cocycle 
identity (\ref{cocycle}) 
implies that the (non-negative) 
function $x\to\psi(x,1)$ is integrable with respect to $\mu$, and 
$\int\psi(x,1)d\mu\leq b$.

Recall that length functions are smooth along
Weil-Petersson geodesics. Therefore by construction
of the conjugacy $\Lambda$,   
for every $q\in {\cal E}\subset {\cal Q}(S)$ the
function $t\to \psi(q,t)$ is continuously differentiable,
with derivative $f(q)$ at $t=0$ 
depending measurably on $q$. By invariance of 
the measure $\mu$ under
the Teichm\"uller flow we have
\[\int fd\mu=\int(\int_0^1f(\Phi^t_{\cal T}q)dt)d\mu(q)=\int\psi(q,1)d\mu(q)\leq b.\]

On the other hand, for
$\mu$-almost every $q$ the Radon-Nikodym derivative 
of $\mu_0$ with respect to
$\Lambda_*(\mu)$ exists at $\Lambda(q)$ 
and equals $f(q)$. Therefore we have
\[\mu_0({\cal Q}_{WP}(S))=\int fd\mu\leq b\] as claimed.

As a consequence, the conjugacy $\Lambda$ induces a map
\[\Theta:{\cal M}_{\cal T}({\cal Q}(S)\to 
{\cal M)_{WP}(\cal Q}_{WP}(S).\] Namely, we showed so far
that $\Lambda$ defines a map $\hat\Theta$ from the set of ergodic 
$\Phi^t_{\cal T}$-invariant Borel probability measure on
${\cal Q}(S)$ to a $\Phi^t_{\cal Q}$-invariant
Borel measure on ${\cal Q}_{WP}(S)$ whose total mass is
at most $b$. The map $\hat \Theta$ does not depend on any
choices made and hence it is compatible with convex
combinations. Thus it naturally extends to 
a map from ${\cal M}_{\cal T}({\cal Q}(S))$
into the space of $\Phi^t_{WP}$-invariant Borel measures 
on ${\cal Q}_{WP}(S)$
of total mass at most $b$.
We then define 
\[\Theta(\mu)=\hat\Theta(\mu)/\hat\Theta(\mu)
({\cal Q}_{WP}(S).\]

We claim that the map $\Theta$ is injective. Namely,
by construction, if $\mu$ is an ergodic
$\Phi^t_{\cal T}$-invariant Borel probability measure on 
${\cal Q}(S)$ and if $\tilde q\in {\cal Q}^1(S)$ is the lift
of a typical point for $\mu$ with vertical and 
horizontal measured geodesic laminations $\tilde q^v,\tilde q^h$, 
respectively, then there is a lift of a typical 
point for $\Theta(\mu)$ which determines a biinfinite
Weil-Petersson geodesic with forward and backward 
ending measures $[\tilde q^v],[\tilde q^h]$.
This Weil-Petersson geodesic is recurrent and hence 
by Theorem 1.3 of \cite{BMM10} and the fact that 
$\overline{\cal T}(S)$ is a ${\rm CAT}(0)$-space 
without flat strips contained in its open dense
subset ${\cal T}(S)$, such a recurrent geodesic
is determined up to parametrization by its 
projective ending measures.
As the $\Phi^t_{\cal T}$-orbit of 
$\tilde q$ is determined by the pair
$([\tilde q^v],[\tilde q^h])$ as well, 
this means that the pairs of all ending measures of
all geodesics whose initial cotangents are typical
for $\Theta(\mu)$ determine both $\Theta(\mu)$ and $\mu$.
In other words, the restriction of the 
map $\Theta$ to the extreme points of 
${\cal M}_{\cal T}({\cal Q}(S))$ is injective, and its image
consists of extreme points of ${\cal M}_{WP}({\cal Q}_{WP}(S))$.
By naturality of the map $\hat\Theta$ with respect to 
convex combination, injectivity of $\Theta$ follows.

Next we show
that $h(\Theta(\mu))\geq h(\mu)/b$ for
every measure $\mu\in {\cal M}_{\cal T}({\cal Q}(S))$. For this
we use Rudolph's theorem (see Section 11.4 in \cite{CFS82})
and Abramov's formula. 
Namely, let $\epsilon >0$. Then there is a \emph{special 
representation} of the flow $\Phi^t_{\cal T}$ on $({\cal E},\mu)$ given by
a Lebesgue space $(M,\nu)$, a measure preserving automorphism
$H:(M,\nu)\to (M,\nu)$ and a roof function 
$\rho:M\to [1-\epsilon,1]$.
The flow $\Phi^t_{\cal T}$ on ${\cal E}$ is just the vertical flow on the
space $\{(x,t)\in M\times \mathbb{R}\mid 0\leq t<\rho(x)\}/\sim$ where
$(x,\rho(x))\sim (Hx,0)$ for all $x$. The measure $\mu$ is the
product of $\nu$ with the Lebesgue measure on $\mathbb{R}$. 
In particular, since $\mu$ is a probability measure, the total mass of 
$\nu$ is contained in the interval $[1,\frac{1}{1-\epsilon}]$.

Let $h_1\geq 0$ be the entropy of the $H$-invariant measure $\nu$.
By Abramov's formula, the entropy $h(\mu)$ of $\mu$ equals
$h_1/\int \rho d\nu$. Since $\rho$ assumes values in 
$[1-\epsilon,1]$ we have
$\int\rho d\nu\in [1-\epsilon,1]$. In particular, the entropy
of $\nu$ 
is within $h_1\epsilon$ of the entropy $h(\mu)$ of $\mu$. 

The space $(M,\nu)$ can be thought of 
as a measurable section for the flow $\Phi^t_{\cal T}$.
Via the conjugacy $\Lambda$, it determines
a measurable section for the flow $\Phi^t_{WP}$. Let  
$\hat \rho$ be the corresponding first return time.
Using again Abramov's formula,  
the entropy of $\Theta(\mu)$ equals $h_1/\int \hat\rho d\nu$.

Now for each $t\in [0,1-\epsilon]$ the set $\Phi^t_{\cal T}M$ is a 
measurable section
for $\Phi^t_{\cal T}$ as well to which the above reasoning can be applied.
Since $t\to \psi(x,t)$ is increasing and non-negative we can estimate
\[\int\hat \rho d\nu\leq
\frac{1}{1-\epsilon}\int_M\int_0^{1-\epsilon}\psi(\Phi^t_{\cal T}x,1)dtd\nu
\leq \frac{1}{1-\epsilon}\int\psi(x,1)d\mu\leq \frac{1}{(1-\epsilon)}b.\]
But $\epsilon >0$ was arbitrary and therefore
$h(\Theta(\mu))\geq h(\mu)/b$ as claimed.

Finally 
we are left with showing 
that the map $\Theta $ is continuous with
respect to the weak$^*$-topology.
For this it suffices to show that this holds
true for the map $\hat \Theta$. As $\hat \Theta$
is natural with respect to convex combinations,
standard properties of the weak$^*$-topology, 
for this it is enough to show the following.
Let $(\mu_i)\subset {\cal M}_{\cal T}({\cal Q}(S))$ be a
sequence of ergodic measures converging to an ergodic measure
$\mu$. Then $\hat\Theta(\mu_i)\to \Theta(\mu)$.

To see that this is the case, note first that
the locally finite ${\rm Mod}(S)$-invariant 
measures $\hat \mu_i$ on ${\cal G}(S)$ which are the disintegrations
of the lifts $\tilde \mu_i$ of the measures $\mu_i$ to $\tilde {\cal Q}(S)$ 
converge weakly to the locally finite ${\rm Mod}(S)$-invariant
measure $\hat\mu$. Let $K\subset {\cal G}(S)$ be any
compact set consisting of typical points for $\hat\mu$ 
in the above sense. In particular, $K$ is contained in 
the support of $\hat \mu$. 
Moreover, we may assume that 
$\hat \mu(K^\prime)<\hat \mu(K)$ for every
proper compact subset $K^\prime$ of $K$.

For $j>0$ let
$U_j$ be an open relative compact neighborhood of $K$ with
$U_j\supset U_{j+1}$ and $\cap_jU_j=K$. 
Then $\mu(K)=\lim_{j\to\infty} \mu(U_j)$.
For each $j$ 
we have 
\[\lim\inf_{i\to\infty}\hat\mu_i(U_j)\geq  \hat\mu(K).\] 
Moreover, as $K$ is compact, we also have
$\lim\sup_{i\to \infty}\hat \mu_i(K)\leq \hat \mu(K)$.

Since 
the measures $\hat\mu_i$ are Borel regular, 
for every $j$ we can find a number $i(j)>0$ with
$i(j+1)>i(j)$ and a 
compact subset $K_j\subset U_j$ consisting of typical points for 
$\hat\mu_{i(j)}$ and such that $\hat \mu_{i(j)}(K_j)\geq \hat\mu(K)$.
By passing to a subsequence we may assume that
the compact sets $K_{i(j)}$ converge
as $j\to \infty$ in the Hausdorff topology to a compact
set $C$. Then $C\subset \cap_jU_j=K$, on the other
hand also $\hat \mu(C)\geq \hat \mu(K)$ and hence 
$C=K$.
In particular, for every $\gamma\in K$ there is a sequence
$\gamma_j\in K_{i(j)}$ with $\gamma_j\to \gamma$.

A point $\gamma\in K$ is determined by its pair of 
projective ending measures $([\xi^+],[\xi^-])$-
If $\gamma_j\to \gamma$ then the projective ending measures
$([\xi_j^+],[\xi_j^-])$ of $\gamma_j$ converge to 
the projective ending measures of $\gamma$.
By continuous dependence of Teichm\"uller geodesic on its
pair of ending lamination this implies the following.
There is a compact set 
$B\subset \tilde {\cal Q}(S)$ so that the cotangent
line of each of the 
geodesics $\gamma_j, \gamma$ intersects $B$. Moreover,
the map $\tilde \Lambda$ is defined on $B$.

By Theorem \ref{muasympt}, the restriction of 
$\tilde \Lambda$ to ever compact subset of $\tilde {\cal E}$ is 
continuous. But this just means that if 
$\gamma_j\in K_j$ and if $\gamma_j\to \gamma$ then 
$\hat \Lambda(\gamma_j)\to \hat \Lambda(\gamma)$. 
By the above discussion and the definition of the 
weak$^*$-topology, we have $\hat \Theta(\mu_j)\to 
\hat \Theta(\mu)$ which is what we wanted to show.
\end{proof}

\section{Invariant measures for the Weil Petersson flow} 
\label{invariantmeasures}

The main goal of this section is to show Theorem \ref{conjug2} from 
the introduction.


The proof relies on estimating the decay of length of 
an ending measure along an orbit for $\Phi^t_{WP}$
which is typical for an invariant ergodic Borel probability
measure $\nu$ on ${\cal Q}_{WP}(S)$.

For a quadratic differential $\tilde z\in 
\tilde {\cal Q}_{WP}(S)$ 
denote by $\zeta_{\tilde z}$ the maximal WP-geodesic
with initial velocity $\tilde z$. 
Call
a point $q\in {\cal Q}_{WP}(S)$ \emph{birecurrent} if it is 
contained in its own $\alpha$- and $\omega$-limit set 
for the action of $\Phi^t_{WP}$. 
Let $\nu$ be a $\Phi^t_{WP}$-invariant ergodic
Borel probability measure on ${\cal Q}_{WP}(S)$. 
Then a typical point $q$ for $\nu$ is birecurrent. 
A preimage $\tilde q$ of $q$ in $\tilde {\cal Q}_{WP}(S)$ 
defines a biinfinite WP-geodesic $\zeta_{\tilde q}$.
This geodesic admits filling topological 
ending laminations $\lambda_+(\tilde q),\lambda_-(\tilde q)$.
Every forward (or backward) ending measure,
i.e. an ending measure for the ray
$\zeta_{\tilde q}[0,\infty)$ (or for the ray
$\zeta_{\tilde q}(-\infty,0]$), 
is supported in $\lambda_+(\tilde q)$
(or in $\lambda_-(\tilde q)$), but
recurrence of the orbit does not necessarily imply that 
such an ending measure is unique up to scale
\cite{BMo14}. 

Our first goal is to establish that for a typical orbit for $\nu$, an
ending lamination is uniquely ergodic. 
We begin with a length estimate for measured
laminations along a typical orbit for $\nu$. 

Let as before $P:T^*{\cal T}(S)\to {\cal T}(S)$ be the canonical projection.
Let $z\in {\cal Q}_{WP}(S)$ and let $\tilde z\in \tilde {\cal Q}_{WP}(S)$ be a 
preimage of $z$. 
For a number $u>0$ 
define a measured lamination $\beta$ to be
\emph{$u$-admissible} for $\tilde z$ if the length of $\beta$ is 
decreasing along the segment  
$\zeta_{\tilde z}[0,u]$.
This  only depends on the
projective class of the lamination. 
Moreover, it is invariant under the 
action of ${\rm Mod}(S)$ on 
$\tilde {\cal Q}_{WP}(S)\times {\cal P\cal M\cal L}$.

For the purpose of the next lemma, 
note that if $z\in \tilde {\cal Q}_{WP}(S)$ is the initial 
velocity of a biinfinite
geodesic then for every $R>0$ there is a compact neighborhood $B$ of
$\tilde z$ in $\tilde {\cal Q}_{WP}(S)$ so that for every
$y\in B$ the WP-geodesic with initial velocity $y$ is defined on 
$[-R,R]$.

\begin{lemma}\label{descent}
Let $q$ be a typical point for $\nu$. Then 
there is a number $T>0$, and there is a compact 
neighborhood
$V$ of $q$ with the following properties.
Let $z\in V$ and let $\beta\in {\cal M\cal L}$ be
$T$-admissible for a preimage $\tilde z$ of $z$; then 
\[\log \ell_{\beta}(P\tilde z)-
\log \ell_{\beta}(P\Phi^T_{WP}(\tilde z))\geq 10.\]
\end{lemma}
\begin{proof} Let $\mu$ be an ending
measure for the geodesic 
$\zeta_{\tilde q}$ with initial velocity a lift $\tilde q$ 
of $q$. Then the length of $\mu$ is 
strictly decreasing along $\zeta_{\tilde q}$.
Moreover, if $\xi$ is a measured lamination 
whose length strictly decreases along $\zeta_{\tilde q}$ then
$\xi$ belongs to 
the cone $\Delta\subset {\cal M\cal L}$ 
of measured laminations whose support 
coincides with the support of $\mu$.

We argue by contradiction and we assume that the
lemma does not hold. 
Then there is a sequence $t_i\to \infty$,
a sequence $\tilde q_i\subset \tilde {\cal Q}_{WP}(S)$
with $\tilde q_i\to  \tilde q$, 
and for each $i$ there is a $t_i$-admissible lamination
$\xi_i\in {\cal M\cal L}$
for $\tilde q_i$ so that $\log \ell_{\xi_i}(P\tilde q_i)=1$ and
$\log \ell_{\xi_i}(P\Phi^{t_i}(\tilde q_i))\geq -10$.
By compactness of ${\cal P\cal M\cal L}$ and continuity
of length, after passing to a subsequence 
we may assume that the measured laminations 
$\xi_i$ converge as $i\to \infty$ 
to a measured lamination $\xi$ with 
$\ell_{\xi}(P\tilde q)=1$.

By continuity of length functions
and convexity of length functions along WP-geodesics,
the length of $\xi$ is decreasing
along the geodesic $\zeta_{\tilde q}$
(compare the proof of Lemma \ref{thickgeo} and 
Proposition \ref{wpcomp} where such
an argument is used for the first time in this work).
Thus $\xi$ is contained in the cone $\Delta$, in particular the
length of $\xi$ tends to zero along $\zeta_{\tilde q}$.

On the other hand, by the definition of admissibility, 
we have 
$\log \ell_{\xi_i}(P\Phi^s(\tilde q_i))\geq -10$ for all
$s\in [0,t_i]$ and hence by continuity,
$\log \ell_{\xi}(\zeta_{\tilde q}(s))
\geq -10$ for all $s\geq 0$. This is a contradiction which
yields the lemma. 
\end{proof}

For a typical point $q\in {\cal Q}_{WP}(S)$ for $\nu$ 
and a preimage $\tilde q$ of $q$ in $\tilde {\cal Q}_{WP}(S)$ 
let as above
$\lambda_+(\tilde q)$ be the forward ending lamination of $\tilde q$. 
This is a topological lamination which
a priori may admit more than one transverse measure up to scale.
For $t>0$ and a transverse measure $\mu$ for $\lambda_+(\tilde q)$
let 
\[\tilde \alpha(\tilde q,t,\mu)=\log \ell_{\mu}(P\tilde q)-
\log \ell_{\mu}(P\Phi^t_{WP}\tilde q).\] 
This
does not depend on the normalization of $\mu$. 
The thus defined function is invariant under the action of 
${\rm Mod}(S)$ on $\tilde {\cal Q}_{WP}(S)\times 
{\cal P\cal M\cal L}$. 
The cocycle equality
\begin{equation}\label{cocycleq}
\tilde\alpha(\tilde q,s+t,\mu)=
\tilde\alpha(\tilde q,s,\mu)+\tilde \alpha(\Phi^s_{WP}\tilde q,t,\mu)
\end{equation}
holds true.

Define
\[\tilde \alpha(\tilde q,t)=\min\{\tilde \alpha(\tilde q,t,\mu)\mid \mu\}.\]
The function $\tilde \alpha:\tilde {\cal Q}_{WP}(S)\times 
\mathbb{R}\to \mathbb{R}$ is invariant under the action of 
${\rm Mod}(S)$ and hence it descends to a function 
\[\alpha:{\cal Q}_{WP}(S)\times [0,\infty)\to 
[0,\infty).\]
This function is clearly measurable, and equation (\ref{cocycleq}) 
implies that 
\begin{equation}\label{upper}
\alpha(z,s+t)\geq \alpha(z,s)+\alpha(\Phi^s_{WP}z,t).\end{equation}
Thus by the subadditive ergodic theorem \cite{Kr85}, 
for $\nu$-almost all $z$ the limit
\[\lim_{t\to \infty} \frac{1}{t}\alpha(z,t)\in [0,\infty]\]
exists, and its value $\sigma$ is independent of $z$.

\begin{lemma}\label{positive}
$\sigma >0$.
\end{lemma}
\begin{proof} It suffices to assume that $\sigma<\infty$. 
Let $q\in {\cal Q}_{WP}(S)$ be a typical point for $\nu$ and let
$V$ be a compact neighborhood of $q$ as in Lemma \ref{descent}. 
Denote by $\epsilon >0$ the $\nu$-mass of $V$. 
Let moreover $T>0$ be as in Lemma \ref{descent}.
Choose $n>0$ sufficiently large that the set
\[Z=\{z\mid  \vert \frac{1}{t}\alpha(z,t)-\sigma\vert \leq \epsilon/4 
\text{ for all }t\geq nT\}\] 
satisfies $\nu(Z)\geq 1-\epsilon/4$.
By the Birkhoff ergodic theorem, we may assume that
for $z\in Z$ and all $k>n$ 
we have
\[\frac{1}{k}\sum_{i=0}^{k-1}\chi_V(\Phi_{WP}^{iT}z)\geq 
3\epsilon/4\]
where $\chi_V$ denotes the characteristic function of $V$. 

Now by the choice of $Z,T,\epsilon$, 
the $\Phi^{iT}_{WP}$-orbit $(i\geq 1)$  
of a point $z\in Z$ intersects $V$ in a frequency of at least
$3\epsilon/4$. Since the function $\alpha$ is non-negative, 
equation (\ref{upper}) and the choice of $V$ implies that 
the logarithmic length decrease of any ending measure along 
an orbit segment through a point in $Z$ 
is at least $3\epsilon/4$. On the other hand, by definition 
the minimum of this length
decrease over all ending measures 
is $\epsilon/4$-close to $\sigma$ whence the lemma.
\end{proof}

Lemma \ref{positive} is used to relate a typical orbit for $\nu$ to
an orbit for the Teichm\"uller flow $\Phi^t_{\cal T}$ on 
${\cal Q}(S)$ 
which recurs to a compact set for
arbitrarily large times. The vertical geodesic lamination of 
a quadratic differential defining such an orbit is known to be
uniquely ergodic \cite{M82}, and we deduce that the same
holds true for the ending lamination of a typical
point for $\nu$.

Assume for the moment that for $\nu$-almost all $z$ the ending lamination 
$\lambda_+(z)$ is uniquely ergodic. For simplicity denote again
by $\lambda_+(z)$ a measured lamination supported in 
$\lambda_+(z)$.
We then can
unambiguously define
\[\beta(z)=-\frac{d}{dt}\log \ell_{\lambda_+(z)}(\Phi_{WP}^t(z))
\vert _{t=0}.\] 
The function $z\to \beta(z)$
is measurable and positive.
We have

\begin{proposition}\label{uniqueergodic}
Let $\nu$  be a $\Phi^t_{WP}$-invariant 
ergodic Borel probability measure on 
${\cal Q}_{WP}(S)$.
\begin{enumerate}
\item
The ending lamination of $\nu$-almost every $q\in {\cal Q}_{WP}(S)$ is 
uniquely ergodic.
\item There is a $\Phi^t_{WP}$-invariant subset 
$Z$ of ${\cal Q}_{WP}(S)$ of full measure, and 
there is a measurable conjugacy
$\Xi:Z\to ({\cal Q}(S),\Phi^t_{\cal T})$ into the Teichm\"uller flow.
\item There is a measure $\mu\in {\cal M}_{\cal T}({\cal Q}(S))$ with
$\Theta(\mu)=\nu$ if and only if $\int \beta d\nu<\infty$.
\end{enumerate}
\end{proposition}
\begin{proof} The set of all projective 
transverse measures for a forward ending lamination
of a typical point for $\nu$ 
is a simplex whose dimension is at most $3g-3+m$. Its 
vertices are precisely the ergodic projective transverse measures
for the lamination. (This is well known, but we were not able
to locate the statement in this form in the literature.
The work \cite{K73} shows that there are only finitely many
ergodic projective invariant measures for an interval exchange transformation, 
and this implies finiteness of ergodic projective transverse measures for 
an orientable geodesic lamination which is all we need in the sequel.
The case of a non-orientable measured laminations follows via
passing to the orientation cover).

By ergodicity of the measure $\nu$, the dimension 
of this simplex of projective transverse measures is $\nu$-almost
everywhere constant. Similarly, for $\nu$-almost every 
$z$ the backward ending lamination of $z$ 
supports a simplex of transverse measures whose
dimension does not depend on $z$.
We have to show that the dimension of these simplices 
equals zero almost everywhere.

Let $n_+\geq 1,n_-\geq 1$ be the number of vertices of the
forward and backward simplex, respectively, 
and let $Z$ be a $\Phi^t_{WP}$-invariant Borel set of 
full measure consisting of points where these simplices are defined. 
There is an $n=n_+\cdot n_-$-sheeted cover 
$Z_n$ of $Z$ as follows. Each point in $Z_n$ corresponds to a triple
$(q,[\xi_+],[\xi_-])$ where $q\in Z$ and where $[\xi_{+}]$ 
(or $[\xi_-])$ 
is a vertex of the simplex 
of projective transverse measures for the forward 
(or backward) ending lamination of $q$. 
The flow $\Phi_{WP}^t$ on $Z$ naturally lifts to a flow on $Z_n$.
This flow preserves 
a finite Borel measure $\hat \nu$ 
which projects to $\nu$.
 The measure $\hat \nu$ has at most $n=n_+\cdot  n_-$ 
 ergodic components.

The preimage $\tilde Z$ of $Z$ in $\tilde {\cal Q}_{WP}(S)$
is a ${\rm Mod}(S)$-invariant $\Phi^t_{WP}$-invariant 
Borel subset of $\tilde {\cal Q}_{WP}(S)$. The covering $Z_n$ of $Z$ 
induces a (formal) $n_+\cdot n_-$-sheeted covering $\tilde Z_n$ of $\tilde Z$. 
A point $\tilde z\in \tilde Z$ is a triple 
$(\tilde q,[\xi_+],[\xi_-])$ which 
consists of a quadratic differential
$\tilde q\in \tilde {\cal Q}_{WP}(S)$ and a choice of a pair 
$([\xi_+],[\xi_-])$ of
ergodic forward and backward projective ending
measures for the WP-geodesic determined by $\tilde q$.  
The support of each of these measures fills $S$, 
and as in the proof of Lemma \ref{anglecontrol}, 
the supports of $[\xi_+],[\xi_-]$ are distinct. 
Thus this pair of projective measured geodesic
laminations binds $S$ and hence it 
determines a 
Teichm\"uller geodesic $\Psi([\xi_+],[\xi_-])$.  

Let $\xi_+(\tilde z), \xi_-(\tilde z)\in {\cal M\cal L}$ 
be the measured laminations
whose  projective classes are determined by 
the triple $\tilde z$ and whose
lengths on the surface $P\tilde q$ underlying 
the quadratic differential $\tilde z$ equal one. 
Wirte $i(\xi_+(\tilde z),\xi_-(\tilde z))=a^{-2}$ where as before,
$i$ is the intersection form. 
Let $\tilde \Xi(\tilde z,[\xi_+],[\xi_-])$ 
be the unit cotangent vector for
the geodesic $\Psi([\xi_+],[\xi_-])$ with the property that 
the vertical and horizontal 
measured geodesic laminations
of $\tilde \Xi(\tilde z,[\xi_+],[\xi_-])$
equal 
$a\xi_+(\tilde z),a\xi_-(\tilde z)$.

By smoothness and 
strict convexity of length functions along Weil-Petersson geodesics 
and by Lemma \ref{positive}, the map 
$t\to \tilde \Xi(\Phi^t_{WP}(\tilde q),[\xi_+],[\xi_-])$ is a homeomorphism onto the 
cotangent line of the geodesic $\Psi([\xi_+],[\xi_-])$.
This construction defines a measurable map 
$\tilde \Xi:\tilde Z_n\to \tilde {\cal Q}(S)$ 
which is equivariant under the action of the 
mapping class group. Thus this map descends to a measurable 
map $\Xi:Z_n\to {\cal Q}(S)$ which conjugates the flow on
$Z_n$ into the Teichm\"uller flow.

Let $K\subset Z_n$ be a compact set of 
positive $\hat \nu$-measure   
so that the restriction of 
$\Xi$ to $K$ is continuous. The image of $K$
under the map $\Xi$ is compact. 
The projection of $K$ to $Z$ is a compact subset $K_0$ of
$Z$ of positive measure. 
By ergodicity of $\nu$, 
for $\nu$-almost every $q\in K_0$ the 
$\Phi^t_{WP}$-orbit of $q$ recurs to $K_0$ for 
arbitrarily large times. 
As every point in $K_0$ has  $n$ preimages in $Z_n$, 
this means that there is a Borel subset $A$ of $K$ 
with $\hat \nu(A)>0$  which consists of points whose
orbit under the flow on $Z_n$ recurs to $K$ for 
arbitrarily large times.

For $z\in A$   
the orbit of $\Xi(z)$ under the Teichm\"uller flow $\Phi^t_{\cal T}$
recurs to the compact set $\Xi(K)$ for arbitrarily large times.
Therefore by 
Masur's result  \cite{M82}, for all $z\in A$ the vertical
measured lamination of $\Xi(z)$ 
is uniquely ergodic. But this vertical measured lamination
is the forward ending measure of the geodesic
defined by a lift of $z$. 
Thus for all $z\in A$ the forward ending lamination
for $q$ is uniquely ergodic. But $A$ 
projects to a subset of $Z$ of positive measure and hence 
by ergodicity  
of $\nu$, 
for almost all $z\in {\cal Q}_{WP}(S)$ 
the forward ending measure
is uniquely ergodic. 
The same argument applies to the backward ending measure.
Together the  
first part of the proposition follows, and the second part
is an immediate consequence of the first and the above construction.

The push-forward $\Xi_*\nu$ of $\nu$ under the map $\Xi$ is a
Borel probability measure on ${\cal Q}(S)$ which is quasi-invariant 
under the flow $\Phi^t_{\cal T}$. To construct an invariant 
measure for $\Phi^t_{\cal T}$  
we use again a special representation of the flow $\Phi^t_{WP}$ 
on $Z$ given by a Lebesgue space $(M,\chi)$, a measure preserving
automorphism $H:(M,\chi)\to (M,\chi)$ and a roof function
$\rho:M\to [1-\epsilon,1[$. The flow $\Phi^t_{WP}$ is just the 
vertical flow on the space 
$\{(x,t)\in M\times \mathbb{R}\mid 0\leq t<\rho(x)\}/\sim$ where
$(x,\rho(x))\sim (Hx,0)$ for all $x$. 
Viewing $M$ as a Borel section for the flow $\Phi^t_{WP}$, 
we can map $M$ with $\Xi$ to 
a Borel section for $\Phi^t_{\cal T}$. 
The image $\Xi_*(\chi)$ is a finite Borel measure which 
determines a $\Phi^t_{\cal T}$-invariant locally finite Borel measure
$\mu$ on ${\cal Q}(S)$.

For $x\in M$ let $f(x)$ be the length of the orbit segment
$\cup_{0\leq t<\rho(x)}\Xi(\Phi^t_{WP}(x))$.
It follows as in the proof of Theorem \ref{final} that the measure
$\mu$ is finite 
if and only if $\int fd\chi<\infty$. Moreover, if this is the case then
$\Theta(\mu)=\nu$ is immediate from our construction. 

We are left with showing that $\int fd\chi<\infty$ if and only if
$\int \beta d\nu<\infty$.  
To this end 
observe that the restriction of the map $\Xi$ to a flow line 
of the Weil-Petersson flow is smooth. Namely, lengths functions
are smooth along Weil-Petersson geodesics and strictly convex, and 
by Lemma \ref{positive}, 
the length of the forward ending measure decays with 
exponential rate. Thus if $\omega(z,t)$ is the function defining the 
conjugacy $\Xi$, i.e. if we have
\[\Xi(\Phi^t_{WP}(z))=\Phi_{\cal T}^{\omega(z,t)}\Xi(z),\] 
then $\omega$ can be differentiated in direction of the real parameter. 
Now using again the explicit construction, the function $\beta$ 
in part (3) of the proposition coincides with the function
\[\frac{d}{dt}(t\to \omega(z,t))]\vert _{t=0}>0\]
almost everywhere. 
Thus the measure $\mu$ is finite 
if and only if the function $\beta$ is integrable.
The proposition follows.
\end{proof}

We conclude this section with showing that 
Proposition \ref{uniqueergodic} is not a redundant.

\begin{proposition}\label{notsurjective}
The map $\Theta$ is not surjective.
\end{proposition}
\begin{proof} It suffices to construct a (not necessarily ergodic)
Borel probability measure $\nu$ such that $\int \beta d\nu=\infty$.

For this let $\phi$ be a pseudo-Anosov mapping class which
admits an invariant train track $\tau$ with the following
properties.
First we require that the transition matrix for the 
carrying relation $\phi(\tau)\prec \tau$ is positive. Second we require that
there is a simple closed curve $c$ smoothly 
embedded in $\tau$ as
a subgraph consisting of two branches, one large branch $b$ and
one small branch. 

Splitting $\phi(\tau)$ at $\phi(b)$ results in a
train track which is obtained from $\phi(\tau)$ by
a single Dehn twist $T(\phi(c))$ about $\phi(c)$. 
The train track $T(\phi(c))(\phi(\tau))$ is carried
by $\phi(\tau)$ and hence by $\tau$, and the 
transition matrix $T(\phi(c))(\phi(\tau))\prec \tau$ is positive.

As a consequence, the mapping class $T(\phi(c))\circ \phi$ 
is pseudo-Anosov and admits $\tau$ as an invariant train track. 
Iteration of this construction shows that 
for each $k>0$ the mapping class $T(\phi(c))^k\circ \phi$ is 
pseudo-Anosov. Moreover, for a suitable chocie of  $\phi$,
the closed orbit for the Teichm\"uller flow in ${\cal Q}(S)$ which 
defines the conjugacy class of $T(\phi(c))^k\phi$ contains
a subarc of fixed positive length in a fixed compact 
subset $K$ of ${\cal Q}(S)$. 

Let $\ell(k)$ be the length of the periodic orbit for the Teichm\"uller flow
which defines the conjugacy class of 
$T(\phi(c))^k\circ \phi$. Then $\ell(k)\to \infty$ $(k\to \infty)$.

As Teichm\"uller space with 
the Weil-Petersson metric is quasi-isometric to the
pants graph \cite{B03},  
the periodic orbits for $\Phi^t_{WP}$ corresponding
to the conjugacy classes of
$T(\phi(c))^k\circ \phi$ have uniformly bounded length, say
their length is bounded by $n>0$. 

Choose a sequence $k(i)\to \infty$ so that 
$\sum_i \frac{1}{i^2}\ell(k(i))=\infty$. 
Let $\mu(i)$ be the (unnormalized) 
$\Phi^t_{WP}$-invariant Borel measure supported on
the periodic orbit in ${\cal Q}_{WP}(S)$ which defines
the conjugacy class of $T(\phi(c))^k\circ \phi$.
Define a $\Phi^t_{WP}$-invariant Borel measure 
$\nu$ on ${\cal Q}_{WP}(S)$ by
\[\nu=\sum_i \frac{1}{i^2}\mu(k(i)).\]
This measure is finite and can be normalized to a probability measure. 

The conjugacy $\Xi$ constructed in the proof of 
Proposition \ref{uniqueergodic} maps the measure $\nu$ to 
a weighted sum of invariant measures on the 
periodic orbits for $\Phi^t_{\cal T}$. 
As $\sum_i\frac{1}{i^2}\ell(k(i))=\infty$,  
this measure on ${\cal Q}(S)$ is infinite. Thus $\Theta$ is not
surjective.
\end{proof}

{\bf Remark:} As the space ${\cal M}_{\cal T}({\cal Q}(S))$ is not compact,
Proposition \ref{notsurjective} does not imply that there is an ergodic
Borel probability measure for $\Phi_{WP}^t$ not contained in the image
of $\Theta$. We do now know whether such a measure exists.

\section{The Lebesgue Liouville measure}\label{thelebesgue}

Recall from Section \ref{invariant} the definition of the map
\[\Theta:{\cal M}_{\cal T}({\cal Q}(S))\to {\cal M}_{WP}({\cal Q}_{WP}(S)).\]
The goal of this section is to show

\begin{proposition}\label{image}
The Lebesgue Liouville measure
of the Weil-Petersson metric is contained in the image of 
the map $\Theta$.
\end{proposition}
\begin{proof}
Let $\nu$ be the Lebesgue Liouville measure of the 
Weil-Petersson metric. 
Recall from Section 8 the definition of the function 
$\beta$, defined on a Borel subset of ${\cal Q}_{WP}(S)$ which
is of full measure for every $\Phi^t_{WP}$-invariant 
Borel probability measure. 
By the third part of Proposition \ref{uniqueergodic}, 
we have to show that 
$\int \beta d\nu<\infty$.

For $\tilde q\in \tilde {\cal Q}_{WP}(S)$ let
$\zeta_{\tilde q}$ be the WP-geodesic with initial velocity
$\tilde q$. 
Recall that for every 
$\tilde q\in \tilde {\cal Q}_{WP}(S)$ 
and every measured lamination $\sigma\in {\cal M\cal L}$
the derivative
\[\frac{d}{dt}\log \ell_{\sigma}(\zeta_{\tilde q}(t))\vert_{t=0}\]
is defined and depends continuously on
$\tilde q$ and $\sigma$. Moreover, 
this derivative does not depend on the normalization 
of $\sigma$ and hence this defines a
continuous function on 
$\tilde {\cal Q}_{WP}(S)\times {\cal P\cal M\cal L}$ which 
is invariant under the action of the mapping class group.

The function
$\tilde f:\tilde {\cal Q}_{WP}(S)\to (0,\infty)$ defined by
\[\tilde f(\tilde q)=
\max\{\frac{d}{dt}\log \ell_{\sigma}(\zeta_{\tilde q}(t))\vert_{t=0}\mid
\sigma\in {\cal M\cal L}\}\]
is ${\rm Mod}(S)$-invariant and continuous and hence it
descends to a continuous function 
$f$ on ${\cal Q}_{WP}(S)$.
By definition of the function $\beta$, if $\mu$ is a
$\Phi^t_{WP}$- invariant Borel probability measure on 
${\cal Q}_{WP}(S)$ with 
$\int f d\mu<\infty$ then
$\int \beta d\mu<\infty$.

The WP-metric on moduli space 
${\cal M}(S)={\cal T}(S)/{\rm Mod}(S)$ 
is incomplete. Its completion coincides with 
the Deligne Mumford compactification 
$\overline{{\cal M}(S)}$ of 
${\cal M}(S)$. Let 
\[\rho:{\cal M}(S)\to (0,\infty)\] be the 
function which associates to a point $x$ the
distance from the boundary 
$\partial {\cal M}(S)=\overline{{\cal M}(S)}-{\cal M}(S)$. 
This boundary consists of surfaces with nodes, i.e. surfaces
where each component of 
some non-trivial simple multicurve has been pinched to a point.

The boundary $\partial {\cal M}(S)$ of 
${\cal M}(S)$ is divided into strata according to 
the number and types of nodes. 
For a stratum $\Sigma$ defined by the vanishing
of the geodesic-length sum
$\ell=\ell_1+\dots +\ell_n$, the distance to the 
stratum is given locally as
\[d_{WP}(p,\Sigma)=(2\pi \ell)^{1/2}+O(\ell^2)\]
(Corollary 21 of \cite{W03}). 
In particular, the distance of a point $x\in {\cal M}(S)$
to the boundary $\partial {\cal M}(S)$
equals $(2\pi\ell_\alpha)^{1/2}+O(\ell_\alpha^2)$ 
where $\alpha$ is a systole of $x$.  

The WP-gradients of the geodesic-length functions
also have general expansions \cite{W87}. 
For a curve $\alpha$ of length at most $\epsilon$ 
we have
\[ \Vert {\rm grad} \ell_\alpha \Vert =
\sqrt{\frac{2}{\pi}}\ell_\alpha^{1/2} +O(\ell_\alpha^{3/2})\]
and hence 
\[\Vert d \log \ell_\alpha\Vert \asymp \ell_{\alpha}^{-1/2}.\] 
where as before, the symbol $\asymp$ 
relating two positive functions means that
their quotient is bounded from above and below by a universal
positive constant. 
The sharpest infinitesimal length
decrease of any normalized 
measured lamination
along any Teichm\"uller geodesic 
issuing from $x$ is the length decrease of a systole
along a shortest path
connecting $x$ to $\partial {\cal M}(S)$
\cite{W87,W08}. 

Let $P:{\cal Q}_{WP}(S)\to {\cal M}(S)$ be the 
canonical projection. 
For sufficiently small $\rho(Pq)$ 
and a systole $\alpha$ of $Pq$ 
we obtain
\begin{equation}\label{f}
f(q)\asymp \Vert d\log \ell_{\alpha}(Pq)\Vert 
\asymp \rho^{-1}(Pq).\end{equation}

Thus to show that the function $f$ is integrable with
respect to $\nu$ it suffices
to show that there is a number $\delta >0$ such that 
for sufficiently small $r$ the WP-volume 
of the set $\rho^{-1}(0,r)$ is at most $r^{1+\delta}$ 
(compare also \cite{BMW12}).

That this volume is $O(r^4)$ is immediate from 
Wolpert's asymptotic expansion of the Weil-Petersson
metric near $\partial {\cal M}(S)$ (as explained on p.889 of 
\cite{BMW12}). Alternatively, we can use 
Wolpert's formula for the Weil-Petersson K\"ahler 
form $\omega$ in 
Fenchel Nielsen coordinates for a Bers decomposition
by simple closed curves $\alpha_i$. 
This expression equals 
\[\omega=\sum_id\ell_{\alpha_i}\wedge d\theta_{\alpha_i}.\]

The Fenchel Nielsen twist $\theta_{\alpha_i}$ is
the unit speed twist along the simple closed
curve $\alpha_i$, and its period equals $\ell_{\alpha_i}$.
Replacing the length-twist coordinates about a systole $\alpha$  
by distance-angle 
coordinates $(\rho,\kappa)$ 
$(\rho \asymp \ell_\alpha^{-1},\kappa\in [0,2\pi))$
multiplies the twist speed
by $2\pi$ times the length of $\alpha$ which is
$O(\rho^2)$.
Thus in distance-angle coordinates, the volume form 
is bounded from above by a constant multiple of $\rho^3$. 
This gives
\[{\rm vol}\{\rho\leq r\}\asymp \int_0^r r^3ds =\frac{1}{4} r^4.\]
Together with equation (\ref{f}), we conclude that 
the function $f$ is integrable 
with respect to the Lebesgue Liouville measure $\nu$. 
This completes the proof of the proposition.
\end{proof}

To summarize, there is a $\Phi^t_{\cal T}$-invariant Borel probability
measure $\mu$ on ${\cal Q}(S)$ such that $\Theta(\mu)=\nu$.
We do not have a description of the measure $\mu$, but we conjecture
that in the case of the once punctured torus, this measure is
the Lebesgue measure on ${\cal Q}(S)$ (this makes sense even
though we assumed throughout the paper that the surface $S$ is 
non-exceptional). 

We conclude this work with some remarks on 
absolute continuity and invariant measure classes on 
${\cal P\cal M\cal L}$.

The mapping class group ${\rm Mod}(S)$ acts diagonally on 
${\cal P\cal M\cal L}\times {\cal P\cal M\cal L}-\Delta$.
The space  of oriented geodesic ${\cal G}(S)$ for the 
Teichm\"uller metric is the invariant subset of 
${\cal P\cal M\cal L}\times {\cal P\cal M\cal L}-\Delta$
of all pairs
$(\mu,\nu)$ which bind $S$. Any $\Phi^t_{\cal T}$-invariant
Borel probability measure $\mu$ on ${\cal Q}(S)$
lifts to a measure on $\tilde {\cal Q}(S)$ which 
disintegrates to a ${\rm Mod}(S)$-invariant  locally finite 
Borel measure $\hat \mu$ on ${\cal G}(S)$. 
By Masur's result \cite{M82}, the measure $\hat \mu$ 
gives full mass to the set of 
pairs of uniquely ergodic projective measured
laminations.

The space ${\cal G}_{WP}(S)$ 
of biinfinite oriented geodesics for the Weil-Petersson metric
does not have such an easy description. However, 
the main result of this paper shows that there is such a description for a 
${\rm Mod}(S)$-invariant subset of 
${\cal G}_{WP}(S)$ 
whose unit tangent lines project to a subset of ${\cal Q}_{WP}(S)$ 
of full mass for every $\Phi^t_{WP}$-invariant Borel probability measure.

To be more precise, call
a point $q\in {\cal Q}_{WP}(S)$ \emph{typical}
if $q$ has the following two properties.
\begin{enumerate}
\item[$\bullet$] Let $\tilde q\in \tilde {\cal Q}(S)$ be a preimage of $q$.
Then $\tilde q$ determines a biinfinite geodesic whose ending 
measures are uniquely ergodic.
\item[$\bullet$] $q$ is contained in its own $\alpha$-and $\omega$ limit set.
\end{enumerate}

The following is immediate from Proposition \ref{uniqueergodic}
and the Poincar\'e recurrence theorem.

\begin{lemma}\label{typicaltypical}
The set ${\cal Z}\subset 
{\cal Q}_{WP}(S)$ of typical points has full measure with respect to 
every invariant Borel probability measure. 
\end{lemma}

By Lemma \ref{typicaltypical},
a $\Phi^t_{WP}$-invariant Borel probability measure $\mu$ on 
${\cal Q}_{WP}(S)$ determines a locally finite 
${\rm Mod}(S)$-invariant Borel measure $\hat \mu$ on 
${\cal P\cal M\cal L}\times {\cal P\cal M\cal L}-\Delta$
which gives full mass to the set of pairs of uniquely ergodic 
projective measured laminations.
The measure $\mu$ is ergodic if and only if $\hat \mu$
is ergodic under the action of ${\rm Mod}(S)$.

Call  an invariant Borel probability measure $\mu$ on 
${\cal Q}_{WP}(S)$ 
\emph{absolutely continuous
with respect to the stable foliation} if the following holds true.
Let $\hat \mu$ be the induced invariant measure on 
${\cal P\cal M\cal L}\times {\cal P\cal M\cal L}-\Delta$.
Let $\hat \mu_1$ be a conditional measure on 
a leaf ${\cal P\cal M\cal L}\times \{[\beta]\}$ of the product
foliation; then 
for a Borel set $A\subset {\cal P\cal M\cal L}\times 
\{[\beta]\}$ we have 
$\hat \mu_1(A)=0$ if and only if 
$\hat \mu(A\times {\cal P\cal M\cal L})=0$. 
Similarly there is a notion of absolute continuity 
with respect to the unstable foliation.

The classical Hopf argument (as explained in Section III.3 of
\cite{Mn87}) implies

\begin{proposition}\label{ergodic3}
If $\mu$ is absolutely continuous with respect to the 
stable and unstable foliation then $\mu$ is ergodic.
\end{proposition}

Examples of absolutely continuous measures are the 
invariant Lebesgue measure for the Teichm\"uller flow
\cite{M82,V86} and the Lebesgue Liouville measure
for the Weil-Petersson flow
\cite{BMW12}.

\bigskip

\noindent
MATHEMATISCHES INSTITUT DER UNIVERSIT\"AT BONN\\
Endenicher Allee 60 \\
D-53115 BONN, GERMANY\\
e-mail: ursula@math.uni-bonn.de


\begin{thebibliography}{ABEM06}












\bibitem[Bo86]{B86} F.~Bonahon, {\em Les bouts des vari\'et\'es
hyperboliques de dimension 3}, Ann. Math. 124 (1986), 71--158.







\bibitem[BH99]{BH99} M.~Bridson, A.~Haefliger, {\sl Metric
spaces of non-positive curvature}, Springer Grund\-leh\-ren 319,
Springer, Berlin 1999.

\bibitem[B03]{B03} J.~Brock, {\em The Weil-Petersson metric
and volumes of 3-dimensional hyperbolic convex cores}, 
J. Amer. Math. Soc. 16 (2003), 495--535.

\bibitem[BMM10]{BMM10} J.~Brock, H.~Masur,
Y.~Minsky, {\em Asymptotics of Weil-Petersson geodesics I:
ending laminations, recurrence, and flows},
Geom. Funct. Anal. 19 (2010), 1229--1257.


\bibitem[BMM11]{BMM11} J.~Brock, H.~Masur, 
Y.~Minsky, {\em Asymptotics of Weil-Petersson geodesic II:
bounded geometry and unbounded entropy}, 
Geom. Funct. Anal. 21 (2011), 820--850.

\bibitem[BMo14]{BMo14} J.~Brock, B.~Modami, 
{\em Recurrent Weil-Petersson geodesic rays with 
non-uniquely ergodic ending laminations},
arXiv:1409.1562, to appear in Geom. Top. 







\bibitem[BMW12]{BMW12} K.~Burns, H.~Masur, A.~Wilkinson,
{\em The Weil-Petersson geodesic flow
is ergodic}, Ann. of Math. 175 (2012), 835--908.


\bibitem[B92]{B92} P.~Buser, {\sl Geometry and spectra
of compact Riemann surfaces}, Birkh\"auser, Boston 1992.

\bibitem[C93]{C93} R.~Canary, {\em Ends of hyperbolic $3$-manifolds},  
J. Amer. Math. Soc. 6 (1993),  1--35. 









\bibitem[CFS82]{CFS82} I.P.~Cornfeld, S.V.~Fomin, Y.G.~Sinai,
{\sl Ergodic theory}, Springer Grundlehren 245, Springer, New York 1982.


\bibitem[DW03]{DW03} G.~Daskolopoulos, R.~Wentworth,
{\em Classification of Weil-Petersson isometries},
Amer. J. Math. 125 (2003), 941--975.


\bibitem[EM11]{EM10} A.~Eskin, M.~Mirzakhani,
{\em Counting closed geodesics in moduli space}, 
J. Mod. Dyn. 5 (2011), 71--105. 












\bibitem[H05]{H05} U.~Hamenst\"adt,
{\em Closed geodesics in the thin part of moduli space}, 
arXiv:math.GR/0511349.




\bibitem[H06]{H06} U.~Hamenst\"adt, {\em Train
tracks and the Gromov boundary of the complex
of curves}, in ``Spaces of Kleinian groups''
(Y.~Minsky, M.~Sakuma, C.~Series, eds.),
London Math. Soc. Lec. Notes 329 (2006), 187--207.










\bibitem[H10a]{H10a} U.~Hamenst\"adt, 
{\em Stability of quasi-geodesics in Teichm\"uller space},
Geom. Dedicata 146 (2010), 101--116.

\bibitem[H!0b]{H10b} U. Hamenst\"adt,
{\em Dynamical properties of the Weil-Petersson metric},
"In the tradition of Ahlfors-Bers V", Contemporary Math. 510
(2010), 109--127.


\bibitem[H11]{H11} U.~Hamenst\"adt,
{\em Symbolic dynamics for the Teichm\"uller flow},
arXiv:1112.6107.


\bibitem[H13]{H13} U. Hamenst\"adt,
{\em Bowen's construction for the Teichm\"uller flow},
J. Mod. Dynamics 7 (2013), 489--526.















\bibitem[K73]{K73} A. Katok, {\em Invariant measures of flown on
oriented surfaces}, Dok. Nauk. SSR 211 (1973), Sov Math. Dokl. 14
(1973), 1104-1108.


\bibitem[K92]{K92} S.~Kerckhoff, {\em Lines of
minima in Teichm\"uller space}, Duke Math. J. 65 (1992), 187--213.

\bibitem[Kl99]{Kl99} E.~Klarreich, {\em The boundary
at infinity of the curve complex and the relative
Teichm\"uller space}, unpublished manuscript, 1999.




\bibitem[Kr85]{Kr85} U.~Krengel. {\sl Ergodic theorems}, 
Walter de Gruyter, 1985.



\bibitem[LM10]{LM10} A.~Lenzhen, H.~Masur, 
{\em Criteria for the divergence of pairs of 
Teichm\"uller geodesics}, Geom. Dedicata 144 (2010), 
191--210.


\bibitem[LR11]{LR11} A.~Lenzhen, K.~Rafi, 
{\em Length of a curve is quasi-convex along a 
Teichm\"uller geodesic}, J. Diff. Geom. 88 (2011),
267--295.



\bibitem[Li74]{Li74} M. Linch, {\em A comparison of
metrics on Teichm\"uller space}, Proc. Amer. Math. Soc. 43 (1974),
349--352.

\bibitem[Mn87]{Mn87} R.~Ma\~{n}e, {\sl Ergodic theory and
differentiable dynamics}, Springer Ergebnisse der
Mathematik 8, Springer Berlin Heidelberg 1987.



\bibitem[M82]{M82} H.~Masur, {\em Interval exchange transformations
and measured foliations}, Ann. Math. 115 (1982), 169-201.





\bibitem[MM99]{MM99} H.~Masur, Y.~Minsky, {\em Geometry of the
complex of curves I: Hyperbolicity}, Invent. Math. 138 (1999),
103-149.

\bibitem[MM00]{MM00} H.~Masur, Y.~Minsky, {\em Geometry of the complex
of curves II: Hierarchical structure}, Geom. Funct. Anal. 10 (2000), 902-974.














\bibitem[Mo03]{Mo03} L.~Mosher, {\em Stable Teichm\"uller 
quasigeodesics and ending laminations}, Geom. Top. 7 (2003), 33--90.






\bibitem[PWW10]{PWW10} M.~Pollicott, H.~Weiss, S.~Wolpert,
{\em Topological dynamics of the Weil-Petersson geodesic
flow}, Adv. Math. 223 (2010), 1225--1235.


\bibitem[R05]{R05} K.~Rafi, {\em A characterization of short
curves of a Teichm\"uller geodesic}, Geom. Top. 9 (2005),
179--202.

\bibitem[R07a]{R07a} K.~Rafi, {\em Thick-thin decomposition
of quadratic differentials}, Math. Res. Lett. 14 (2007), 333--341.

\bibitem[R07b]{R07b} K.~Rafi, {\em A combinatorial model for
the Teichm\"uller metric}, Geom. Funct. Anal. 17 (2007),
936--959.


\bibitem[R14]{R14} K.~Rafi, {\em Hyperbolicity in Teichm\"uller space},
Geom. Topol. 18 (2014), 3025--3053.




\bibitem[V86]{V86} W.~Veech, {\em The Teichm\"uller geodesic
flow}, Ann. Math. 124 (1986), 441--530.


\bibitem[W79]{W79} S.~Wolpert, {\em The length
spectra as moduli for compact Riemann surfaces}, 
Ann. Math. 109 (1979), 323--351.


\bibitem[W87]{W87} S. Wolpert, {\em Geodesic length
functions and the Nielsen problem}. J. Diff. Geom. 25
(1987), 275--296.

\bibitem[W03]{W03} S.~Wolpert, {\em Geometry of the
Weil-Petersson completion of Teichm\"uller space}, in
``Surveys in Differential Geometry VIII: Papers in honor
of Calabi, Lawson, Siu and Uhlenbeck'', Intl. Press,
Cambridge, MA 2003, 357--393. 



\bibitem[W08]{W08} S. Wolpert, {\em Behavior of geodesic-length 
functions on Teichm\"uller space}, 
J. Diff. Geom. 79 (2008), 277--343.




\bibitem[W09]{W09} S.~Wolpert, {\em The Weil-Petersson metric
geometry}, Handbook of Teichm\"uller theory, Vol. II, 47--64,
IRMA Lect. math. Theor. Phys., 13, Eur. Math. Soc, Z\"urich, 2009.




\end{thebibliography}
\end{document}